\documentclass{amsart}

\usepackage{amsmath}
\usepackage{amssymb}
\usepackage{xcolor}

\allowdisplaybreaks

\newtheorem{thm}{Theorem}
\newtheorem{cor}[thm]{Corollary}
\newtheorem{lem}[thm]{Lemma}
\newtheorem{propo}[thm]{Proposition}

\theoremstyle{definition}

\theoremstyle{remark}
\newtheorem{rem}{Remark}

\newcommand{\mb}[1]{\mathbf{#1}}
\newcommand{\supp}{\operatorname{supp}}
\title[Multidimensional discrete harmonic analysis]{Harmonic analysis for a multidimensional discrete Laplacian}

\author[\'O. Ciaurri]{\'Oscar Ciaurri}
\address{Departamento de Matem\'aticas y Computaci\'on,
Universidad de La Rioja, Complejo Cient\'{\i}fico-Tecnol\'ogico,
Calle Madre de Dios 53, 26006 Logro\~no, Spain}
\email{oscar.ciaurri@unirioja.es}

\keywords{Discrete harmonic analysis, heat semigroup, Poisson semigroup, fractional integral, Riesz transform, fractional Laplacian, $g_k$-function.}
\subjclass[2010]{Primary: 42C10.}
\thanks{The author was supported by grant PID2021-124332NB-C22 AEI, from Spanish Government.}

\begin{document}

\begin{abstract}
In this paper we analyze some classical operators in harmonic analysis associated to the multidimensional discrete Laplacian
\[
\Delta_N f(\mathbf{n})=\sum_{i=1}^{N}(f(\mathbf{n}+\mathbf{e}_i)-2f(\mathbf{n})+f(\mathbf{n}-\mathbf{e}_i)), \qquad \mathbf{n}\in \mathbb{Z}^N.
\]
We deal with the heat and Poisson semigroups, the fractional integrals, the Riesz transforms, the fractional powers of the Laplacian, and the $g_k$-square functions.
\end{abstract}

\maketitle

\tableofcontents

\section{Introduction}
Our target is to analyze mapping properties of some operators related to the multidimensional discrete Laplacian
\[
\Delta_N f(\mb{n})=\sum_{i=1}^{N}\Delta_{N,i}f(\mb{n}),
\]
where
\[
\Delta_{N,i}f(\mb{n})=f(\mathbf{n}+\mathbf{e}_i)-2f(\mathbf{n})+f(\mathbf{n}-\mathbf{e}_i),
\]
with $\{\mathbf{e}_i\}_{i=1,\dots,N}$ being $N$-dimensional vectors each of whose components are all zero, except the $i$-th that equals one. This operator is the natural extension to higher dimensions of the one-dimensional discrete Laplacian studied in \cite{CGRTV} and \cite{CRSTV}. We have to observe that
\[
\Delta_{N,i}f(\mb{n})=\delta_i^{-}\delta_i^{+}f(\mb{n})
\]
with
\[
\delta_i^{+}f(\mb{n})=f(\mb{n}+\mb{e}_i)-f(\mb{n})\qquad \text{and} \qquad \delta_i^{-}f(\mb{n})=f(\mb{n})-f(\mb{n}-\mb{e}_i).
\]
Note that $\delta_i^{-}\delta_i^{+}=\delta_i^{+}\delta_i^{-}$.
When we consider sequences of one variable we write $\delta^{+} f(n)=f(n+1)-f(n)$ and similarly for $\delta^{-}f$ and $\Delta f$.

We will focus in this paper on the study of the heat and Poisson semigroups, the fractional integrals, the Riesz transforms, the fractional powers of the Laplacian, and the $g_k$-square functions related to $\Delta_N$.

The heat semigroup $W_t=e^{t\Delta_N}$ is the solution of the $N$-dimensional heat equation
\[
\begin{cases}
\Delta_N u(\mb{n},t)=\dfrac{d}{dt}u(\mb{n},t),\\[3pt]
u(\mb{n},t)=f(\mb{n}),
\end{cases}
\]
and it is given by (see \cite{GI} for the one-dimensional case)
\begin{equation}
\label{eq:heat}
W_tf(\mb{n})=\sum_{\mb{k}\in \mathbb{Z}^N}G_{t,N}(\mb{n}-\mb{k})f(\mb{k}),
\end{equation}
where, for $\mb{n}=(n_1,\dots,n_N)$,
\[
G_{t,N}(\mb{n})=\prod_{k=1}^{N}(e^{-2t}I_{n_k}(2t)),
\]
with $I_a$ being the modified Bessel function of first kind and order $a$. This semigroup will be the main tool to analyze our discrete multidimensional operators.

M. Riesz initiated the study of discrete operators in harmonic analysis in \cite{Riesz}, where he proved the boundedness of the Hilbert transform in $L^p(\mathbb{R})$ and its discrete analogue in $\ell^p(\mathbb{Z})$.
Later, in \cite{Calderon-Zygmund} A. P. Calderon and A. Zygmund obtained the boundedness in $\ell^p(\mathbb{Z}^N)$ of the discrete analogues of the singular integrals from his result in $L^p(\mathbb{R}^N)$. Weighted inequalities for the discrete Hilbert and discrete maximal operator in the one-dimensional case were proved by  R. Hunt, B. Muckenhoupt, and R. Wheeden in \cite{HMW}. In the last few decades, several discrete analogues of continuous operators have been analyzed; some significative contributions are due to I. Arkhipov and K. I. Oskolkov \cite{Ar-Os}, J. Bourgain \cite{Bour}, and E. M. Stein and S. Wainger \cite{Ste-Wa-1,Ste-Wa-2}. In [13], we can find a brief history and a nice exposition on discrete analogues of classical operators in harmonic analysis.

In general, along the paper will work on weighted Lebesgue spaces of sequences on $\mathbb{Z}^N$. More precisely, we define
\[
\ell^p(\mathbb{Z}^n,w)=\left\{f=\{f(\mb{n})\}_{\mb{n}\in \mathbb{Z}^N}:\|f\|_{\ell^p(\mathbb{Z}^N,w)}<\infty\right\}, \qquad 1\le p\le \infty,
\]
where the weight $w=\{w(\mb{n})\}_{\mb{n}\in \mathbb{Z}^N}$ is a sequence of positive numbers and
\[
\|f\|_{\ell^p(\mathbb{Z}^N,w)}=\left(\sum_{\mb{n}\in \mathbb{Z}^N}|f(\mb{n})|^pw(\mb{n})\right)^{1/p}, \qquad 1\le p<\infty,
\]
and
\[
\|f\|_{\ell^\infty(\mathbb{Z}^N,w)}=\sup\{|f(\mb{n})|:\mb{n}\in \mathbb{Z}^N\}.
\]
When $w(\mb{n})=1$ for all $\mb{n}\in \mathbb{Z}^N$ we will write $\ell^p(\mathbb{Z}^N)$ only.

The paper is organized as follows. In the next section we present the main tools to deal with the kernel $G_{t,N}$. Section \ref{sec:semi} will be focused on the analysis of the heat and Poisson semigroups. Section \ref{sec:Riesz} contains our results about the fractional integrals and the Riesz transforms. In Section \ref{sec:frac} we present some results on the fractional powers of $\Delta_N$. Finally, in Section \ref{sec:gk} we analyze the mapping properties of the $g_k$-square functions.

\section{Some facts about the modified Bessel functions of first kind}
The modified Bessel function of first kind and order $a$, with $a\in \mathbb{R}$, is given by
\[
I_a(x)=\sum_{k=0}^{\infty}\frac{c_{k,a}}{k!}\left(\frac{x}{2}\right)^{2k+a},
\]
where
\[
c_{k,a}=\frac{1}{\Gamma(a+k+1)},
\]
understanding that for $a\le -1$
\[
c_{k,a}=\frac{(a+k+1)\cdots (a+[-a])}{\Gamma(a+[-a]+1)}, \qquad 0\le k <\lfloor -a \rfloor.
\]
In this section we compile some relevant inequalities for the modified Bessel functions of first kind. They will be the main tools to obtain our results.
\subsection{An AM-GM type inequality for the modified Bessel functions of first kind}
In \cite{Indios}, we find the identity
\begin{equation}
\label{eq:indios}
I_a(x)I_b(x)=\sum_{k=0}^{\infty}\binom{a+b+2k}{k}c_{k,a}c_{k,b}\left(\frac{x}{2}\right)^{2k+a+b},
\end{equation}
which can be proved by applying the Cauchy product of two absolutely convergent series. From this fact, we can obtain the next inequalities for the modified Bessel functions.
\begin{propo}
For $a_1,\dots, a_n>-1$ and $x\ge 0$ the inequalities
\begin{equation}
\label{eq:AM-GM-I}
\frac{\left(\Gamma\left(\frac{a_1+\cdots+a_n}{n}+1\right)\right)^n}{\Gamma(a_1+1)\cdots \Gamma(a_n+1)}\left(I_{\frac{a_1+\cdots+a_n}{n}}(x)\right)^n\le I_{a_1}(x)\cdots I_{a_n}(x)\le \left(I_{\frac{a_1+\cdots+a_n}{n}}(x)\right)^n
\end{equation}
hold.
\end{propo}

\begin{proof}
We will provide an \textit{\`a la Cauchy} proof, obtaining the cases $n=2^m$ first and extending the result for general $n$.
To prove the case $n=2$ we are going to check that
\begin{equation}
\label{eq:double-ineq}
\frac{c_{0,a_1}c_{0,a_2}}{\big(c_{0,\frac{a_1+a_2}{2}}\big)^2}\le \frac{c_{k,a_1}c_{k,a_2}}{\big(c_{k,\frac{a_1+a_2}{2}}\big)^2}\le 1.
\end{equation}
From the Weiertrass' product formula for the Gamma function
\[
\Gamma(z)=\frac{e^{-\gamma z}}{z}\prod_{n=1}^{\infty}\left(\left(1+\frac{z}{n}\right)^{-1}e^{z/n}\right),\qquad z\in \mathbb{C}\setminus\{0,-1,-2,\dots\},
\]
we have
\[
\frac{c_{k,a_1}c_{k,a_2}}{\big(c_{k,\frac{a_1+a_2}{2}}\big)^2}=\prod_{n=1}^{\infty} \frac{(n+a_1+k+1)(n+a_2+k+1)}{\left(n+\frac{a_1+a_2}{2}+k+1\right)^2}.
\]
Then, by using that
\[
\frac{(n+a_1+1)(n+a_2+1)}{\left(n+\frac{a_1+a_2}{2}+1\right)^2}\le \frac{(n+a_1+k+1)(n+a_2+k+1)}{\left(n+\frac{a_1+a_2}{2}+k+1\right)^2}
\le 1
\]
the inequalities in \eqref{eq:double-ineq} follow.

Now, from \eqref{eq:indios}, \eqref{eq:double-ineq} and the identity
\[
\frac{c_{0,a_1}c_{0,a_2}}{\big(c_{0,\frac{a_1+a_2}{2}}\big)^2}=
\frac{\left(\Gamma\left(\frac{a_1+a_2}{2}+1\right)\right)^2}{\Gamma(a_1+1)\, \Gamma(a_2+1)},
\]
we have
\begin{multline*}
\frac{\left(\Gamma\left(\frac{a_1+a_2}{2}+1\right)\right)^2}{\Gamma(a_1+1)\, \Gamma(a_2+1)}\binom{(a_1+a_2)/2+(a_1+a_2)/2+2k}{k}\big(c_{k,\frac{a_1+a_2}{2}}\big)^2\\\le
\binom{a_1+a_2+2k}{k}c_{k,a_1}c_{k,a_2}\le \binom{(a_1+a_2)/2+(a_1+a_2)/2+2k}{k}\big(c_{k,\frac{a_1+a_2}{2}}\big)^2
\end{multline*}
and
\[
\frac{\left(\Gamma\left(\frac{a_1+a_2}{2}+1\right)\right)^2}{\Gamma(a_1+1)\, \Gamma(a_2+1)}\left(I_{\frac{a_1+a_2}{2}}(x)\right)^2\le I_{a_1}(x)I_{a_2}(x)\le \left(I_{\frac{a_1+a_2}{2}}(x)\right)^2.
\]
With an elementary induction process, it can be deduced \eqref{eq:AM-GM-I} for $n=2^m$. When $n$ is not a power of two, we can find $m$ such that $n<2^m$. Taking the values $b_1=a_1$, \dots, $b_n=a_n$, and
\[
b_k=\frac{a_1+\cdots +a_n}{n}, \qquad n+1\le k\le 2^m,
\]
from \eqref{eq:AM-GM-I} with $n=2^m$, we have
\begin{multline*}
\frac{\left(\Gamma\left(\frac{b_1+\cdots+b_{2^m}}{2^m}+1\right)\right)^{2^m}}{\Gamma(b_1+1)\cdots \Gamma(b_{2^m}+1)}\left(I_{\frac{b_1+\cdots+b_{2^m}}{2^m}}(x)\right)^{2^m}\\\le I_{b_1}(x)I_{b_2}(x)\cdots I_{b_n}(x)I_{b_{n+1}}(x)\cdots I_{b_{2^m}}(x)\le \left(I_{\frac{b_1+\cdots+b_{2^m}}{2^m}}(x)\right)^{2^m}.
\end{multline*}
Finally, applying that
\[
\frac{b_1+\cdots+b_{2^m}}{2^m}=\frac{a_1+\cdots+a_n}{n}
\]
and
\[
I_{b_{n+1}}(x)\cdots I_{b_{2^m}}(x)=\left(I_{\frac{a_1+\cdots+a_n}{n}}(x)\right)^{2^m-n}
\]
the general case of \eqref{eq:AM-GM-I} follows.
\end{proof}

The result in the previous proposition is a particular case of \cite[Theorem 2.7]{Chinos}. We have included this new proof because it is elementary and to do this paper self-contained.

Taking the normalized modified Bessel functions of the first kind
\[
\mathcal{I}_a(x)=\frac{2^a \Gamma(a+1)}{x^a}I_a(x),
\]
that verify $\mathcal{I}_a(0)=1$, the left inequality in \eqref{eq:AM-GM-I} reads
\[
\left(\mathcal{I}_{\frac{a_1+\cdots+a_n}{n}}(x)\right)^n\le \mathcal{I}_{a_1}(x)\cdots \mathcal{I}_{a_n}(x).
\]
\subsection{Some inequalities for the difference of modified Bessel functions of first kind}
It is easy to check that the modified Bessel functions decrease with the order; i.e., $I_a(x)<I_b(x)$ when $-1<b<a$ and $x>0$. We will use this fact sometimes along the paper without explicit mention to it.

In our next result we analyze the difference of two consecutive modified Bessel functions of first kind.

\begin{lem}
  For $a>-1$ and $x>0$ the inequalities
  \begin{equation}
  \label{eq:diff-I}
  0<I_a(x)-I_{a+1}(x)< \frac{a+1}{x}I_{a}(x)
  \end{equation}
  hold.
\end{lem}
\begin{proof}
The lower bound is a consequence of the monotonicity respect to the order of the functions  $I_a(x)$. To obtain the upper bound, we use the inequalities \cite[Theorem 2]{Cantabros}
\begin{equation}
\label{eq:cantabros}
f_{a+1}(x)<\frac{I_{a+1}(x)}{I_a(x)}<f_{a+1/2}(x),
\end{equation}
where
\[
f_\alpha(x)=\frac{x}{\alpha+\sqrt{\alpha^2+x^2}}
\]
and the upper bound holds for $a\ge -1/2$ and the lower one for $a\ge-1$. Then, it is clear that
\begin{equation}
\label{eq:cantabros2}
g_{a+1/2}(x)<1-\frac{I_{a+1}(x)}{I_a(x)}<g_{a+1}(x),
\end{equation}
with
\[
g_\alpha(x)=1-f_\alpha(x)=\frac{2\alpha}{\alpha+x+\sqrt{\alpha^2+x^2}}
\]
and now the upper bound holds for $a\ge -1$ and the lower one for $a\ge -1/2$.
Then, by using the upper bound in \eqref{eq:cantabros2}, for $a>-1$, we have
\[
I_{a}(x)-I_{a+1}(x)=\left(1-\frac{I_{a+1}(x)}{I_a(x)}\right)I_{a}(x)< g_{a+1}(x)I_{a}(x)\le  \frac{a+1}{x}I_a(x).\qedhere
\]
\end{proof}

\begin{lem}
  For $a\ge -1/2$ and $x>0$ the inequality
  \begin{equation}
  \label{eq:diff2-I}
  \left|I_a(x)-2I_{a+1}(x)+I_{a+2}(x)\right|<\left(\frac{3}{2x}+\frac{(a+1)(a+2)}{x^2}\right)I_{a}(x)
  \end{equation}
  holds.
\end{lem}

\begin{proof}
We estimate
\[
\left|1-2\frac{I_{a+1}(x)}{I_{a}(x)}+\frac{I_{a+2}(x)}{I_{a}(x)}\right|.
\]
To do this we use the obvious identity
\begin{equation*}
1-2\frac{I_{a+1}(x)}{I_{a}(x)}+\frac{I_{a+2}(x)}{I_{a}(x)}=S_1+S_2
\end{equation*}
with
\[
S_1=\frac{I_{a+2}(x)}{I_{a+1}(x)}-\frac{I_{a+1}(x)}{I_{a}(x)}
\]
and
\[
S_2=\left(1-\frac{I_{a+1}(x)}{I_{a}(x)}\right)\left(1-\frac{I_{a+2}(x)}{I_{a+1}(x)}\right).
\]
By using \eqref{eq:cantabros} and that $f_\alpha(x)<f_\beta(x)$ when $\beta<\alpha$, for $a\ge -1/2$ we have
\[
|S_1|< \max\{f_{a+1}(x)-f_{a+3/2}(x), f_{a+1/2}(x)-f_{a+2}(x)\}=f_{a+1/2}(x)-f_{a+2}(x).
\]
Now, with the bound
\begin{equation}
\label{eq:diff-fa}
f_{a}(x)-f_{a+p}(x)=\frac{p + \sqrt{x^2+a^2} -\sqrt{x^2+(a+p)^2}}{x}\le \frac{p}{x},\qquad p>0
\end{equation}
we have $|S_1|< 3/(2x)$.

For $S_2$ we use \eqref{eq:cantabros2} to obtain that
\[
S_2< g_{a+1}(x)g_{a+2}(x)\le  \frac{(a+1)(a+2)}{x^2}
\]
and the proof is finished.
\end{proof}

Finally, we estimate an alternating sum of four consecutive modified Bessel functions of first kind.

\begin{lem}
  For $a\ge -1/2$ and $x>0$ the inequality
  \begin{multline}
  \label{eq:diff-3}
  \left|I_a(x)-3I_{a+1}(x)+3I_{a+2}(x)-I_{a+3}(x)\right|<\\C\left(\frac{a+2}{x^2}+\frac{(a+1)(a+2)(a+3)}{x^3}\right)I_{a}(x)
  \end{multline}
  holds.
\end{lem}

\begin{proof}
As in the previous proof, we focus on an appropriate bound for
\[
\left|1-3\frac{I_{a+1}(x)}{I_a(x)}+3\frac{I_{a+2}(x)}{I_a(x)}-\frac{I_{a+3}(x)}{I_a(x)}\right|.
\]
In this time, we have
\[
1-3\frac{I_{a+1}(x)}{I_a(x)}+3\frac{I_{a+2}(x)}{I_a(x)}-\frac{I_{a+3}(x)}{I_a(x)}=T_1+T_2+T_3,
\]
with
\[
T_1=-\frac{I_{a+1}(x)}{I_a(x)}+2\frac{I_{a+2}(x)}{I_{a+1}(x)}-\frac{I_{a+3}(x)}{I_{a+2}(x)},
\]
\begin{multline*}
T_2=
\left(1-\frac{I_{a+1}(x)}{I_a(x)}\right)
\left(\frac{I_{a+3}(x)}{I_{a+2}(x)}-\frac{I_{a+2}(x)}{I_{a+1}(x)}\right)
\\+
\left(1-\frac{I_{a+2}(x)}{I_{a+1}(x)}\right)
\left(\frac{I_{a+3}(x)}{I_{a+2}(x)}-\frac{I_{a+1}(x)}{I_{a}(x)}\right)
\end{multline*}
and
\[
T_3=\left(1-\frac{I_{a+1}(x)}{I_a(x)}\right)\left(1-\frac{I_{a+2}(x)}{I_{a+1}(x)}\right)
\left(1-\frac{I_{a+3}(x)}{I_{a+2}(x)}\right).
\]

We are going to estimate $|T_1|$ by using continued fractions. Let
\[
\xi:=\langle q_0, q_1,q_2,\dots\rangle =q_0+\frac{1}{q_1+\frac{1}{q_2+\frac{1}{q_3+\cdots}}}
\]
be a continued fraction. If
\[
\langle q_0,q_1,\dots,q_n \rangle=\frac{P_n}{Q_n},
\]
it is known \cite[Ch. 2, Theorem 9]{Kinchin} that
\begin{equation}
\label{eq:Kin}
\left|\xi-\frac{P_n}{Q_n}\right|\le \frac{1}{Q_nQ_{n+1}}, \qquad n\ge 0.
\end{equation}
In \cite[10.33.1]{NIST}, we find the expansion
\[
\frac{I_{a+1}(x)}{I_a(x)}=\langle q_0(a,x),q_1(a,x),q_2(a,x),\dots \rangle, \qquad x>0, \quad a>-1,
\]
with $q_0(a,x)=0$ and $q_k(a,x)=2(a+k)/x$ for $k\ge 1$. Then, taking $n=2$ in \eqref{eq:Kin},
\begin{multline*}
|T_1|\le \frac{1}{Q_2(a,x)Q_3(a,x)}+\frac{2}{Q_2(a+1,x)Q_3(a+1,x)}+\frac{1}{Q_2(a+2,x)Q_3(a+2,x)}\\+
\left|\frac{P_2(a,x)}{Q_2(a,x)}-2\frac{P_2(a+1,x)}{Q_2(a+1,x)}+\frac{P_2(a+2,x)}{Q_2(a+2,x)}\right|.
\end{multline*}
By using that
\[
Q_2(a,x)=4(a+1)(a+2)+x^2\qquad \text{and}\qquad Q_3(a,x)=4(a+2)(2(a+1)(a+3)+x^2)
\]
and applying the inequalities
\[
(a+1)^2+x^2\le Q_2(a,x)\le Q_2(a+1,x)
\]
and
\[
4(a+2)((a+1)^2+x^2)\le Q_3(a,x)\le Q_3(a+1,x), 
\]
for $a> -1$ and $x>0$, we have
\begin{multline}
\label{eq:T1-1}
\frac{1}{Q_2(a,x)Q_3(a,x)}+\frac{2}{Q_2(a+1,x)Q_3(a+1,x)}+\frac{1}{Q_2(a+2,x)Q_3(a+2,x)}\\\le \frac{1}{(a+2)((a+1)^2+x^2)^2}.
\end{multline}
Moreover, due to $P_2(a,x)=2(a+2)x$, for $a\ge -1/2$
\begin{multline*}
\frac{P_2(a,x)}{Q_2(a,x)}-2\frac{P_2(a+1,x)}{Q_2(a+1,x)}+\frac{P_2(a+2,x)}{Q_2(a+2,x)}\\=
\frac{16x(4(a+2)(a+3)(a+4) - x^2(8+3 a))}{Q_2(a,x)Q_2(a+1,x)Q_2(a+2,x)}
\end{multline*}
and
\begin{multline}
\label{eq:T1-2}
\left|\frac{P_2(a,x)}{Q_2(a,x)}-2\frac{P_2(a+1,x)}{Q_2(a+1,x)}+\frac{P_2(a+2,x)}{Q_2(a+2,x)}\right|\\
\le C\frac{x(a+3)((a+2)(a+4)+x^2)}{((a+1)^2+x^2)^3}\le \frac{C}{(a+1)^2+x^2}.
\end{multline}
In this way, from \eqref{eq:T1-1} and \eqref{eq:T1-2}, we obtain that
\[
|T_1|\le \frac{C}{x^2}, \qquad a\ge-1/2.
\]

Now, from \eqref{eq:cantabros} and \eqref{eq:cantabros2}, it is clear that
\begin{equation*}
|T_2|\le g_{a+1}(x)(f_{a+3/2}(x)-f_{a+3}(x))+g_{a+2}(x)(f_{a+1/2}(x)-f_{a+3}(x))
\end{equation*}
and applying \eqref{eq:diff-fa} we have
\[
|T_2|\le \frac{C}{x}(g_{a+1}(x)+g_{a+2}(x))\le C\frac{a+2}{x^2}.
\]

Finally, from \eqref{eq:cantabros2}, the bound
\[
T_3< g_{a+1}(x)g_{a+2}(x)g_{a+3}(x)\le \frac{(a+1)(a+2)(a+3)}{x^3}
\]
can be deduced and the proof of \eqref{eq:diff-3} is completed.
\end{proof}

\subsection{A uniform bound for the modified Bessel functions of first kind}
Now, we present a uniform bound respect to the order for the function $I_a(x)$. It will be fundamental all along the paper.

\begin{lem}
  For $a>-1/2$, let $\alpha$ be a real number such that $-1/2\le \alpha<a$. Then the inequality
  \begin{equation}
  \label{eq:bound-I}
  x^{-\alpha}e^{-x}I_a(x)\le \frac{C}{(a+1)^{2\alpha+1}}
  \end{equation}
holds with a constant $C$ independent of $a$.
\end{lem}

\begin{proof}
From the identity \cite[10.32.2]{NIST}
\[
I_a(x)=\frac{x^a}{\sqrt{\pi}2^a\Gamma(a+1/2)}\int_{-1}^{1}e^{-xs}(1-s^2)^{a-1/2}\, ds, \qquad a>-1/2,
\]
by using the change of variable $x(1+s)=2w$, we can deduce that
\begin{align*}
x^{-\alpha}e^{-x}I_a(x)&=\frac{2^{a}x^{-\alpha-1/2}}{\sqrt{\pi}\Gamma(a+1/2)}\int_{0}^{x}e^{-2w}w^{a-1/2}
\left(1-\frac{w}{x}\right)^{a-1/2}\, dw\\&
=\frac{2^{a}}{\sqrt{\pi}\Gamma(a+1/2)}\int_{0}^{x}e^{-2w}w^{a-1-\alpha}\left(\frac{w}{x}\right)^{\alpha+1/2}
\left(1-\frac{w}{x}\right)^{a-1/2}\, dw.
\end{align*}
We start with the case $a\ge 1/2$. Using the inequality
\[
z^p(1-z)^q\le \frac{p^p q^q}{(p+q)^{p+q}}, \qquad p,q\ge 0,\quad 0<z<1,
\]
where in the cases $p=0$, $q=0$, and $p=q=0$ the bound has to be understood as one,
we have
\[
\left(\frac{w}{x}\right)^{\alpha+1/2}
\left(1-\frac{w}{x}\right)^{a-1/2}\le \frac{C}{a^{\alpha+1/2}},
\]
with $C$ depending on $\alpha$ but not on $a$, 
and
\begin{align*}
x^{-\alpha}e^{-x}I_a(x)&\le C\frac{2^{a}}{\Gamma(a+1/2)a^{\alpha+1/2}}\int_{0}^{t}e^{-2w}w^{a-1-\alpha}\, dw
\\&\le \frac{C}{\Gamma(a+1/2)a^{\alpha+1/2}}\int_{0}^{\infty}e^{-r}r^{a-1-\alpha}\, dr\\&=C\frac{\Gamma(a-\alpha)}{\Gamma(a+1/2)a^{\alpha+1/2}}\le \frac{C}{a^{2\alpha+1}},
\end{align*}
where in the last step we have applied that $\Gamma(r+p)/\Gamma(r+q)\sim r^{p-q}$, and this is enough for our purpose.

For $-1/2<a<1/2$ we have to prove that $x^{-\alpha}e^{-x}I_a(x)\le C$. We consider the decomposition
\begin{multline*}
x^{-\alpha}e^{-x}I_a(x)=\frac{2^{a}}{\sqrt{\pi}\Gamma(a+1/2)}\\\times\left(\int_{0}^{x/2}+\int_{x/2}^{x}
e^{-2w}w^{a-1-\alpha}\left(\frac{w}{x}\right)^{\alpha+1/2}
\left(1-\frac{w}{x}\right)^{a-1/2}\, dw\right).
\end{multline*}
The first integral can be controlled easily. In fact,
\begin{multline*}
\int_{0}^{x/2}e^{-2w}w^{a-1-\alpha}\left(\frac{w}{x}\right)^{\alpha+1/2}
\left(1-\frac{w}{x}\right)^{a-1/2}\, dw\\\le \frac{1}{2^{\alpha+a}} \int_{0}^{x/2}e^{-2w}w^{a-1-\alpha}\, dw\le  \frac{\Gamma(a-\alpha)}{2^{2a}}.
\end{multline*}
For the second integral, with the change of variable $w=x(1-r)$, we have
\begin{multline*}
\int_{x/2}^{x}e^{-2w}w^{a-1-\alpha}\left(\frac{w}{x}\right)^{\alpha+1/2}
\left(1-\frac{w}{x}\right)^{a-1/2}\, dw\\
\begin{aligned}
&\le  \left(\frac{x}{2}\right)^{a-1-\alpha}\int_{x/2}^{x}e^{-2w}\left(1-\frac{w}{x}\right)^{a-1/2}\, dw
\\&=\frac{x^{a-\alpha}}{2^{a-1-\alpha}}\int_{0}^{1/2}e^{-2x(1-r)}r^{a-1/2}\, dr\\&\le \frac{x^{a-\alpha}}{2^{a-1-\alpha}}e^{-x}\int_{0}^{1/2}r^{a-1/2}\, dr\le \frac{(a-\alpha)^{a-\alpha}e^{-a+\alpha}}{2^{a-1-\alpha}(a+1/2)},
\end{aligned}
\end{multline*}
where in the last step we have applied that $x^b e^{-x}\le b^be^{-b}$ for $x,b>0$. Then, using that $0<a-\alpha<1$
\[
x^{-\alpha}e^{-x}I_a(x)\le \frac{2^{a}}{\sqrt{\pi}\Gamma(a+1/2)}\left(\frac{\Gamma(a-\alpha)}{2^{2a}}+
\frac{(a-\alpha)^{a-\alpha}e^{-a+\alpha}}{2^{a-1-\alpha}(a+1/2)}\right)\le C
\]
and the proof is finished.
\end{proof}

\subsection{Further properties of the modified Bessel functions of first kind}
In this subsection we collect some known properties of the modified Bessel functions which will be used in this paper.

For each $k\in \mathbb{Z}$, it is verified that (see \cite[10.27.1]{NIST})
\begin{equation}
\label{eq:par}
I_{-k}(t)=I_{k}(t).
\end{equation}
Moreover, from its definition it is clear that $I_k(t)> 0$ for $t> 0$,
\begin{equation}
\label{eq:Ik-zero}
I_k(0)=0, \qquad k\not =0,
\end{equation}
and $I_0(0)=1$.

We have the generating function (see \cite[10.35.1]{NIST})
\begin{equation}
\label{eq:gen}
e^{-t(u+u^{-1})/2}=\sum_{k\in \mathbb{Z}}u^kI_{k}(t)
\end{equation}
and it implies, taking $u=1$,
\begin{equation}
\label{eq:sum-L1}
\sum_{k\in \mathbb{Z}}e^{-t}I_k(t)=1.
\end{equation}
Another easy consequence of the given generating function is the Neumann's identity (see \cite[Chapter II, 7.10]{Feller})
\begin{equation}
\label{eq:Neumann}
I_n(t_1+t_2)=\sum_{k\in \mathbb{Z}}I_k(t_1)I_{n-k}(t_2).
\end{equation}
It can be obtained from \eqref{eq:gen} by using the Cauchy product of two series.

The asymptotic expansion for large order and $t$ fix (see \cite[10.41.1]{NIST})
\begin{equation}
\label{eq:asym-I}
I_\nu(t)\sim \frac{1}{\sqrt{2\pi \nu}}\left(\frac{et}{2\nu}\right)^\nu
\end{equation}
will be used in some points.
\section{The heat semigroup and beyond}
\label{sec:semi}
Basic properties of the operator \eqref{eq:heat} can be deduced easily. First, by using \eqref{eq:sum-L1}, we have
\begin{equation}
\label{eq:sum-L1-multi}
\sum_{\mb{k}\in \mathbb{Z}^N}G_{t,N}(\mb{k})=1
\end{equation}
and, due to the positivity of $G_{t,N}$,
\[
|W_t f(\mb{n})|\le \|f\|_{\ell^{\infty}(\mathbb{Z}^N)}\sum_{\mb{k}\in \mathbb{Z}^N}G_{t,N}(\mb{n}-\mb{k})\le \|f\|_{\ell^{\infty}(\mathbb{Z}^N)}.
\]
So $W_t$ is well defined for any sequence in $\ell^\infty(\mathbb{Z}^N)$. Moreover, from \eqref{eq:Ik-zero} and \eqref{eq:Neumann}, we deduce the properties
\[
W_0 f(\mb{n})=f(\mb{n})\qquad \text{and}\qquad W_{t_1}W_{t_2}f(\mb{n})=W_{t_1+t_2}f(\mb{n})
\]
and the family of operators $W_t$ is a semigroup. By using again the positive of the heat kernel, we obtain that $W_t$ is a positive operator (it is positive for positive sequences) and from \eqref{eq:sum-L1} it is possible to prove that it is Markovian; i. e.,
\[
W_t \mb{1}(\mb{n})=\mb{1}(\mb{n}),
\]
where $\mb{1}(\mb{n})$ is the sequence having one in all its entries.

Defining the convolution on $\mathbb{Z}^N$ by
\[
f\ast g(\mb{n})=\sum_{\mb{k}\in \mathbb{Z}^N}f(\mb{n}-\mb{k})g(\mb{k}),
\]
we can observe that
\[
W_t f(\mb{n})=G_{t,N}\ast f(\mb{n}).
\]
Then, by Young's inequality
\begin{equation}
\label{eq:Young}
\|f\ast g\|_{\ell^p(\mathbb{Z}^N)}\le \|f\|_{\ell^q(\mathbb{Z}^N)}\|g\|_{\ell^r(\mathbb{Z}^N)},
\end{equation}
with $1/q+1/r=1+1/p$ and $1\le p,q,r\le \infty$, we obtain that
\begin{equation}
\label{eq:contrac-heat}
\|W_t f\|_{\ell^p(\mathbb{Z}^N)}\le \|G_{t,N}\|_{\ell^1(\mathbb{Z}^N)}\|f\|_{\ell^p(\mathbb{Z}^N)}\le \|f\|_{\ell^p(\mathbb{Z}^N)},\qquad 1\le p\le \infty
\end{equation}
where in the last step we have applied \eqref{eq:sum-L1-multi}.

For a sequence $f\in \ell^1(\mathbb{Z})$, its Fourier transform is given by
\[
\mathcal{F}f(x)=\sum_{\mb{k} \in \mathbb{Z}^N}f(\mb{k})e^{2\pi i\langle \mb{k},x \rangle}, \qquad x\in [-1/2,1/2]^N.
\]
It is well known that $\mathcal{F}$ can be extended as an isometry from the space $\ell^2(\mathbb{Z}^N)$ into $L^2([-1/2,1/2]^N)$ with inverse
\[
\mathcal{F}^{-1}f(\mb{n})=\int_{[-1/2,1/2]^N}f(x)e^{-2\pi i\langle \mb{n},x \rangle}\, dx, \qquad f\in L^2([-1/2,1/2]^N).
\]
Moreover, the Parseval's identity
\[
\|\mathcal{F}f\|_{L^2([-1/2,1/2]^N)}=\|f\|_{\ell^2(\mathbb{Z}^N)}
\]
holds.

Now, from the integral representation \cite[10.32.3]{NIST}
\begin{equation}
\label{eq:integral-I}
I_n(t)=\frac{1}{\pi}\int_{0}^{\pi}e^{t\cos \theta}\cos(n\theta)\, d\theta,\qquad  n\in \mathbb{Z},
\end{equation}
we get
\[
e^{-t}I_n(t)=\int_{-1/2}^{1/2}e^{-2t\sin^2(\pi \theta)}e^{-2\pi i n \theta}\, d\theta.
\]
and
\begin{equation}
\label{eq:Gt-Fourier}
G_{t,N}(\mb{n})=\int_{[-1/2,1/2]^N} e^{-4t \sum_{k=1}^{N}\sin^2(\pi x_k)}e^{-2\pi i \langle \mb{n},x\rangle}\, dx.
\end{equation}
Then
\[
\mathcal{F}(G_{t,N})(x)=e^{-4t \sum_{k=1}^{N}\sin^2(\pi x_k)}.
\]
Hence, using
\[
\mathcal{F}(f\ast g)=\mathcal{F}f\cdot \mathcal{F}g,
\]
we deduce that
\[
\lim_{t\to 0}\|W_tf-f\|_{\ell^2(\mathbb{Z}^N)}=\lim_{t\to 0}\left\|(e^{-4t \sum_{k=1}^{N}\sin^2(\pi x_k)}-1)\mathcal{F}f\right\|_{L^2([-1/2,1/2]^N)}=0
\]
and we have the convergence in $\ell^2(\mathbb{Z}^N)$ of the heat semigroup. Moreover, due to the structure of the spaces $\ell^p(\mathbb{Z}^N)$, we have the pointwise convergence
\[
W_tf(\mb{n})\longrightarrow f(\mb{n}), \qquad \mb{n}\in \mathbb{Z}^N.
\]
Note that, by using the inequality $1-e^{-s}\le s$ with $s\ge 0$, it is verified that
\begin{equation}
\label{eq:cota-l2}
|W_t f(\mb{n})-f(\mb{n})|\le C t\|f\|_{\ell^2(\mathbb{Z}^N)}.
\end{equation}
\subsection{Decay of the heat semigroup}
Now we focus on the study of the decay of the heat semigroup in time. This is a natural question when we are deling with this operator. We start with a $\ell^q-\ell^p$ estimate. The analogous for $\mathbb{R}^N$ can be seen, for example, in \cite[Ch. 3, p. 44]{Caze}. Our discrete result can be found in \cite{Ignat} but we include here it to do this paper self-contained and to provide a new and simple proof.
\begin{thm}
\label{thm:bound-heat}
For $1\le q\le p\le \infty$, the inequality
\begin{equation}
\label{eq:decay}
\|W_tf\|_{\ell^p(\mathbb{Z}^N)}\le C t^{-N/2(1/q-1/p)}\|f\|_{\ell^q(\mathbb{Z}^N)}
\end{equation}
holds.
\end{thm}

\begin{proof}
From the Young's inequality for the convolution, it is enough to prove that
\begin{equation}
\label{eq:upper-G}
\|G_{t,N}\|_{\ell^r(\mathbb{Z}^N)}\le C t^{-N/2(1-1/r)}
\end{equation}
and to do that we use Littlewood's inequality
\begin{equation}
\label{eq:Lya}
\|f\|_{\ell^r(\mathbb{Z}^N)}\le \|f\|_{\ell^1(\mathbb{Z}^N)}^{1/r}\|f\|_{\ell^\infty(\mathbb{Z}^N)}^{1-1/r}, \qquad 1\le r\le\infty,
\end{equation}
the identity \eqref{eq:sum-L1-multi} and the bound $G_{t, N}(\mb{n})\le C t^{-N/2}$.
This last estimate can be deduced from \eqref{eq:AM-GM-I} and \eqref{eq:bound-I} with $\alpha=-1/2$. Indeed, using the notations $\mb{n}_+=(|n_1|,\dots,|n_N|)$ and $|\mb{n}|=|n_1|+\cdots+|n_N|$ and taking into account that $G_{t,N}(\mb{n})=G_{t,N}(\mb{n}_+)$, it is verified that
\begin{equation}
\label{eq:size-G-t}
G_{t, N}(\mb{n})\le \left(e^{-2t}I_{|\mb{n}|/N}(2t)\right)^N=t^{-N/2}\left(t^{1/2}e^{-2t}I_{|\mb{n}|/N}(2t)\right)^N\le Ct^{-N/2}. \qedhere
\end{equation}
\end{proof}

\begin{rem}
(a)  By using \eqref{eq:AM-GM-I} and \eqref{eq:bound-I} with $\alpha=0$, it is easy to check that 
\begin{equation}
\label{eq:Gt-decay-n}
G_{t,N}(\mb{n})\le C (1+|\mb{n}|)^{-N}\le C. 
\end{equation}
Then, for $0<t<T$, with $T$ finite,  we deduce the bound
\begin{equation}
\label{eq:heat-0}
\|G_{t,N}\|_{\ell^r(\mathbb{Z}^N)}\le C
\end{equation}
and the estimate
\begin{equation*}
\|W_tf\|_{\ell^p(\mathbb{Z}^N)}\le C \|f\|_{\ell^q(\mathbb{Z}^N)}.
\end{equation*}
Moreover, combining \eqref{eq:decay} and \eqref{eq:heat-0}, it is possible to obtain the inequality
\[
\|W_tf\|_{\ell^p(\mathbb{Z}^N)}\le C\min\{1, t^{-N/2(1/q-1/p)}\} \|f\|_{\ell^q(\mathbb{Z}^N)}, \qquad 1\le q\le p\le \infty.
\]

(b) From the previous theorem it is possible to prove that for $t>1$
\begin{equation}
\label{eq:sim-G}
\|G_{t,N}\|_{\ell^r(\mathbb{Z}^N)}\sim t^{-N/2(1-1/r)}, \qquad 1\le r\le 2.
\end{equation}
The upper bound is given in \eqref{eq:upper-G}. To prove the lower one, we consider the sequence $f(\mb{k})=G_{t,N}(\mb{k})$ to obtain
\[
W_tf(\mb{n})=\sum_{\mb{k}\in \mathbb{Z}^N}G_{t,N}(\mb{k})G_{t,N}(\mb{n-k})=G_{2t,N}(\mb{n}),
\]
where in the las step we have applied \eqref{eq:Neumann}. Then, from \eqref{eq:decay}, we deduce that
\[
\|G_{2t,N}\|_{\ell^2(\mathbb{Z}^N)}\le t^{-N/2(1/r-1/2)}\|G_{t,N}\|_{\ell^r(\mathbb{Z}^N)}, \qquad 1\le r \le 2.
\]
In this way, proving that
\begin{equation}
\label{eq:lower-G}
\|G_{2t,N}\|_{\ell^2(\mathbb{Z}^N)}\ge C t^{-N/4}
\end{equation}
the proof of the lower bound in \eqref{eq:sim-G} will be completed. From Parseval's identity
\begin{align*}
\|G_{2t,N}\|_{\ell^2(\mathbb{Z}^N)}^2&=\int_{[-1/2,1/2]^N}e^{-8t\sum_{k=1}^{N}\sin^2(\pi x_k)}\, dx=
\left(\int_{-1/2}^{1/2}e^{-8t\sin^2(\pi s)}\, ds\right)^N
\\&\ge \left(\int_{-1/(2\sqrt{t})}^{1/(2\sqrt{t})}e^{-8t\sin^2(\pi s)}\, ds\right)^N\ge Ct^{-N/2},
\end{align*}
where we have used that $e^{-8t\sin^2(\pi s)}\ge K$ for $s\in [-1/(2\sqrt{t}),1/(2\sqrt{t})]$, and the proof of \eqref{eq:lower-G} is finished.
\end{rem}

The next result contains the main novelty of this subsection. For the classical heat equation in $\mathbb{R}^N$, with an initial data having polynomial decay at infinity, in \cite{Duo-Zua} the authors study how the mass of the solution is distributed for large value of $t$. In fact, it is proved the inequality
\[
\left\|H_tf-\left(\int_{\mathbb{R}^N}f(x)\, dx\right)G_t\right\|_{L^p(\mathbb{R}^N)}\le C t^{-1/2-N/2(1/q-1/p)}\||\cdot|f\|_{L^q(\mathbb{R}^N)},
\]
where $H_t$ is the solution of the heat equation, $G_t$ its kernel (the Gaussian function), $1\le q\le p\le\infty$, $1\le q<N/(N-1)$, $f\in L^1(\mathbb{R}^N)$, and $|\cdot|f\in L^q(\mathbb{R}^N)$.
Remember that the mass of the solution $H_t$ is a conservative quantity; i.e.,
\[
\int_{\mathbb{R}^N}H_tf(x)\, dx=\int_{\mathbb{R}^N}f(x)\, dx
\]
Note that \eqref{eq:sum-L1-multi} implies that
\[
\sum_{\mb{k}\in \mathbb{Z}^N}W_tf(\mb{k})=\sum_{\mb{k}\in \mathbb{Z}^N}f(\mb{k}),
\]
so in the discrete setting the mass of the solution of the heat equation is conservative also. 
The discrete analogous of this result is given in \cite{Ignat} only in the case $q=1$, with $1\le p \le \infty$, and the case $1<q\le p$ is left as an open problem. Our next result closes this problem completely.


\begin{thm}
\label{thm:decay}
Let  $N\ge 1$, $1\le q\le p\le\infty$, $1\le q<N/(N-1)$, and $|\cdot|f\in \ell^q(\mathbb{Z}^N)$. Then the inequality
\begin{equation}
\label{eq:DZI}
\left\|W_tf-\left(\sum_{\mb{k}\in \mathbb{Z}^N}f(\mb{k})\right)G_{t,N}\right\|_{\ell^p(\mathbb{Z}^N)}\le C t^{-1/2-N/2(1/q-1/p)}\||\cdot|f\|_{\ell^q(\mathbb{Z}^N)}
\end{equation}
holds.
\end{thm}

To prove our result we need some preliminart lemmas.
\begin{lem}
\label{lem:diff-G}
  Let $N\ge 1$, $\mb{n},\mb{m}\in \mathbb{Z}^N$ such that $|\mb{n}|>2|\mb{n}-\mb{m}|$,  and
  \[
  H_{t,N}(z)=\frac{z}{t}
  \left(G_{t,1}\left(z\right)\right)^N,\qquad z>0.
  \]
  Then the inequality
  \begin{equation}
  \label{eq:diff-G}
  |G_{t,N}(\mb{n})-G_{t,N}(\mb{m})|\le C |\mb{n}-\mb{m}| H_{t,N}\left(\frac{|\mb{n}|+|\mb{m}|}{K}\right)
  \end{equation}
  holds with constants $C$ and $K>1$ independent of $\mb{n}$ and $\mb{m}$.
\end{lem}

\begin{proof}
Without lost of generality we suppose that $n_i\not= m_i$ for $i=1,\dots,N$. In the opposite case the situation can be obviously reduced to this one easily. Indeed, let us suppose that $n_i=m_i$ for $j$ values of $i$ with $0<j<N$, then we consider the decompositions
\[
\mb{n}=\mb{n}_1\cup \mb{n}_2\qquad \text{and} \qquad \mb{m}=\mb{n}_1\cup \mb{m}_2,
\]
where $\mb{n}_1$ are the common values of $\mb{n}$ and $\mb{m}$ and $\mb{n}_2$ and $\mb{m}_2$ are the other components.
Then, if we suppose that the result has been proved for different values, we have
\begin{align*}
|G_{t,N}(\mb{n})-G_{t,N}(\mb{m})|&=G_{t,j}(\mb{n}_1)|G_{t,N-j}(\mb{n}_2)-G_{t,N-j}(\mb{m}_2)|\\&\le C
G_{t,j}(\mb{n}_1)|\mb{n}_2-\mb{m}_2| H_{t,N-j}\left(\frac{|\mb{n}_2|+|\mb{m}_2|}{K}\right)\\& \le C |\mb{n}-\mb{m}| H_{t,N}\left(\frac{|\mb{n}|+|\mb{m}|}{2KN}\right).
\end{align*}
To check the last inequality we use \eqref{eq:AM-GM-I} to have
\begin{multline*}
G_{t,j}(\mb{n}_1)\left(G_{t,1}\left(\frac{|\mb{n}_2|+|\mb{m}_2|}{K}\right)\right)^{N-j}\\\le  \left(G_{t,1}\left(\frac{(N-j)(|\mb{n}_2|+|\mb{m}_2|)/K+|\mb{n}_1|}{N}\right)\right)^N,
\end{multline*}
then the inequality follows applying the monotonicity of the functions $I_a$ and the estimate
\begin{align*}
\frac{(N-j)(|\mb{n}_2|+|\mb{m}_2|)}{K}+|\mb{n}_1|&\ge \frac{(N-j)(|\mb{n}_2|+|\mb{m}_2|)+|\mb{n}_1|}{K}
\\&\ge \frac{|\mb{n}_2|+|\mb{m}_2|+2|\mb{n}_1|}{2K}=\frac{|\mb{n}|+|\mb{m}|}{2K}.
\end{align*}

  From \eqref{eq:par}, it is clear that $G_{t,N}(\mb{n})=G_{t,N}(\mb{n}_+)$. Moreover, we have to observe that $G_{t,N}(\mb{n})$ is invariant under permutations of $\mb{n}$. With these previous reductions, we can consider $\mb{n}=(n_1,\dots,n_N)$ and $\mb{m}=(m_1,\cdots,m_N)$ with
  \begin{equation}
  \label{eq:mono-n}
  0\le n_1\le n_{2}<\cdots\le n_{N}\qquad \text{and}\qquad 0\le m_{N}\le m_{N-1}\le \cdots\le m_{1}.
  \end{equation}

An important tool in the proof will be the identity
\begin{equation}
\label{eq:rara}
G_{t,N}(\mb{n})-G_{t,N}(\mb{m})
=\sum_{i=1}^{N}\left(G_{t,1}(n_i)-G_{t,1}(m_i)\right)
\prod_{j=1}^{i-1}G_{t,1}(m_j)\prod_{k=i+1}^{N}G_{t,1}(n_k),
\end{equation}
where the empty products has to be understood as one. The proof of this relation we can done induction on $N$. To do that we consider $\tilde{\mb{n}}=(\mb{n},n_{N+1})$ and $\tilde{\mb{m}}=(\mb{m},m_{N+1})$ and we apply the identity
\begin{multline*}
G_{t,N+1}(\tilde{\mb{n}})-G_{t,N+1}(\tilde{\mb{m}})=(G_{t,N}(\mb{n})-G_{t,N}(\mb{m}))G_{t,1}(n_{N+1})\\+
G_{t,N}(\mb{m})(G_{t,1}(n_{N+1})-G_{t,1}(m_{N+1})).
\end{multline*}

For $n_i<m_i$, applying \eqref{eq:diff-I} and the monotonicity of the modified Bessel with respect to order, we have
\begin{align*}
|G_{t,1}(n_i)-G_{t,1}(m_i)|&\le \sum_{j=n_i}^{m_i-1}|\delta^+ G_{t,1}(j)|
\\&\le \frac{1}{t} \sum_{j=n_i}^{m_i-1}(j+1)
G_{t,1}(j)\le \frac{(m_i-n_i)(m_i+n_i+1)}{t}G_{t,1}(n_i).
\end{align*}
Then, in general, we obtain the bound
\begin{equation}
\label{eq:lem-aux-1}
|G_{t,1}(n_i)-G_{t,1}(m_i)|\le \frac{|m_i-n_i|(m_i+n_i+1)}{2t}\max\{G_{t,1}(n_i),G_{t,1}(m_i)\}
\end{equation}
and, from \eqref{eq:rara},
\begin{multline}
\label{eq:lem-aux-2}
|G_{t,N}(\mb{n})-G_{t,N}(\mb{m})|\\\le \sum_{i=1}^{N}\frac{|m_i-n_i|(m_i+n_i+1)}{2t}\max\{G_{t,1}(n_i),G_{t,1}(m_i)\}\prod_{j=1}^{i-1} G_{t,1}(m_i)\kern-2.5pt \prod_{k=i+1}^{N}\kern-4.5pt G_{t,1}(n_k).
\end{multline}
Now, applying \eqref{eq:AM-GM-I}, we deduce the estimate
\begin{multline*}
\max\{G_{t,1}(n_i),G_{t,1}(m_i)\}\prod_{j=1}^{i-1}G_{t,1}(m_j)\prod_{k=i+1}^{N}G_{t,1}(n_k)
\\\le \left(G_{t,1}\left(\frac{m_1+\cdots+m_{i-1}+\min\{m_i,n_i\}+n_{i+1}+\cdots+n_N}{N}\right)\right)^N.
\end{multline*}
From the condition $|\mb{n}|>2|\mb{n}-\mb{m}|$, we can deduce that $|\mb{n}|/2<|\mb{m}|< 3|\mb{n}|/2$. Then
if $n_1<m_1$
\[
\min\{m_1,n_1\}+n_2+\cdots+n_N=|\mb{n}|>\frac{2}{5}(|\mb{n}|+|\mb{m}|),
\]
and if $m_1<n_1$, by our supposition \eqref{eq:mono-n}, we can prove easily that
\[
\min\{m_1,n_1\}+n_2+\cdots+n_N\ge \frac{|\mb{n}|+|\mb{m}|}{N+1}.
\]
Similarly, for $i=2,\dots,N-1$, by applying again \eqref{eq:mono-n}, we have
\[
m_1+\cdots+m_{i-1}+\min\{m_i,n_i\}+n_{i+1}+\cdots+n_N\ge \frac{|\mb{n}|+|\mb{m}|}{N+1}.
\]
Moreover, for $n_N<m_N$ it is verified that
\[
m_1+m_2+\cdots+\min\{m_N,n_N\}\ge \frac{|\mb{n}|+|\mb{m}|}{N+1},
\]
and for $m_N<n_N$
\[
m_1+m_2+\cdots+\min\{m_N,n_N\}=|\mb{m}|>\frac{|\mb{n}|+|\mb{m}|}{3}.
\]
As conclusion,
\[
\frac{m_1+\cdots+m_{i-1}+\min\{m_i,n_i\}+n_{i+1}+\cdots+n_N}{N}\ge \frac{|\mb{n}|+|\mb{m}|}{K}
\]
with $K>1$. In this way
\begin{equation*}
\max\{G_{t,1}(n_i),G_{t,1}(m_i)\}\prod_{j=1}^{i-1}G_{t,1}(m_j)\prod_{k=i+1}^{N}G_{t,1}(n_k)
\le \left(G_{t,1}\left(\frac{|\mb{n}|+|\mb{m}|}{K}\right)\right)^N
\end{equation*}
and by \eqref{eq:lem-aux-2} the proof is completed.
\end{proof}

\begin{rem}
\label{rem:lemma-heat}
(a) From our proof of \eqref{eq:diff-G} it is clear that it holds for $|\mb{n}|\sim |\mb{m}|$ also.

(b) From \eqref{eq:lem-aux-1} and \eqref{eq:lem-aux-2}, using \eqref{eq:bound-I} with $\alpha=-1/2$, it is easy to deduce that
 \begin{equation}
\label{eq:H-infty}
 |G_{t,N}(\mb{n})-G_{t,N}(\mb{m})|\le C |\mb{m}-\mb{n}|(|\mb{n}|+|\mb{m}|) t^{-N/2-1}, \qquad \mb{n},\mb{m}\in \mathbb{Z}^N.
 \end{equation}

(c) Moreover, we can check that for $|\mb{n}|\ge B|\mb{m}|$, for some positive constant $B$, the inequality
  \begin{equation}
  \label{eq:diff-porsi}
  |G_{t,N}(\mb{n})-G_{t,N}(\mb{m})|\le C \frac{|\mb{n}-\mb{m}|(|\mb{n}|+|\mb{m}|)}{t} \left(G_{t,1}\left(\frac{|\mb{m}|}{K}\right)\right)^N,
  \end{equation}
  holds for some constant $K$.
\end{rem}

The following estimate will be useful in the proof of Theorem \ref{thm:decay}.

\begin{lem}
The inequality
\begin{equation}
\label{eq:bound-diff-P}
\frac{|G_{t,N}(\mb{n})-G_{t,N}(\mb{k})|}{|\mb{n}-\mb{k}|}\le C t^{-1/2-N/2}
\end{equation}
holds
\end{lem}

\begin{proof}
From \eqref{eq:Gt-Fourier} we have
\begin{align*}
\frac{|G_{t,N}(\mb{n})-G_{t,N}(\mb{k})|}{|\mb{n}-\mb{k}|}&\le \int_{[-1/2,1/2]^N}e^{-4t\sum_{i=1}^{N}\sin^2(\pi x_i)}\frac{|e^{-2\pi i \langle \mb{n}-\mb{k},x\rangle}-1|}{|\mb{n}-\mb{k}|}\, dx\\&\le
C \int_{[-1/2,1/2]^N}e^{-4t\sum_{i=1}^{N}\sin^2(\pi x_i)}|x|\, dx, 
\end{align*}
where we have applied that $|(e^{-iz}-1)/z|\le C$ for $z\in \mathbb{R}$. Now, with the estimates
\[
\int_{-1/2}^{1/2}e^{-4t\sin^2(\pi x)}\, dx\le C t^{-1/2} \qquad\text{and}\qquad \int_{-1/2}^{1/2}e^{-4t\sin^2(\pi x)}|x|\, dx\le C t^{-1}
\]
the result is obtained.
\end{proof}

The function $H_{t,N}$ in Lemma \eqref{lem:diff-G} will appear in convolution with a specific sequence in the proof of Theorem \eqref{thm:decay}, so we need the size of its norm.

\begin{lem}
  For $N\ge 1$ and $1\le r \le \infty$, it is verified that
  \begin{equation}
  \label{eq:norm-H}
  \left\|H_{t,N}\left(\frac{|\cdot|}{K}\right)\right\|_{\ell^r(\mathbb{Z}^N)}\le C t^{-1/2-N/2(1-1/r)},
  \end{equation}
  with the constant $C$ being independent of $t$.
\end{lem}
\begin{proof}
From \eqref{eq:bound-I} with  $\alpha=-1/2+1/(2N)$, we have
\begin{equation}
\label{eq:lem-1}
\left|H_{t,N}\left(\frac{|\mb{n}|}{K}\right)\right| \le C t^{-N/2-1/2}|\mb{n}|\left(t^{1/2-1/(2N)}G_{t,1}\left(\frac{|\mb{n}|}{K}\right)\right)^{N}\le Ct^{-N/2-1/2}
\end{equation}
and
\[
\left\|H_{t,N}\left(\frac{|\cdot|}{K}\right)\right\|_{\ell^\infty(\mathbb{Z}^N)}\le C t^{-N/2-1/2},
\]

By using again \eqref{eq:bound-I} with $\alpha=1/N$ we deduce that
\begin{equation}
\label{eq:lem-2}
\left|H_{t,N}\left(\frac{|\mb{n}|}{K}\right)\right|\le C |\mb{n}|\left(t^{-1/N}G_{t,1}\left(\frac{|\mb{n}|}{K}\right)\right)^N
\le \frac{C}{(|\mb{n}|+1)^{N+1}}.
\end{equation}
Now from the previous estimate, for $0<t\le T$  with $T<100$, for example, it is obtained that
\[
\left\|H_{t,N}\left(\frac{|\cdot|}{K}\right)\right\|_{\ell^1(\mathbb{Z}^N)}\le C \sum_{\mb{n} \in \mathbb{Z}^N}\frac{1}{(|\mb{n}|+1)^{N+1}}\le C \sum_{m=1}^{\infty}\frac{1}{m^2}\le C.
\]
For $t>T$ we consider the decomposition
\[
\left\|H_{t,N}\left(\frac{|\cdot|}{K}\right)\right\|_{\ell^1(\mathbb{Z}^N)}=S_1+S_2,
\]
where
\[
S_1=\sum_{|\mb{n}|\le n_t}\left|H_{t,N}\left(\frac{|\mb{n}|}{K}\right)\right|
\qquad
\text{and}
\qquad
S_2=\sum_{|\mb{n}|> n_t}\left|H_{t,N}\left(\frac{|\mb{n}|}{K}\right)\right|,
\]
where $n_t=\lfloor \sqrt{t}\rfloor$ (this notation will used in some points along the paper without explicit mention to it). For $S_1$, applying \eqref{eq:lem-1}, we have
\[
S_1\le C t^{-N/2-1/2}\sum_{|\mb{n}|\le n_t}1\le C t^{-N/2-1/2}\sum_{m=0}^{n_t}(m+1)^{N-1}\le C t^{-1/2}.
\]
Now, by using \eqref{eq:lem-2}, for $S_2$ we obtain the estimate
\[
S_2\le C \sum_{|\mb{n}|> n_t}\frac{1}{|\mb{n}|^{N+1}}\le C\sum_{m=n_t}^{\infty}\frac{1}{m^2}\le C t^{-1/2}.
\]
In this way
\[
\left\|H_{t,N}\left(\frac{|\cdot|}{K}\right)\right\|_{\ell^1(\mathbb{Z}^N)}\le C t^{-1/2}.
\]
and the result follows from \eqref{eq:Lya}.
\end{proof}

In some points we will need the weighted estimate of the heat kernel in the following lemma.

\begin{lem}
\label{lem:norm-Gt-weight}
For $N\ge 1$, $1\le r \le \infty$, and $N(1-r)< \gamma < N$, it is verified that
\begin{equation}
\label{eq:Gt-pesos}
\left\|\left(G_{t,1}\left(\frac{|\cdot|}{N}\right)\right)^N\right\|_{\ell^r(\mathbb{Z}^N,(|\mb{n}|+1)^{-\gamma})}\le C t^{-1/2((\gamma-N)/r+N)},
\end{equation}
with the constant $C$ being independent of $t$.
\end{lem}
\begin{proof}
From the bound \eqref{eq:size-G-t} the result for $r=\infty$ is clear, so we analyze the estimate for $1\le r <\infty$. From \eqref{eq:bound-I} with $\alpha=-1/2$ and $\alpha=0$, we get
\begin{multline*}
\left\|\left(G_{t,1}\left(\frac{|\cdot|}{N}\right)\right)^N\right\|^r_{\ell^r(\mathbb{Z}^N,(|\mb{n}|+1)^{-\gamma})}\\
\begin{aligned}
&\le C\left(t^{-Nr/2}\sum_{|\mb{n}|\le n_t}(|\mb{n}|+1)^{-\gamma}+\sum_{|\mb{n}|\ge n_t}(|\mb{n}|+1)^{-\gamma-Nr}\right)
\\& \le C\left(t^{-Nr/2}\sum_{m=0}^{n_t}(m+1)^{-\gamma+N-1}+\sum_{m= n_t}^{\infty}(m+1)^{-\gamma-Nr+N-1}\right)\\&\le C t^{-1/2(\gamma-N+Nr)}
\end{aligned}
\end{multline*}
finishing the proof.
\end{proof}

The proof of our main result will be divided in the cases $N=1$ and $N>1$ and to deal with the first one we will apply the next estimate.
\begin{lem}
For $1\le r \le \infty$, the inequality
\begin{equation}
\label{eq:norm-diff}
\|\delta^+ G_{t,1}\|_{\ell^r(\mathbb{Z})}\le C t^{-1/2-(1-1/r)/2}
\end{equation}
holds.
\end{lem}
\begin{proof}
As usual, we prove the inequality for $r=1$ and $r=\infty$ and the result follows from \eqref{eq:Lya}. From \eqref{eq:diff-I} and \eqref{eq:bound-I} with $\alpha=0$, we have 
\[
|\delta^+ G_{t,1}(n)|\le Ct^{-1}|n|G_{t,1}(n)\le Ct^{-1}
\]
and this is enough for $r=\infty$. For $r=1$ we use \eqref{eq:diff-I} and decompose the sum 
\[
\|\delta^+ G_{t,1}\|_{\ell^1(\mathbb{Z})}\le C t^{-1}(B_1+B_2), 
\]
where
\[
B_1=\sum_{|n|\le n_t}|n|G_{t,1}(n)\qquad \text{and} \qquad B_2=\sum_{|n|> n_t}|n|G_{t,1}(n). 
\]
To treat $B_1$, by using \eqref{eq:bound-I} with $\alpha=-1/2$, we deduce that
\[
t^{-1}B_1\le t^{-3/2}\sum_{|n|\le n_t}|n|\le C t^{-1/2}
\]
and this is enough for our purpose. The factor $B_2$ can be analyzed with the estimate \eqref{eq:bound-I} taking $\alpha=1$. Indeed,
\[
t^{-1}B_2\le C \sum_{|n|>n_t}\frac{1}{n^2}\le C t^{-1/2} 
\]
and the proof is completed.
\end{proof}


\begin{proof}[Proof of the Theorem \ref{thm:decay}]

From the restriction $1\le q <N/(N-1)$, applying H\"older inequality, we have
\[
\|f\|_{\ell^1(\mathbb{Z}^N)}\le C\||\cdot|f\|_{\ell^q(\mathbb{Z}^N)},
\]
so, in particular $f\in \ell^1(\mathbb{Z}^N)$ and the left hand side of \eqref{eq:DZI} is finite.

We start showing the inequality \eqref{eq:DZI} for $N=1$. In \cite[Theorem 6.1]{Ab-et-al} we find (with our notation) the identity
\[
W_tf(n)-\left(\sum_{n\in \mathbb{Z}}f(n)\right)G_{t,1}(n)=\left(\delta^+ G_{t,1}\right)\ast F(n),
\]
where
\[
F(n)=\sum_{j=-\infty}^{n-1}f(j),\quad \text{for $n\le 0$},\qquad \text{and}\qquad F(n)=-\sum_{j=n}^{\infty}f(j), \quad \text{for $n\ge 1$}.
\]
Then, by using \eqref{eq:Young} and \eqref{eq:norm-diff}, we have
\[
\left\|W_tf-\left(\sum_{k\in \mathbb{Z}}f(k)\right)G_{t,1}\right\|_{\ell^p(\mathbb{Z})}\le C t^{-1/2-(1/q-1/p)/2}\|F\|_{\ell^q(\mathbb{Z})}.
\]
Finally, applying the well known inequality
\[
\sum_{n=1}^{\infty}\left|\sum_{j=n}^{\infty}f(j)\right|^q\le C \sum_{n=1}^{\infty}|nf(n)|^{q}, \qquad 1\le q<\infty, 
\]
which is a consequence of the classical Hardy's inequality for $q>1$ and for $q=1$ is obvious, the result follows in this case because
\begin{align*}
\|F\|_{\ell^q(\mathbb{Z})}^q& \le \sum_{n=1}^{\infty}\left(\left(\sum_{j=n}^{\infty}|f(-j)|\right)^q+\left(\sum_{j=n}^{\infty}|f(j)|\right)^q\right)
\\& \le C \sum_{n=1}^{\infty}n^q(|f(-n)|^q+|f(n)|^q)= C\||\cdot|f\|_{\ell^q(\mathbb{Z})}^q. 
\end{align*}

Now, we consider $N>1$. Firstly, we are going to prove the inequality
\begin{equation}
\label{eq:DZI-pp}
\left\|W_tf-\left(\sum_{\mb{k}\in \mathbb{Z}^N}f(\mb{k})\right)G_{t,N}\right\|_{\ell^p(\mathbb{Z}^N)}\le C t^{-1/2}\||\cdot|f\|_{\ell^p(\mathbb{Z}^N)}
\end{equation}
for $1\le p <N/(N-1)$. To do that, for a fix $\mb{n}\in \mathbb{Z}^N$, we consider the sets $A_1=\{\mb{k}\in \mathbb{Z}^N:|\mb{k}|\le |\mb{n}|/2\}$, $A_2=\{\mb{k}\in \mathbb{Z}^N:|\mb{n}|/2<|\mb{k}|< 2|\mb{n}|\}$, and $A_3=\{\mb{k}\in \mathbb{Z}^N:|\mb{k}|\ge 2|\mb{n}|\}$. With them we do the decomposition
\[
\left|W_tf(\mb{n})-\left(\sum_{\mb{k}\in \mathbb{Z}^N}f(\mb{k})\right)G_{t,N}(\mb{n})\right|\le P_1f(\mb{n})+P_2f(\mb{n})+P_3f(\mb{n}),
\]
with
\[
P_if(\mb{n})=\left|\sum_{\mb{k}\in \mathbb{Z}^N}\chi_{A_i}(\mb{n},\mb{k})f(\mb{n}-\mb{k})(G_{t,N}(\mb{k})-G_{t,N}(\mb{n}))\right|, \qquad i=1,2,3.
\]

From \eqref{eq:AM-GM-I}, in $A_1$ it is verified that $|G_{t,N}(\mb{k})-G_{t,N}(\mb{n})|\le C (G_{t,1}(|\mb{k}|/N))^N$ and 
\begin{align}
\label{eq:P1-aux}
P_1f(\mb{n})&\le C \sum_{|\mb{k}|\le |\mb{n}|/2}|f(\mb{n}-\mb{k})|(G_{t,1}(|\mb{k}|/N))^N\\&\le C \sum_{|\mb{k}|\le |\mb{n}|/2}|\mb{n}-\mb{k}||f(\mb{n}-\mb{k})|\frac{(G_{t,1}(|\mb{k}|/N))^N}{|\mb{k}|+1}.\notag
\end{align}
Then, by Minkowsky integral inequality and Lemma \ref{lem:norm-Gt-weight} (with $r=\gamma=1$), we obtain that
\[
\|P_1f\|_{\ell^p(\mathbb{Z}^N)}\le C\||\cdot|f\|_{\ell^p(\mathbb{Z}^N)}\sum_{k\in \mathbb{Z}^N}\frac{(G_{t,1}(|\mb{k}|))^N}{|\mb{k}|+1}\le C t^{-1/2}\||\cdot|f\|_{\ell^p(\mathbb{Z}^N)}.
\]

For $P_2f$, by using that $|\mb{k}|\sim |\mb{n}|$, we can apply Remark \ref{rem:lemma-heat} (a) to get 
\[
P_2(\mb{n})\le C\sum_{\mb{k}\in \mathbb{Z}^N}|\mb{n}-\mb{k}||f(\mb{n}-\mb{k})|H_{t,N}(\mb{k})=CF\ast H_{t,N}(\mb{n}),
\]
where $F(\mb{n})=|\mb{n}||f(\mb{n})|$. Then, by applying \eqref{eq:norm-H} and the Young inequality \eqref{eq:Young} for the convolution, we have
\[
\left\|P_2f\right\|_{\ell^p(\mathbb{Z}^N)}\le C t^{-1/2}\||\cdot|f\|_{\ell^p(\mathbb{Z}^N)}.
\]

Finally, we treat $P_3f$. First, we have to observe that for $p<N/(N-1)$, by H\"older's inequality, we obtain that
\begin{align*}
\sum_{|\mb{k}|\ge 2|\mb{n}|}|f(\mb{n}-\mb{k})|&\le \||\cdot|f\|_{\ell^p(\mathbb{Z}^N)}\left(\sum_{|\mb{k}|\ge 2|\mb{n}|}\frac{1}{|\mb{n}-\mb{k}|^{p'}}\right)^{1/p'}\\&\le C \||\cdot|f\|_{\ell^p(\mathbb{Z}^N)}(|\mb{n}|+1)^{N/p'-1}.
\end{align*}
Then, by using that in $A_3$ it is satisfied the estimate $|G_{t,N}(\mb{n})-G_{t,N}(\mb{k})|\le C (G_{t,N}(|\mb{n}|/N))^N$, 
\begin{equation}
\label{eq:P3-aux-0}
P_3f(\mb{n})\le C (G_{t,1}(|\mb{n}|/N))^N(|\mb{n}|+1)^{N/p'-1} \||\cdot|f\|_{\ell^p(\mathbb{Z}^N)}.
\end{equation}
In this way, applying Lemma \ref{lem:norm-Gt-weight} (with $r=p$ and $\gamma=p(1-N/p')$) the inequality
\[
\left\|P_2f\right\|_{\ell^p(\mathbb{Z}^N)}\le C t^{-1/2}\||\cdot|f\|_{\ell^p(\mathbb{Z}^N)}
\]
is attained and the proof of \eqref{eq:DZI-pp} is completed.

Now, we are going to check that
\begin{equation}
\label{eq:DZI-Inq}
\left\|W_tf-\left(\sum_{\mb{k}\in \mathbb{Z}^N}f(\mb{k})\right)G_{t,N}\right\|_{\ell^\infty(\mathbb{Z}^N)}\le C t^{-1/2-N/(2q)}\||\cdot|f\|_{\ell^q(\mathbb{Z}^N)}
\end{equation}
for $1\le q<N/(N-1)$. We consider the same decomposition than in the previous case. To analyze $P_1$, we consider the decomposition
\[
\|P_1f\|_{\ell^\infty(\mathbb{Z}^N)}=P_{1,1}+P_{1,2}
\]
where
\[
P_{1,1}=\sup_{|\mb{n}|/2\le n_t}|P_1f(\mb{n})|\qquad \text{and} \qquad P_{1,2}=\sup_{|\mb{n}|/2> n_t}|P_1f(\mb{n})|.
\]
So, from \eqref{eq:H-infty} and H\"older's inequality, it is verified that
\begin{align*}
P_{1,1}&\le C t^{-N/2-1}\sup_{|\mb{n}|/2\le n_t}|\mb{n}|\sum_{|\mb{k}|\le |\mb{n}|/2}|\mb{n}-\mb{k}||f(\mb{n}-\mb{k})|
\\&\le C t^{-N/2-1}\||\cdot|f\|_{\ell^q(\mathbb{Z}^N)}\sup_{|\mb{n}|/2\le n_t}|\mb{n}|\left(\sum_{|\mb{k}|\le |\mb{n}|/2}1\right)^{1/q'}
\\&\le C t^{-N/2-1}\||\cdot|f\|_{\ell^q(\mathbb{Z}^N)}\sup_{|\mb{n}|/2\le n_t}|\mb{n}|^{1+N/q'}\le C t^{-1/2-N/(2q)}\||\cdot|f\|_{\ell^q(\mathbb{Z}^N)}
\end{align*}
and this is enough. Now, to deal with $P_{1,2}$ we split the sum defining $P_1f$ in two pieces: one for $|\mb{k}|\le n_t$, denoted $P_{1,2}'$, and another one for $n_t<|\mb{k}|\le |\mb{n}|/2$, with the notation $P_{1,2}''$ for this one. Then, from \eqref{eq:bound-diff-P} and using H\"older's inequality, we get
\begin{align*}
P_{1,2}'&\le C t^{-1/2-N/2}\sup_{|\mb{n}|/2>n_t}\sum_{|\mb{k}|\le n_t}|\mb{n}-\mb{k}||f(\mb{n}-\mb{k})|\\&\le Ct^{-1/2-N/2}\||\cdot|f\|_{\ell^q(\mathbb{Z}^N)}\left(\sum_{|\mb{k}|\le n_t}1\right)^{1/q'}\le Ct^{-1/2-N/(2q)}\||\cdot|f\|_{\ell^q(\mathbb{Z}^N)}.
\end{align*}
Now, proceeding as in the proof of \eqref{eq:P1-aux}, we deduce that
\begin{align*}
P_{1,2}''&\le C\sup_{|\mb{n}|/2>n_t}\sum_{n_t<|\mb{k}|\le |\mb{n}|/2}|\mb{n}-\mb{k}||f(\mb{n}-\mb{k})|\frac{(G_{t,1}(|\mb{k}|/N))^N}{|\mb{k}|+1}
\\&\le C \||\cdot|f\|_{\ell^q(\mathbb{Z}^N)}\left(\sum_{n_t<|\mb{k}|}\frac{(G_{t,1}(|\mb{k}|/N))^{q'N}}{(|\mb{k}|+1)^{q'}}\right)^{1/q'}.
\end{align*}
Finally, the required estimate is obtained by using that $G_{t,1}(|\mb{k}|/N)\le C/ (|\mb{k}|+1)$. Indeed,
\begin{align*}
P_{1,2}''&\le C \||\cdot|f\|_{\ell^q(\mathbb{Z}^N)}\left(\sum_{|\mb{k}|> n_t}\frac{1}{(|\mb{k}|+1)^{q'(N+1)}}\right)^{1/q'}\\&\le Ct^{-1/2-N/(2q)}\||\cdot|f\|_{\ell^q(\mathbb{Z}^N)}.
\end{align*}
 
To obtain the estimate
\[
\|P_2f\|_{\ell^\infty(\mathbb{Z}^N)}\le C t^{-1/2-N/(2q)}\||\cdot|f\|_{\ell^q(\mathbb{Z}^N)}
\]
we proceed as in the proof of \eqref{eq:DZI-pp}, so we omit the details. 

Finally, we have to bound $P_3f$ and to do this we take the decomposition 
\[
P_3f(\mb{n})\le P_{3,1}+P_{3,2},
\]
with
\[
P_{3,1}=\sup_{2|\mb{n}|\le n_t}|P_3f(\mb{n})|\qquad \text{and} \qquad P_{3,2}=\sup_{2|\mb{n}|> n_t}|P_3f(\mb{n})|.
\]
For $P_{3,2}$, proceeding as in \eqref{eq:P3-aux-0} and using the bound \eqref{eq:size-G-t}, we deduce that
\begin{align*}
P_{3,2}&\le C \||\cdot|f\|_{\ell^q(\mathbb{Z}^N)} \sup_{2|\mb{n}|>n_t}(G_{t,1}(|\mb{n}|/N))^N(|\mb{n}|+1)^{N/q'-1} \\&\le C t^{-1/2-N/(2q)}\||\cdot|f\|_{\ell^q(\mathbb{Z}^N)}.
\end{align*}
Now, to treat $P_{3,1}$ we decompose the sum defining $P_3f$ in two pieces, $P_{3,2}'$ and $P_{3,2}''$, where $2|\mb{n}|\le |\mb{k}|\le n_t$ and $n_t<|\mb{k}|$, respectively. For $P_{3,2}'$ the required bound can be obtained as for $P_{1,2}'$ and we omit the details. To conclude the result for $P_{3,1}''$ we observe that, due to the restriction $q<N/(N-1)$,
\begin{align*}
\sum_{|\mb{k}|>n_t}|f(\mb{n}-\mb{k})|&\le \||\cdot|f\|_{\ell^q(\mathbb{Z}^N)}\left(\sum_{|\mb{k}|>n_t}\frac{1}{|\mb{n}-\mb{k}|^{q'}}\right)^{1/q'}\\&\le C t^{-1/2+N/(2q')}\||\cdot|f\|_{\ell^q(\mathbb{Z}^N)}.
\end{align*}
Then, applying \eqref{eq:size-G-t},
\begin{align*}
P_{3,1}''&\le C t^{-1/2+N/(2q')} \||\cdot|f\|_{\ell^q(\mathbb{Z}^N)}\sup_{2|\mb{n}|\le n_t} (G_{t,1}(|\mb{n}|/N))^N\\&\le C t^{-1/2-N/(2q)} \||\cdot|f\|_{\ell^q(\mathbb{Z}^N)}
\end{align*}
and the proof of \eqref{eq:DZI-Inq} is completed.

The proof of the theorem will be finished interpolating between \eqref{eq:DZI-pp} and \eqref{eq:DZI-Inq}.
\end{proof}

\begin{rem}
By using \eqref{eq:heat-0} and Minkowski integral inequality, for $0<t<T$ (with $T>100$ for example) we obtain
\begin{multline*}
\left\|W_tf-\left(\sum_{\mb{k}\in \mathbb{Z}^N}f(\mb{k})\right)G_{t,N}\right\|_{\ell^p(\mathbb{Z}^N)}\\\le
\|f\|_{\ell^1(\mathbb{Z}^N)}\left(\sup_{\mb{k}\in \mathbb{Z}^N}\left\|G_{t,N}(\cdot-\mb{k})\right\|_{\ell^p(\mathbb{Z}^N)}+\left\|G_{t,N}\right\|_{\ell^p(\mathbb{Z}^N)}\right)
\le C\||\cdot|f\|_{\ell^q(\mathbb{Z}^N)}.
\end{multline*}
Then, combining this fact with \eqref{eq:DZI} and with the hypotheses of Theorem \ref{thm:decay}, we get
\begin{equation}
\label{eq:decay-fine}
  \left\|W_tf-\left(\sum_{\mb{k}\in \mathbb{Z}^N}f(\mb{k})\right)G_{t,N}\right\|_{\ell^p(\mathbb{Z}^N)}\le C \min\{1,t^{-1/2-N/2(1/q-1/p)}\}\||\cdot|f\|_{\ell^q(\mathbb{Z}^N)}.
\end{equation}
\end{rem}

\subsection{Maximal operators for the heat and the Poisson semigroups}
From the well known identity
\[
e^{-|\beta| t}=\frac{1}{\sqrt{\pi}}\int_{0}^{\infty}\frac{e^{-u}}{\sqrt{u}}e^{-t^2\beta^2/(4u)}\, du
=\frac{t}{2\sqrt{\pi}}\int_{0}^{\infty}\frac{e^{-t^2/(4v)}}{\sqrt{v}}e^{-v\beta^2}\frac{dv}{v},
\]
we can define by subordination the Poisson semigroup
\begin{equation}
\label{eq:def-Poisson}
P_tf(\mb{n})=\frac{1}{\sqrt{\pi}}\int_{0}^{\infty}\frac{e^{-u}}{\sqrt{u}}W_{t^2/(4u)}f(\mb{n})\, du
=\frac{t}{2\sqrt{\pi}}\int_{0}^{\infty}\frac{e^{-t^2/(4v)}}{\sqrt{v}}W_{v}f(\mb{n})\frac{dv}{v}.
\end{equation}
It is an easy exercise to show that $P_t$ satisfies the Laplace type equation
\[
\frac{d^2}{dt^2}P_tf(\mb{n})+\Delta_N P_tf(\mb{n})=0.
\]

In this part of the paper we are interested in the maximal operators
\[
W^\star f(\mb{n})=\sup_{t>0}|W_tf(\mb{n})|
\]
and
\[
P^\star f(\mb{n})=\sup_{t>0}|P_tf(\mb{n})|.
\]
In fact, we want analyze weighted inequalities for these operators with weights in the Muckenhoupt class. We say that a positive sequence $w$ is a weight in $A_p(\mathbb{Z}^N)$ for $1<p<\infty$ when
\[
\left(\frac{1}{\mu(Q)}\sum_{\mb{n}\in Q}w(\mb{n})\right)\left(\frac{1}{\mu(Q)}\sum_{\mb{n}\in Q}(w(\mb{n}))^{-1/(p-1)}\right)^{p-1}\le C,
\]
where $Q$ is any cube in $\mathbb{Z}^N$ and $\mu(A)=\sum_{\mb{n}\in A}1$ for any $A\subset \mathbb{Z}^N$. The weight $w$ belongs to $A_1(\mathbb{Z}^N)$ when
\[
\left(\frac{1}{\mu(Q)}\sum_{\mb{n}\in Q}w(\mb{n})\right)\sup\left\{(w(\mb{n})^{-1}:\mb{n}\in \mathbb{Z}^N\right\}\le C.
\]

To obtain our result, we will use the Calder\'on-Zygmund theory adapts to our setting. We say that $K$, defined on $\mathbb{Z}^N\setminus\{\mb{0}\}$ and taking values in a Banach space $A$, is a standard kernel when
\begin{equation}
\label{eq:size-CZ}
\|K(\mb{n})\|_{A}\le \frac{C}{|\mb{n}|^{N}}
\end{equation}
and
\begin{equation}
\label{eq:smooth-CZ}
\|K(\mb{n})-K(\mb{n}+\mb{m})\|_{A}\le C \frac{|\mb{m}|}{|\mb{n}|^{N+1}}, \qquad |\mb{n}|>2|\mb{m}|.
\end{equation}
Let $T$ be an operator bounded from $\ell^q(\mathbb{Z}^N)$ to $\ell_A^q(\mathbb{Z}^N)$ (where we understand that $f\in \ell_A^q(\mathbb{Z}^N)$ when $\|f\|_{A}\in \ell^q(\mathbb{Z}^N)$) for some $1<q<\infty$ and for $f\in \ell^q(\mathbb{Z}^N)$ with compact support
\[
Tf(\mb{n})=K\ast f(\mb{n}), \qquad \mb{n}\notin \supp f,
\]
where $K$ is a standard kernel taking values in $A$, we say that $T$ is Calder\'on-Zygmund operator.

A Calder\'on-Zygmung operator $Tf$ is bounded from $\ell^{p}(\mathbb{Z}^N,w)$ into $\ell^{p}_A(\mathbb{Z}^N,w)$ for $1<p<\infty$ and $w\in A_p(\mathbb{Z}^N)$, and from $\ell^{1}(\mathbb{Z}^N,w)$ into $\ell^{1,\infty}_A(\mathbb{Z}^N,w)$ for $w\in A_1(\mathbb{Z}^N)$ (see \cite{RRT,RT}). As usual the space $\ell^{1,\infty}(\mathbb{Z}^{N},w)$ is defined as the space of sequences such that $\|f\|_{\ell^{1,\infty}(\mathbb{Z}^N,w)}<\infty$, where
\[
\|f\|_{\ell^{1,\infty}(\mathbb{Z}^N,w)}=\sup_{\lambda>0}\left(\lambda\sum_{\mb{n}\in A_\lambda}w(\mb{n})\right)
\]
and
\[
A_\lambda=\{\mb{n}\in\mathbb{Z}^N:|f(\mb{n})|>\lambda\}.
\]

\begin{thm}
  For $T$ equal $W^\star$ or $P^\star$ and $w\in A_p(\mathbb{Z}^N)$, the inequalities
  \[
  \|Tf\|_{\ell^{p}(\mathbb{Z}^N,w)}\le \|f\|_{\ell^{p}(\mathbb{Z}^N,w)}, \qquad 1<p<\infty,\qquad \ell^2(\mathbb{Z}^N)\cap \ell^p(\mathbb{Z}^N,w),
  \]
  and
  \[
  \|Tf\|_{\ell^{1,\infty}(\mathbb{Z}^N,w)}\le \|f\|_{\ell^{1}(\mathbb{Z}^N,w)}, \qquad \ell^2(\mathbb{Z}^N)\cap \ell^1(\mathbb{Z}^N,w),
  \]
  hold.
\end{thm}

\begin{proof}
We focus on $W^\star$ because the result for $P^\star$ follows from its definition by subordination \eqref{eq:def-Poisson}. By using that
\[
W^\ast f(\mb{n})=\|W_{t}f(\mb{n})\|_{L^\infty((0,\infty))},
\]
we prove that $W^\star$ is a Calder\'on-Zygmund operator taking values in $L^{\infty}((0,\infty))$. From \eqref{eq:contrac-heat}, applying Stein's Maximal Theorem of diffusion semigroups (see \cite[Chapter III, Section 2]{Stein}), we have
\[
\|W^\ast f\|_{\ell^p(\mathbb{Z}^N)}\le C \|f\|_{\ell^p(\mathbb{Z}^N)}, \qquad 1<p<\infty.
\]
In particular we have the case $p=2$. To conclude it is enough with the estimates
\begin{equation}
\label{eq:size-Wt}
\|G_{t,N}(\mb{n})\|_{L^\infty((0,\infty))}\le \frac{C}{(|\mb{n}|+1)^{N}}
\end{equation}
and
\begin{equation}
\label{eq:smooth-Wt}
\|G_{t,N}(\mb{n})-G_{t,N}(\mb{n}+\mb{m})\|_{L^\infty((0,\infty))}\le C\frac{|\mb{m}|}{(|\mb{n}|+|\mb{n}+\mb{m}|)^{N+1}}, \qquad |\mb{n}|>2|\mb{m}|.
\end{equation}

The bound \eqref{eq:size-Wt} has been proved yet in \eqref{eq:Gt-decay-n}.

Now, from  \eqref{eq:diff-G} and applying \eqref{eq:bound-I} with $\alpha=1/N$, we deduce that for $|\mb{n}|>2|\mb{m}|$ it is verified
\begin{align*}
|G_{t,N}(\mb{n})-G_{t,N}(\mb{n}+\mb{m})|&\le C |\mb{m}|(|\mb{n}|+|\mb{n}+\mb{m}|)\left(t^{-1/N}G_{t,1}\left(\frac{|\mb{n}|+|\mb{n}+\mb{m}|}{K}\right)\right)^{N}
\\&\le C\frac{|\mb{m}|}{(|\mb{n}|+|\mb{n}+\mb{m}|)^{N+1}}
\end{align*}
and the proof of \eqref{eq:smooth-Wt} is completed.
\end{proof}

\section{The fractional integrals and the Riesz transforms}
\label{sec:Riesz}
From the identity
\[
\frac{1}{\lambda^\sigma}=\frac{1}{\Gamma(\sigma)}\int_{0}^{\infty}e^{-\lambda t}t^{\sigma-1}\, dt,
\]
we define the fractional integrals (or negative powers of the Laplacian $\Delta_N$) by the relation
\[
(-\Delta_N)^{-\sigma}f(\mb{n})=\frac{1}{\Gamma(\sigma)}\int_{0}^{\infty} W_t f(\mb{n})t^{\sigma-1}\, dt.
\]
By using that (see \cite{Lebedev} or take $\alpha=-1/2$ in \eqref{eq:asym-I})
\[
e^{-t}I_k(t)\sim t^{-1/2}, \qquad t\to \infty,
\]
it is clear that the powers $(-\Delta_N)^{-\sigma}$ are well defined for $0<2\sigma<N$.

Having these operators, for $N\ge 2$ we can define the Riesz transforms as
\[
R_if(\mb{n})=\delta_i^+ (-\Delta_N)^{-1/2}f(\mb{n}), \qquad i=1,\dots,N.
\]
In \cite{CGRTV} it was analyzed the Riesz transform for $N=1$ defining it as a limit. In fact, it was defined as
\[
Rf(n)=\lim_{\sigma\to (1/2)^{-}}\delta^+ (-\Delta_1)^{-\sigma}f(n). \qquad n\in \mathbb{Z}.
\]
From this, the identity
\[
Rf(n)=\sum_{k\in\mathbb{Z}}\frac{f(n)}{n-k+1/2}
\]
was obtained; i.e., it matches with the discrete Hilbert transform. 

In this section, firstly, we show some basic properties of the fractional integrals and we provide a Hardy-Littlewood-Sobolev inequality in our setting and an extension of it. We continue with the study of multidimensional discrete Schauder inequalities for $(-\Delta_N)^{-\sigma}$. To conclude we analyze the mapping properties of the Riesz transforms on the H\"older classes and on $\ell^p(\mathbb{Z}^N,w)$


\subsection{Basic properties of the fractional integrals}
In this part, we provide some properties of the fractional integrals, we prove a Hardy-Littlewood-Sobolev inequality for them and an extension of it. 

In some points a long the paper we will consider the spaces
\begin{equation}
\label{eq:def-l-sigma}
\ell_\sigma(\mathbb{Z}^N)=\ell^1(\mathbb{Z}^N,w_\sigma)
\end{equation}
with $w_{\sigma}(\mb{n})=(1+|\mb{n}|)^{2\sigma-N}$. Moreover, we will denote $\|f\|_{\ell_\sigma(\mathbb{Z}^N)}=\|f\|_{\ell^1(\mathbb{Z}^N,w_\sigma)}$.

\begin{thm}
Let $0<\sigma <N/2$ and $f\in \ell_\sigma(\mathbb{Z}^N)$. We have the pointwise formula
\begin{equation}
\label{eq:Ksigma-convo}
(-\Delta_N)^{-\sigma}f(\mb{n})=\sum_{\mb{k}\in\mathbb{Z}^N} f(\mb{n}-\mb{k})K_{\sigma}(\mb{k})
\end{equation}
where 
\[
K_{\sigma}(\mb{n})=\frac{1}{\Gamma(\sigma)}\int_{0}^{\infty}G_{t,N}(\mb{n})t^{\sigma-1}\, dt
\]
and it verifies that
\begin{equation}
\label{eq:bound-frac}
0<K_\sigma(\mb{n})\le \frac{C}{(|\mb{n}|+1)^{N-2\sigma}}.
\end{equation}
\end{thm}

\begin{proof}
First we are going to check that the heat operator is well defined for $f\in \ell_{\sigma}(\mathbb{Z}^N)$.
Indeed, for each $L> 0$ with $t$ fixed, by using \eqref{eq:AM-GM-I} and the asymptotic \eqref{eq:asym-I},
\begin{align*}
\sum_{{|\mb{k}|>L}}G_{t,N}(\mb{k})|f(\mb{n}-\mb{k})|&\le C e^{-2Nt} N^{N/2} \sum_{{|\mb{k}|>L}}\left(\frac{(Net)^{|\mb{k}|/N}}{|\mb{k}|^{|\mb{k}|/N+1/2}}\right)^N|f(\mb{n}-\mb{k})|
\\&\le C \|f\|_{\ell_\sigma(\mathbb{Z}^N)}e^{-2Nt}N^{N/2}\sup_{{|\mb{k}|>L}}
\frac{(Net)^{|\mb{k}|}(1+|\mb{n}-\mb{k}|)^{N-2\sigma}}{|\mb{k}|^{|\mb{k}|+N/2}}
\\&\le C \|f\|_{\ell_\sigma(\mathbb{Z}^N)},
\end{align*}
where the last constant depends on $t$, $N$, $L$, $\mb{n}$, and $\sigma$ but not on $f$.

Let us prove \eqref{eq:Ksigma-convo}. To do it we can suppose that $f$ is a non negative sequence because in other case we can consider its decomposition in positive and negative parts. Then, applying Tonelli's theorem, we have
\begin{align*}
(-\Delta_N)^{-\sigma}f(\mb{n})&=\frac{1}{\Gamma(\sigma)}\int_{0}^{\infty}\sum_{\mb{k}\in \mathbb{Z}^N}f(\mb{n}-\mb{k})G_t(\mb{k})t^{\sigma-1}\, dt\\&=\frac{1}{\Gamma(\sigma)}\sum_{\mb{k}\in \mathbb{Z}^N}f(\mb{n}-\mb{k})\int_{0}^{\infty}G_{t,N}(\mb{k})t^{\sigma-1}\, dt=
\sum_{\mb{k}\in \mathbb{Z}^N}f(\mb{n}-\mb{k})K_\sigma(\mb{k}).
\end{align*}

Now, let us go with the proof of \eqref{eq:bound-frac}. The positivity of $K_\sigma$ is a simple consequence of the positivity of $G_{t,N}$.
Now, to obtain the upper bound for the kernel, we consider the decomposition
\[
K_{\sigma}(\mb{n})=I_1+I_2
\]
with
\[
I_1=\frac{1}{\Gamma(\sigma)}\int_{0}^{|\mb{n}|^2}G_{t,N}(\mb{n})t^{\sigma-1}\, dt
\qquad\text{ and }\qquad
I_2=\frac{1}{\Gamma(\sigma)}\int_{|\mb{n}|^2}^\infty G_{t,N}(\mb{n})t^{\sigma-1}\, dt.
\]
We consider $\mb{n}\not=\mb{0}$, because the bound $K_\sigma(\mb{0})\le C$ is clear. For $I_1$ we use \eqref{eq:Gt-decay-n} to have
\[
I_1\le \frac{C}{|\mb{n}|^N}\int_{0}^{|\mb{n}|^2}t^{\sigma-1}\, dt=\frac{C}{|\mb{n}|^{N-2\sigma}},
\]
and for $I_2$ we apply the bound in \eqref{eq:size-G-t} to obtain that
\[
I_2\le C \int_{|\mb{n}|^2}^{\infty}\frac{dt}{t^{N/2+1-\sigma}}=\frac{C}{|\mb{n}|^{N-2\sigma}}.\qedhere
\]
\end{proof}

To prove our next theorem we need the following result that we can see in \cite[Theorem 1.6]{Chinos-Convo}.
\begin{thm}
Let be
  \[
  \mathcal{I}_\lambda f(\mb{n})=\sum_{\begin{smallmatrix}
                                        \mb{k}\in \mathbb{Z}^N \\
                                        \mb{k}\not= \mb{n}
                                      \end{smallmatrix}}\frac{f(\mb{k})}{|\mb{n}-\mb{k}|^{n-2\sigma}}, \qquad 0<2\sigma <N,
  \]
and $\langle \mb{n}\rangle=(1+|\mb{n}|^2)^{1/2}$. Taking $1\le p,q< \infty$, $s,t\in \mathbb{R}$ and the sets of conditions
\[
(\mathcal{C}_1)\qquad
\begin{cases}
s\le t,\\[3pt]
\dfrac{2\sigma}{N}+\max\left\{\dfrac{1}{p}+\dfrac{s}{N},0\right\}\le \max\left\{\dfrac{1}{q}+\dfrac{t}{N},0\right\},
\\[3pt]
\dfrac{2\sigma}{N}+\dfrac{1}{p}+\dfrac{s}{N}\le 1,\\[3pt]
\end{cases}
\]
\[
(\mathcal{C}_2)\qquad
\begin{cases}
s\le t=2\sigma-N,
q=1,\\[3pt]
\dfrac{1}{p}+\dfrac{s}{N}<0
\end{cases}
\]
and
\[
(\mathcal{C}_3)\qquad
\begin{cases}
s\le t,\\[3pt]
\dfrac{2\sigma}{N}+\dfrac{1}{p}+\dfrac{s}{N}=\dfrac{1}{q}+\dfrac{t}{N},
\\[3pt]
\dfrac{1}{p}\le \dfrac{1}{q},\quad 0<\dfrac{1}{p}+\dfrac{s}{N},\quad \dfrac{1}{q}+\dfrac{t}{N}<1,
\end{cases}
\]
it is verified that
\[
\|\langle \mb{\cdot}\rangle^s \mathcal{I}_{2\sigma}f\|_{\ell^p(\mathbb{Z}^N)}\le \|\langle \mb{\cdot}\rangle^t f\|_{\ell^q(\mathbb{Z}^N)}
\]
if and only if $(p,q,s,t)$ satisfies one of the conditions $\mathcal{C}_1$, $\mathcal{C}_2$, and $\mathcal{C}_3$. 
\end{thm}

\begin{rem}
  The previous theorem in its full version includes the case $p,q=\infty$ but the definition of the Lebesgue for such value in \cite{Chinos-Convo} is different from ours one, so we omit this particular value from the theorem. 
\end{rem}

Now, we can present the weighted Hardy-Littlewood-Sobolev for the fractional powers $(-\Delta_N)^{-\sigma}$.
\begin{thm}
Let $0<2\sigma<N$, $1\le p,q< \infty$ and $s,t\in \mathbb{R}$. If $(q,p,s,t)$ verify one of the conditions $\mathcal{C}_1$, $\mathcal{C}_2$, and $\mathcal{C}_3$, then
\[
\|\langle \mb{\cdot}\rangle^s (-\Delta_N)^{-\sigma}f\|_{\ell^p(\mathbb{Z}^N)}\le \|\langle \mb{\cdot}\rangle^t f\|_{\ell^q(\mathbb{Z}^N)}.
\]
\end{thm}
\begin{proof}
The result is immediate from the previous theorem because, by \eqref{eq:bound-frac}, it is verified that
\[
|(-\Delta_N)^{-\sigma}f(\mb{n})|\le C \left(|\mathcal{I}_{2\sigma}f(\mb{n})|+|f(\mb{n})|\right).\qedhere
\]
\end{proof}

Taking $t=s=0$ in the previous theorem, we recover the classical Hardy-Littlewood-Sobolev inequality in our setting. However, we can go further.
\begin{cor}
Let $0<2\sigma<N$, $1<q<p<\infty$ with $1/p\le 1/q-2\sigma/N$, then the inequality
\[
\|(-\Delta_N)^{-\sigma}f\|_{\ell^p(\mathbb{Z}^N)}\le C \|f\|_{\ell^q(\mathbb{Z}^N)}
\]
holds.

Moreover, for $1<q<p\le 2$ the Hardy-Littlewood-Sobolev implies $1/p\le 1/q-2\sigma/N$
\end{cor}
\begin{proof}
We have to proof the second part of the result. To do this, we consider the sequence $f(\mb{n})=G_{t,N}(\mb{n})$ with $t>1$.
In this way,
\begin{align*}
(-\Delta_N)^{-\sigma}f(\mb{n})&=\int_{0}^{\infty}\sum_{\mb{k}\in \mathbb{Z}^N}G_{t,N}(\mb{k})G_{s,N}(\mb{n}-\mb{k})s^{\sigma-1}\, ds
\\&=\int_{0}^{\infty}G_{t+s,N}(\mb{n})s^{\sigma-1}\, ds\ge \int_{0}^{t}G_{t+s,N}(\mb{n})s^{\sigma-1}\, ds\ge G_{2t}(\mb{n})t^{\sigma}.
\end{align*}
Then, by the Hardy-Littlewood-Sobolev, we have
\[
t^\sigma\|G_{2t,N}\|_{\ell^p(\mathbb{Z}^N)}\le C \|G_{t,N}\|_{\ell^q(\mathbb{Z}^N)}
\]
and, applying \eqref{eq:sim-G}, it becomes $t^{\sigma-N/2(1/q-1/p)}\le C$ 
and this implies $1/p\le 1/q-2\sigma/N$.
\end{proof}

To finish this subsection we show a new inequality extending the classical Hardy-Littlewood-Sobolev one. 
\begin{thm}
Let $0<2\sigma<N$, $1\le q\le p\le \infty$ with $1\le q <N/(N-1)$ and $1/p< 1/q-(2\sigma-1)/N$, and $|\cdot|f\in \ell^q(\mathbb{Z}^N)$, then the inequality
\[
\left\|(-\Delta_N)^{-\sigma}f-\left(\sum_{\mb{k}\in \mathbb{Z}^N}f(\mb{k})\right)K_\sigma\right\|_{\ell^p(\mathbb{Z}^N)}\le C \||\cdot|f\|_{\ell^q(\mathbb{Z}^N)}
\]
holds.
\end{thm}
\begin{proof}
The result follows from Minkowsky integral inequality and \eqref{eq:decay-fine}. Indeed,
\begin{multline*}
\left\|(-\Delta_N)^{-\sigma}f-\left(\sum_{\mb{k}\in \mathbb{Z}^N}f(\mb{k})\right)K_\sigma\right\|_{\ell^p(\mathbb{Z}^N)}\\
\begin{aligned}
&\le
\int_{0}^{\infty}\left\|W_tf-\sum_{\mb{k}\in \mathbb{Z}^N}f(\mb{k})G_{t,N}\right\|_{\ell^p(\mathbb{Z}^N)}t^{\sigma-1}\, dt
\\
&\le C \||\cdot|f\|_{\ell^q(\mathbb{Z}^N)}\int_{0}^{\infty}\min\{1,t^{-1/2-N/2(1/q-1/p)}\}t^{\sigma-1}\, dt\\&\le C\||\cdot|f\|_{\ell^q(\mathbb{Z}^N)}.
\end{aligned}
\end{multline*}
\end{proof}

\subsection{Discrete Schauder estimates for the fractional integrals}

Now, we prove multidimensional discrete Schauder estimates for the fractional integrals and to do that we need to define the H\"older spaces. Given a sequence $f=\{f(\mb{k})\}_{\mb{k}\in \mathbb{Z}^N}$ and $0<\alpha\le 1$, we say that $f$ belongs to the H\"older space $C^{0,\alpha}(\mathbb{Z}^N)$ when
\[
[f]_{C^{0,\alpha}(\mathbb{Z}^N)}:=\sup_{\mb{n},\mb{m}\in\mathbb{Z}}\frac{|f(\mb{n})-f(\mb{m})|}{|\mb{n}-\mb{m}|^{\alpha}}
<\infty.
\]
For a natural number $k$ such that $k\ge 1$
we say that $f\in C^{k,\alpha}(\mathbb{Z}^N)$
when it is verified that
\[
[f]_{C^{k,\alpha}(\mathbb{Z}^N)}:=\sum_{m_1+\cdots+m_N=k}[(\delta_1^+)^{m_1}(\delta_2^{+})^{m_2}\cdots (\delta_N^+)^{m_N} f]_{C^{0,\alpha}(\mathbb{Z}^N)}<\infty.
\]

Sometimes it is common write the definition of H\"older classes combining $\delta_i^+$ and $\delta_i^{-}$ but such definitions are equivalents to the our one due to the relation $\delta_i^{-}f(\mb{n})=\delta_i^{+}f(\mb{n}-\mb{e}_i)$.

The next result is the multidimensional version of \cite[Theorem 1.6]{CRSTV}.

\begin{thm}
\label{thm:Schauder}
Let $k\ge 0$, $0<\alpha\le 1$, $0<\sigma<1/2$ and $f\in \ell_{\sigma}(\mathbb{Z}^N)$.
\begin{enumerate}
\item[a)] If $f\in C^{k,\alpha}(\mathbb{Z}^N)$ and $2\sigma+\alpha<1$ then $(-\Delta)^{-\sigma}f\in C^{k,2\sigma+\alpha}(\mathbb{Z}^N)$ and
    \[
    [(-\Delta)^{-\sigma}f]_{C^{k,2\sigma+\alpha}(\mathbb{Z}^N)}\le C [f]_{C^{k,\alpha}(\mathbb{Z}^N)}.
    \]
\item[b)] If $f\in C^{k,\alpha}(\mathbb{Z}^N)$ and $2\sigma+\alpha>1$ then $(-\Delta)^{-\sigma}f\in C^{k+1,2\sigma+\alpha-1}(\mathbb{Z}^N)$ and
    \[
    [(-\Delta)^{-\sigma}f]_{C^{k+1,2\sigma+\alpha-1}(\mathbb{Z}^N)}\le C [f]_{C^{k,\alpha}(\mathbb{Z}^N)}.
    \]
\item[c)] If $f\in \ell^\infty(\mathbb{Z}^N)$  then $(-\Delta)^{-\sigma}f\in C^{0,2\sigma}(\mathbb{Z}^N)$ and
    \[
    [(-\Delta)^{-\sigma}f]_{C^{0,2\sigma}(\mathbb{Z}^N)}\le C \|f\|_{\ell^{\infty}(\mathbb{Z}^N)}.
    \]
\end{enumerate}
\end{thm}

Before starting with the proof of the Schauder estimates, we need the four following lemmas.

\begin{lem}
\label{lem:null}
  For $N\ge 1$ and $0<\sigma<1/2$, we have
  \[
  \sum_{k\in \mathbb{Z}^N}((-\Delta_N)^{-\sigma}(\mb{n}-\mb{k})-(-\Delta_N)^{-\sigma}(\mb{k}))=0.
  \]
\end{lem}

\begin{proof}
From the obvious identity
\[
 \sum_{k\in \mathbb{Z}^N}(G_{t,N}(\mb{n}-\mb{k})-G_{t,N}(\mb{k}))=0,
\]
we deduce that
\[
\int_{0}^{\infty} \sum_{k\in \mathbb{Z}^N}(G_{t,N}(\mb{n}-\mb{k})-G_{t,N}(\mb{k})) t^{\sigma-1}\,dt=0,
\]
Let us see now that, for $0<\sigma<1/2$, we can apply Fubini theorem. From \eqref{eq:sum-L1-multi}, it is clear that
\[
\int_{0}^{|\mb{n}|^2}\sum_{k\in \mathbb{Z}^N}|G_{t,N}(\mb{n}-\mb{k})-G_{t,N}(\mb{k})| t^{\sigma-1}\,dt\le
 2\int_{0}^{|\mb{n}|^2}t^{\sigma-1}\,dt\le C |\mb{n}|^{2\sigma}<\infty,
\]
so we have to focus on the integral in the interval $(|\mb{n}|^2,\infty)$. To do that we consider
\[
I_1= \int_{|\mb{n}|^2}^{\infty}\sum_{|\mb{k}|\le 2|\mb{n}|}|G_{t,N}(\mb{n}-\mb{k})-G_{t,N}(\mb{k})| t^{\sigma-1}\,dt
\]
and
\[
I_2= \int_{|\mb{n}|^2}^{\infty}\sum_{|\mb{k}|> 2|\mb{n}|}|G_{t,N}(\mb{n}-\mb{k})-G_{t,N}(\mb{k})| t^{\sigma-1}\,dt.
\]
For $I_1$ we use the bound \eqref{eq:H-infty} to obtain that
\begin{align*}
I_1&\le C |\mb{n}|^2\int_{|\mb{n}|^2}^{\infty}\left(\sum_{|\mb{k}|\le 2|\mb{n}|}1\right) t^{\sigma-N/2-2}\,dt
\\&\le C |\mb{n}|^{2+N}\int_{|\mb{n}|^2}^{\infty}t^{\sigma-N/2-2}\,dt\le C |\mb{n}|^{2\sigma}<\infty.
\end{align*}
To analyze $I_2$ we apply the identity $G_{t,N}(\mb{n}-\mb{k})=G_{t,N}(\mb{k}-\mb{n})$, \eqref{eq:diff-G} and \eqref{eq:norm-H} to deduce the estimate
\begin{align*}
I_2&\le C |\mb{n}|\int_{|\mb{n}|^2}^{\infty}\left(\sum_{|\mb{k}|> 2|\mb{n}|}H_{t,N}\left(\frac{|\mb{k}|}{K}\right)\right) t^{\sigma-1}\,dt
\\&\le  C |\mb{n}|\int_{|\mb{n}|^2}^{\infty} t^{\sigma-3/2}\,dt\le C|n|^{2\sigma}<\infty.
\end{align*}

In this way,
\begin{multline*}
 \sum_{k\in \mathbb{Z}^N}((-\Delta_N)^{-\sigma}(\mb{n}-\mb{k})-(-\Delta_N)^{-\sigma}(\mb{k}))\\=
 \int_{0}^{\infty} \sum_{k\in \mathbb{Z}^N}(G_{t,N}(\mb{n}-\mb{k})-G_{t,N}(\mb{k})) t^{\sigma-1}\,dt=0
\end{multline*}
and the proof is completed.
\end{proof}

\begin{lem}
\label{lem:diff-K}
  Let $N\ge 1$, $\mb{n},\mb{m}\in \mathbb{Z}^N$ such that $|\mb{n}|>2|\mb{n}-\mb{m}|$, and $0<2\sigma<N$.
  Then the inequality
  \begin{equation}
  \label{eq:diff-K}
  |K_{\sigma}(\mb{n})-K_\sigma(\mb{m})|\le \frac{C|\mb{n}-\mb{m}|}{(|\mb{n}|+|\mb{m}|)^{N+1-2\sigma}}
  \end{equation}
  holds with a constant $C$ independent of $\mb{n}$ and $\mb{m}$.
\end{lem}

\begin{proof}
  It is clear that
  \[
  |K_{\sigma}(\mb{n})-K_\sigma(\mb{m})|\le \int_{0}^{\infty}|G_{t,N}(\mb{n})-G_{t,n}(\mb{m})|t^{\sigma-1}\,dt=J_1+J_2,
  \]
  where
  \[
  J_1=\int_{0}^{(|\mb{n}|+|\mb{m}|)^2}|G_{t,N}(\mb{n})-G_{t,n}(\mb{m})|t^{\sigma-1}\,dt
  \]
  and
  \[
  J_2=\int_{(|\mb{n}|+|\mb{m}|)^2}^{\infty}|G_{t,N}(\mb{n})-G_{t,n}(\mb{m})|t^{\sigma-1}\,dt.
  \]
To estimate $J_1$ we consider Lemma \ref{lem:diff-G} and apply \eqref{eq:bound-I} with $\alpha=1/N$. Indeed,
  \begin{align*}
  J_1&\le C |\mb{n}-\mb{m}|(|\mb{n}|+|\mb{m}|)
  \int_{0}^{(|\mb{n}|+|\mb{m}|)^2}\frac{1}{t}\left(G_{t,1}\left(\frac{|\mb{n}|+|\mb{m}|}{K}\right)\right)^Nt^{\sigma-1}\,dt
\\&\le \frac{C|\mb{n}-\mb{m}|}{(|\mb{n}|+|\mb{m}|)^{N+1}}
  \int_{0}^{(|\mb{n}|+|\mb{m}|)^2}t^{\sigma-1}\,dt
  \le \frac{C|\mb{n}-\mb{m}|}{(|\mb{n}|+|\mb{m}|)^{N+1-2\sigma}}.
  \end{align*}
  The bound for $J_2$ is obtained from \eqref{eq:H-infty} in the following way
  \[
  J_2\le  |\mb{n}-\mb{m}|(|\mb{n}|+|\mb{m}|)\int_{(|\mb{n}|+|\mb{m}|)^2}^{\infty}t^{\sigma-N/2-2}\,dt
  \le \frac{C |\mb{n}-\mb{m}|}{(|\mb{n}|+|\mb{m}|)^{N+1-2\sigma}}.\qedhere
  \]
\end{proof}

\begin{lem}
\label{lem:diff2-G}
  Let $N\ge 1$, $\mb{n},\mb{m}\in \mathbb{Z}^N$ such that $|\mb{n}|>2|\mb{n}-\mb{m}|$, and
  \[
  \mathcal{H}_{t,N}(|\mb{n}|)=\left(\frac{1}{t}+\frac{z^2}{t^2}\right)
  \left(G_{t,1}\left(z\right)\right)^N, \qquad z>0.
  \]
  Then the inequality
  \begin{equation}
  \label{eq:diff2-G}
  |\delta_i^+(G_{t,N}(\mb{n})-G_{t,N}(\mb{m}))|\le C |\mb{n}-\mb{m}| \mathcal{H}_{t,N}\left(\frac{|\mb{n}|+|\mb{m}|}{K}\right),\qquad i=1,\dots,N,
  \end{equation}
  holds with constants $C$ and $K>1$ independent of $\mb{n}$ and $\mb{m}$.
\end{lem}
\begin{proof}
We consider $i=1$ because the other cases can be done in the same way. We can prove the result for $n_i\not= m_i$ for $i=1,\dots,N$ because this situation implies the general case. In fact, to deduce the general case from this one, as in Lemma \ref{lem:diff-G}, let us suppose that $n_i=m_i$ for $j$ values of $i$, with $1<j<N$ and consider the decompositions
\[
\mb{n}=\mb{n}_1\cup \mb{n}_2\qquad \text{and} \qquad \mb{m}=\mb{n}_1\cup \mb{m}_2,
\]
where $\mb{n}_1$ are the common values of $\mb{n}$ and $\mb{m}$ and $\mb{n}_2$ and $\mb{m}_2$ are the other components. Now, we have to distinguish two possibilities $n_1\in \mb{n}_1$ and $n_1\notin \mb{n}_1$. In the first situation, by using \eqref{eq:diff-I} and Lemma \ref{lem:diff-G}, we have
\begin{multline*}
|\delta_1^+(G_{t,N}(\mb{n})-G_{t,N}(\mb{m}))|\le C \frac{|n_1|+1}{t}G_{t,j}(\mb{n}_1)|G_{t,N-j}(\mb{n}_2)-G_{t,N-j}(\mb{m}_2)|\\
\begin{aligned}
&\le C |\mb{n}_2-\mb{m_2}|\frac{(|\mb{n}_1|+1)(|\mb{n_2}|+|\mb{m_2}|)}{t^2}G_{t,j}(\mb{n}_1)
\left(G_{t,1}\left(\frac{|\mb{n}_2|+|\mb{m}_2|}{K}\right)\right)^{N-j}
\\&\le C |\mb{n}-\mb{m}| \left(\frac{1}{t}+\frac{(|\mb{n}|+|\mb{m}|)^2}{t^2}\right)G_{t,j}(\mb{n}_1)
\left(G_{t,1}\left(\frac{|\mb{n}_2|+|\mb{m}_2|}{K}\right)\right)^{N-j}
\end{aligned}
\end{multline*}
and we conclude as in Lemma \ref{lem:diff-G}. When $n_1\notin \mb{n}_1$, if we consider proved the case with different components, it is verified that
\begin{align*}
|\delta_1^+(G_{t,N}(\mb{n})-G_{t,N}(\mb{m}))|&= G_{t,j}(\mb{n}_1)|\delta_1^+(G_{t,N-j}(\mb{n}_2)-G_{t,N-j}(\mb{m}_2))|
\\&\le C  G_{t,j}(\mb{n}_1)|\mb{n}_2-\mb{m}_2| \mathcal{H}_{t,N-j}\left(\frac{|\mb{n}_2|+|\mb{m}_2|}{K}\right)
\end{align*}
and the proof of this case is finished again as in Lemma \ref{lem:diff-G}.

Now, we assume the restrictions \eqref{eq:mono-n} on $\mb{n}$ and $\mb{m}$ as in the proof of Lemma \ref{lem:diff-G}.
Then, using \eqref{eq:rara}, we get
\begin{multline}
\label{eq:diff2-aux1}
\delta_1^+(G_{t,N}(\mb{n})-G_{t,N}(\mb{m}))=\delta^+ (G_{t,1}(n_1)- G_{t,1}(m_1))\prod_{k=2}^{N}G_{t,1}(n_k)\\+
\delta^+ G_{t,1}(m_1)\sum_{i=2}^{N}(G_{t,1}(n_i)-G_{t,1}(m_i))\prod_{j=2}^{i-1}G_{t,1}(m_j)\prod_{k=i+1}^{N}G_{t,1}(n_k).
\end{multline}
Proceeding as in Lemma \ref{lem:diff-G} and using \eqref{eq:diff-I}, we have
\begin{multline}
\label{eq:diff2-aux2}
|\delta^+ G_{t,1}(m_1)(G_{t,1}(n_i)-G_{t,1}(m_i))|\\\le C \frac{(m_1+1)|n_i-m_i|(m_i+n_i+1)}{t^2}G_{t,1}(m_1)\max\{G_{t,1}(n_i),G_{t,1}(m_i)\}.
\end{multline}
Now, taking $n_1< m_1 $, it is obtained that
\begin{align*}
\delta^+ (G_{t,1}(n_1)- G_{t,1}(m_1))&=-\sum_{k=n_1}^{m_1-1}(\delta^+ G_{t,1}(k+1)-\delta^+ G_{t,1}(k))
\\&=-\sum_{k=n_1}^{m_1-1}(G_{t,1}(k+2)-2G_{t,1}(k+1)+G_{t,1}(k))
\end{align*}
and, applying \eqref{eq:diff2-I},
\begin{align*}
|\delta^+ (G_{t,1}(n_1)- G_{t,1}(m_1))|&\le C \sum_{k=n_1}^{m_1-1}\left(\frac{1}{t}+\frac{(k+1)(k+2)}{t^2}\right)G_{t,1}(k)
\\&\le C(m_1-n_1)\left(\frac{1}{t}+\frac{(m_1+n_1)^2}{t^2}\right)G_{t,1}(n_1).
\end{align*}
In general,
\begin{multline}
\label{eq:diff2-aux3}
|\delta^+(G_{t,1}(n_1)- G_{t,1}(m_1))|\\\le  C|m_1-n_1|\left(\frac{1}{t}+\frac{(m_1+n_1)^2}{t^2}\right)\max\{G_{t,1}(n_1),G_{t,1}(m_1)\}.
\end{multline}
With \eqref{eq:diff2-aux1}, \eqref{eq:diff2-aux2}, and \eqref{eq:diff2-aux3} the result can be concluded as in Lemma \ref{lem:diff-G}.
\end{proof}

\begin{rem}
Again, following the proof of the previous lemma, it is easy to check that
\begin{equation}
\label{eq:HH-infty}
|\delta_i^{+}(G_{t,N}(\mb{n})-G_{t,N}(\mb{m}))|\le C |\mb{n}-\mb{m}|\left(\frac{1}{t}+\frac{(|\mb{n}|+|\mb{m}|)^2}{t^2}\right)t^{-N/2}, \qquad \mb{n},\mb{m}\in \mathbb{Z}^N.
\end{equation}
\end{rem}

\begin{lem}
\label{lem:diff2-K}
  Let $N\ge 1$, $\mb{n},\mb{m}\in \mathbb{Z}^N$ such that $|\mb{n}|>2|\mb{n}-\mb{m}|$, and $0<2\sigma<N$.
  Then the inequality
  \begin{equation}
  \label{eq:diff2-K}
  |\delta_i^+ (K_{\sigma}(\mb{n})- K_\sigma(\mb{m}))|\le \frac{C|\mb{n}-\mb{m}|}{(|\mb{n}|+|\mb{m}|)^{N+2-2\sigma}},\qquad i=1,\dots,N,
  \end{equation}
  holds with a constant $C$ independent of $\mb{n}$ and $\mb{m}$.
\end{lem}

\begin{proof}
It is clear that
  \[
  |\delta_i^{+}(K_{\sigma}(\mb{n})-K_\sigma(\mb{m}))|\le \int_{0}^{\infty}|\delta_i^+(G_{t,N}(\mb{n})- G_{t,n}(\mb{m}))|t^{\sigma-1}\,dt=Q_1+Q_2,
  \]
  with
  \[
  Q_1=\int_{0}^{(|\mb{n}|+|\mb{m}|)^2}|\delta_i^+(G_{t,N}(\mb{n})-G_{t,n}(\mb{m}))|t^{\sigma-1}\,dt
  \]
  and
  \[
  Q_2=\int_{(|\mb{n}|+|\mb{m}|)^2}^{\infty}|\delta_i^+(G_{t,N}(\mb{n})-G_{t,n}(\mb{m}))|t^{\sigma-1}\,dt.
  \]

  To obtain the required bound for $Q_1$ we consider Lemma \ref{lem:diff2-G} and apply \eqref{eq:bound-I} with $\alpha=1/N$ and $\alpha=2/N$. Indeed,
  \begin{align*}
  Q_1&\le C |\mb{n}-\mb{m}|\left(
  \int_{0}^{(|\mb{n}|+|\mb{m}|)^2}\frac{1}{t}\left(G_{t,1}\left(\frac{|\mb{n}|+|\mb{m}|}{K}\right)\right)^Nt^{\sigma-1}\,dt\right.\\&\kern10pt \left.+
  (|\mb{n}|+|\mb{m}|)^2
  \int_{0}^{(|\mb{n}|+|\mb{m}|)^2}\frac{1}{t^2}\left(G_{t,1}\left(\frac{|\mb{n}|+|\mb{m}|}{K}\right)\right)^N t^{\sigma-1}\,dt\right)
\\&\le \frac{C|\mb{n}-\mb{m}|}{(|\mb{n}|+|\mb{m}|)^{N+2}}
  \int_{0}^{(|\mb{n}|+|\mb{m}|)^2}t^{\sigma-1}\,dt
  \le \frac{C|\mb{n}-\mb{m}|}{(|\mb{n}|+|\mb{m}|)^{N+2-2\sigma}}.
  \end{align*}
  The estimate for $Q_2$ can be deduced from \eqref{eq:HH-infty} in the following way
  \begin{align*}
  Q_2&\le  |\mb{n}-\mb{m}|\left(\int_{(|\mb{n}|+|\mb{m}|)^2}^{\infty}t^{\sigma-N/2-2}\,dt +
  (|\mb{n}|+|\mb{m}|)^2\int_{(|\mb{n}|+|\mb{m}|)^2}^{\infty}t^{\sigma-N/2-3}\,dt\right)\\
  &\le \frac{C |\mb{n}-\mb{m}|}{(|\mb{n}|+|\mb{m}|)^{N+2-2\sigma}}.\qedhere
  \end{align*}
\end{proof}

\begin{proof}[Proof of Theorem \ref{thm:Schauder}]
a) By using that
\[
(\delta_i^{+})^k(-\Delta_N)^{-\sigma}f(\mb{n})=(-\Delta_N)(\delta_i^{+})^k f(\mb{n}),
\]
it is enough to prove the result for $k=0$. Now, from Lemma \ref{lem:null}, it is verified that
\[
(-\Delta_N)^{-\sigma}f(\mb{n})-(-\Delta_N)^{-\sigma}f(\mb{m})=\sum_{\mb{k}\in \mathbb{Z}^N}(K_\sigma(\mb{n}-\mb{k})-K_\sigma(\mb{m}-\mb{k}))(f(\mb{k})-f(\mb{n})).
\]
Now, we split the sum on $0<|\mb{n}-\mb{k}|\le 2|\mb{n}-\mb{m}|$ and its complementary. By using \eqref{eq:bound-frac} we get
\begin{multline*}
\sum_{0<|\mb{n}-\mb{k}|\le 2|\mb{n}-\mb{m}|}K_\sigma(\mb{n}-\mb{k})|f(\mb{k})-f(\mb{n})|\\
\begin{aligned}
&\le C [f]_{C^{0,\alpha}(\mathbb{Z}^N)}\sum_{0<|\mb{n}-\mb{k}|\le 2|\mb{n}-\mb{m}|}\frac{|\mb{n}-\mb{k}|^\alpha}{|\mb{n}-\mb{k}|^{N-2\sigma}}\\
&
\le C [f]_{C^{0,\alpha}(\mathbb{Z}^N)}\sum_{0<|\mb{j}|\le 2|\mb{n}-\mb{m}|}|\mb{j}|^{\alpha+2\sigma-N}
\\&\le C [f]_{C^{0,\alpha}(\mathbb{Z}^N)}|\mb{n}-\mb{m}|^{\alpha+2\sigma}.
\end{aligned}
\end{multline*}
For the other term in the first part of the decomposition, by using that $|\mb{n}-\mb{k}|\le 2|\mb{n}-\mb{m}|$ implies $|\mb{m}-\mb{k}|\le 3|\mb{n}-\mb{m}|$ and \eqref{eq:bound-frac}, we have
\begin{multline*}
\sum_{0<|\mb{n}-\mb{k}|\le 2|\mb{n}-\mb{m}|}K_\sigma(\mb{m}-\mb{k})|f(\mb{k})-f(\mb{n})|
\\
\begin{aligned}
&\le C [f]_{C^{0,\alpha}(\mathbb{Z}^N)}\sum_{|\mb{m}-\mb{k}|\le 3|\mb{n}-\mb{m}|}\frac{|\mb{n}-\mb{k}|^\alpha}{(|\mb{m}-\mb{k}|+1)^{N-2\sigma}}\\
&\le C [f]_{C^{0,\alpha}(\mathbb{Z}^N)}\left(\sum_{|\mb{m}-\mb{k}|\le 3|\mb{n}-\mb{m}|}\frac{|\mb{n}-\mb{m}|^\alpha}{(|\mb{m}-\mb{k}|+1)^{N-2\sigma}}\right.\\&\kern20pt\left.+
\sum_{|\mb{m}-\mb{k}|\le 3|\mb{n}-\mb{m}|}(|\mb{m}-\mb{k}|+1)^{\alpha+2\sigma-N}\right)
\\&\le C [f]_{C^{0,\alpha}(\mathbb{Z}^N)}|\mb{n}-\mb{m}|^{\alpha+2\sigma}.
\end{aligned}
\end{multline*}
Finally, applying \eqref{eq:diff-K}, we have
\begin{multline*}
\sum_{|\mb{n}-\mb{k}|\ge 2|\mb{n}-\mb{m}|}|K_\sigma(\mb{n}-\mb{k})-K_\sigma(\mb{m}-\mb{k})||f(\mb{k})-f(\mb{n})|\\
\begin{aligned}
&\le C [f]_{C^{0,\alpha}(\mathbb{Z}^N)}|\mb{n}-\mb{m}|\sum_{|\mb{n}-\mb{k}|\ge 2|\mb{n}-\mb{m}|}\frac{|\mb{n}-\mb{k}|^\alpha}{|\mb{n}-\mb{k}|^{N+1-2\sigma}}\\&
\le C [f]_{C^{0,\alpha}(\mathbb{Z}^N)}|\mb{n}-\mb{m}|^{\alpha+2\sigma}
\end{aligned}
\end{multline*}
and we have concluded this part.

b) Again, we can reduce the proof to the case $k=0$. Now, from Lemma \ref{lem:null}, for any $i=1,\dots,n$ we have
\begin{multline*}
\delta_i^+ ((-\Delta_N)^{-\sigma}f(\mb{n})-(-\Delta_N)^{-\sigma}f(\mb{m}))\\=\sum_{k\in \mathbb{Z}^N}\delta_{i}^+(K_\sigma(\mb{n}-\mb{k})-K_\sigma(\mb{m}-\mb{k}))(f(\mb{k})-f(\mb{n}))).
\end{multline*}
Splitting the sum as in a), the proof follows the same steps but using in the region $|\mb{n}-\mb{k}|\le 2|\mb{n}-\mb{m}|$ the bounds (deduced from \eqref{eq:diff-K})
\[
|K_\sigma(\mb{n}+\mb{e_i}-\mb{k})-K_\sigma(\mb{n}-\mb{k})|\le C|\mb{n}-\mb{k}|^{2\sigma-N-1}, \qquad \mb{n},\mb{k}\in \mathbb{Z},
\]
and
\[
|K_\sigma(\mb{m}+\mb{e_i}-\mb{k})-K_\sigma(\mb{m}-\mb{k})|\le C|\mb{m}-\mb{k}|^{2\sigma-N-1}, \qquad \mb{m},\mb{k}\in \mathbb{Z},
\]
and in the region $|\mb{n}-\mb{k}|> 2|\mb{n}-\mb{m}|$ the estimate \eqref{eq:diff2-K}.

c) The proof of this part can be done as in a) but changing the bound $|f(\mb{n})-f(\mb{k})|\le C [f]_{C^{0,\alpha}(\mathbb{Z}^N)}|\mb{n}-\mb{k}|^\alpha$ by $|f(\mb{n})-f(\mb{k})|\le 2\|f\|_{\ell^{\infty}(\mathbb{Z}^N)}$.
\end{proof}

\subsection{The Riesz transforms}
As we have said at the beginning of this section the Riesz transform are well defined, 
because the kernel
\[
 \int_{0}^{\infty}\delta_i^+ G_{t,N}(\mb{n})t^{-1/2}\, dt
\]
is absolutely convergent.

Of course the Riesz transforms can be written for $N\ge 2$ as
\[
R_if(\mb{n})=\mathcal{R}_i \ast f(\mb{n}),
\]
with
\[
\mathcal{R}_i(\mb{n})=\frac{1}{\sqrt{\pi}}\int_{0}^{\infty}\delta_i^+ G_{t,N}(\mb{n})t^{-1/2}\, dt.
\]
%
To present our results about the Riesz transforms we need two preliminary lemmas.

\begin{lem}
  Let $N\ge 2$, then the inequalities
  \begin{equation}
  \label{eq:size-Riesz}
  |\mathcal{R}_i(\mb{n})|\le \frac{C}{(|\mb{n}|+1)^N}
  \end{equation}
  and
  \begin{equation}
  \label{eq:smooth-Riesz}
  |\mathcal{R}_{i}(\mb{n})-\mathcal{R}_{i}(\mb{m})|\le C\frac{|\mb{n}-\mb{m}|}{(|\mb{n}|+|\mb{m}|)^{N+1}},\qquad |\mb{n}|>2|\mb{n}-\mb{m}|,
  \end{equation}
  hold.
\end{lem}
\begin{proof}
To prove \eqref{eq:size-Riesz} we consider $\mb{n}\not=\mb{0}$ because the bound $|\mathcal{R}_i(\mb{0})|\le C$ is obvious. Now, from \eqref{eq:diff-I} and \eqref{eq:AM-GM-I}, it is deduced that
\[
|\delta_i G_{t,N}(\mb{n})|\le C\frac{|n_i|+1}{t}G_{t,N}(\mb{n})\le C\frac{|\mb{n}|}{t}\left(G_{t,1}\left(\frac{|\mb{n}|}{N}\right)\right)^N.
\]
Then, using \eqref{eq:bound-I} with $\alpha=1/N$ and $\alpha=-1/2$, we deduce that
\[
|\mathcal{R}_i(\mb{n})|\le C |\mb{n}|\left(\frac{1}{|\mb{n}|^{N+2}}\int_0^{|\mb{n}|^2}t^{-1/2}\, dt+\int_{|\mb{n}|^2}^{\infty}t^{-N/2-3/2}\, dt\right)\le \frac{C}{|\mb{n}|^N}.
\]

Now, we decompose the difference in \eqref{eq:smooth-Riesz} in the following way
\begin{multline*}
|\mathcal{R}_{i}(\mb{n})-\mathcal{R}_{i}(\mb{m})|\le \int_{0}^{(|\mb{n}|+|\mb{m}|)^2}|\delta_i(G_{t,N}(\mb{n})-G_{t,N}(\mb{m}))|t^{-1/2}\, dt\\+
\int_{(|\mb{n}|+|\mb{m}|)^2}^\infty |\delta_i(G_{t,N}(\mb{n})-G_{t,N}(\mb{m}))|t^{-1/2}\, dt:=L_1+L_2.
\end{multline*}
To analyze $L_1$ we apply \eqref{eq:diff2-G} and \eqref{eq:bound-I} with $\alpha=1/N$ and $\alpha=2/N$ to obtain that
\[
|\delta_i(G_{t,N}(\mb{n})-G_{t,N}(\mb{m}))|\le C\frac{|\mb{n}-\mb{m}|}{(|\mb{n}|+|\mb{m}|)^{N+2}}
\]
and
\[
L_1\le C \frac{|\mb{n}-\mb{m}|}{(|\mb{n}|+|\mb{m}|)^{N+2}}\int_{0}^{(|\mb{n}|+|\mb{m}|)^2}t^{-1/2}\, dt\le C\frac{|\mb{n}-\mb{m}|}{(|\mb{n}|+|\mb{m}|)^{N+1}}.
\]
From \eqref{eq:HH-infty}, we have
\begin{align*}
L_2&\le C |\mb{n}-\mb{m}|\left(\int_{(|\mb{n}|+|\mb{m}|)^2}^\infty t^{-N/2-3/2}\, dt+(|\mb{n}|+|\mb{m}|)^2\int_{(|\mb{n}|+|\mb{m}|)^2}^\infty t^{-N/2-5/2}\, dt\right)\\&\le C\frac{|\mb{n}-\mb{m}|}{(|\mb{n}|+|\mb{m}|)^{N+1}}
\end{align*}
and the proof is finished.
\end{proof}

\begin{lem}
\label{lem:sum-Riesz-null}
  For $N\ge 2$ it is verified that
\begin{equation}
\label{eq:sum-Riesz-null}
\sum_{\mb{k}\in \mathbb{Z}^N}\mathcal{R}_{i}(\mb{k})=0, \qquad i=1,\dots,N,
\end{equation}
and, for a fix value $\ell\in \mathbb{N}$,
\begin{equation}
\label{eq:sum-Riesz-partial}
\left|\sum_{|\mb{k}|\ge \ell}\mathcal{R}_{i}(\mb{k})\right|\le C, \qquad i=1,\dots,N.
\end{equation}
\end{lem}

\begin{proof}
To prove the result we focus on $i=1$ because the other cases are similar. We consider $M\in \mathbb{N}$ big enough and the partial sums
\[
S_M=\sum_{|\mb{k}|\le M}\mathcal{R}_{1}(\mb{k}).
\]
Now, we observe that
\begin{equation}
\label{eq:oddnes-Ri}
\mathcal{R}_1(-\mb{k})=-\mathcal{R}_1(\mb{k}-\mb{e}_1).
\end{equation}
Indeed, it is clear that
\begin{align*}
\delta_1^+ G_{t,N}(-\mb{k})&=(G_{t,1}(-k_1+1)-G_{t,1}(-k_1))\prod_{j=2}^{N}G_{t,1}(-k_j)
\\&=-(G_{t,1}(k_1)-G_{t,1}(k_1-1))\prod_{j=2}^{N}G_{t,1}(k_j)=-\delta_1^+ G_{t,N}(\mb{k}-\mb{e}_1)
\end{align*}
and \eqref{eq:oddnes-Ri} follows immediately. With \eqref{eq:oddnes-Ri} we have
\begin{align*}
2S_M&=\sum_{|\mb{k}|\le M}\mathcal{R}_1(\mb{k})+\sum_{|\mb{k}|\le M}\mathcal{R}_1(-\mb{k})
\\&=\sum_{|\mb{k}|\le M}\mathcal{R}_1(\mb{k})-\sum_{|\mb{k}|\le M}\mathcal{R}_1(\mb{k}-\mb{e}_1)
=\sum_{|\mb{k}|\le M}\mathcal{R}_1(\mb{k})-\sum_{|\mb{k}+\mb{e}_1|\le M}\mathcal{R}_1(\mb{k}).
\end{align*}
Then, applying \eqref{eq:size-Riesz} we get
\[
|S_M|\le C\sum_{M-1\le |\mb{k}|\le M+1}|\mathcal{R}_1(\mb{k})|\le \frac{C}{M}
\]
and
\[
\sum_{\mb{k}\in \mathbb{Z}^N}\mathcal{R}_{1}(\mb{k})=\lim_{M\to \infty}S_M=0
\]
proving \eqref{eq:sum-Riesz-null}.

To analyze \eqref{eq:sum-Riesz-partial}, we have to note that
\[
\left|\sum_{|\mb{k}|\ge \ell}\mathcal{R}_{i}(\mb{k})\right|=\left|\sum_{\mb{k}\in \mathbb{Z}^N}\mathcal{R}_i(\mb{k})-S_{\ell-1}\right|=|S_{\ell-1}|\le \frac{C}{\ell-1}
\]
and the result follows for $\ell>1$. The case $\ell=1$ is obvious.
\end{proof}

First, we present the behaviour of the Riesz transforms on the H\"older classes.

\begin{thm}
  Let $N\ge 2$, $0<\alpha<1/2$, and $f\in C^{0,\alpha}(\mathbb{Z}^N)$. Then for $1\le i \le N$ the inequality
  \[
  [R_i f]_{C^{0,\alpha}(\mathbb{Z}^N)}\le C [f]_{C^{0,\alpha}(\mathbb{Z}^N)}
  \]
  holds.
\end{thm}

\begin{proof}

By using Lemma \ref{lem:sum-Riesz-null}, for $\mb{n},\mb{m}\in \mathbb{Z}^N$ it is easy to check that
\begin{align*}
R_if(\mb{n})&=\sum_{k\in \mathbb{Z}^N}(f(\mb{k})-f(\mb{n}))\mathcal{R}_i(\mb{n}-\mb{k})\\
&=
\sum_{0<|\mb{n}-\mb{k}|\le 2 |\mb{n}-\mb{m}|}(f(\mb{k})-f(\mb{n}))\mathcal{R}_i(\mb{n}-\mb{k})
\\& \kern20pt+\sum_{|\mb{n}-\mb{k}|> 2 |\mb{n}-\mb{m}|}(f(\mb{k})-f(\mb{n}))\mathcal{R}_i(\mb{n}-\mb{k})\\&:=S_1(\mb{n})+S_2(\mb{n}).
\end{align*}
For the difference $S_1(\mb{n})-S_1(\mb{m})$, by using that $|\mb{n}-\mb{k}|\le 2 |\mb{n}-\mb{m}|$ implies $|\mb{m}-\mb{k}|\le 3|\mb{n}-\mb{m}|$ and the estimate \eqref{eq:size-Riesz}, it is verified that
\begin{multline*}
|S_1(\mb{n})-S_1(\mb{m})|\le |S_1(\mb{n})|+|S_1(\mb{m})|\\
\begin{aligned}
&\le C [f]_{C^{0,\alpha}(\mathbb{Z}^N)}\left(\sum_{0<|\mb{n}-\mb{k}|\le 2|\mb{n}-\mb{m}|}|\mb{k}-\mb{n}|^{\alpha-N}+\sum_{0<|\mb{m}-\mb{k}|\le 3|\mb{n}-\mb{m}|}|\mb{k}-\mb{m}|^{\alpha-N}\right)\\&
\le C [f]_{C^{0,\alpha}(\mathbb{Z}^N)}|\mb{n}-\mb{m}|^\alpha.
\end{aligned}
\end{multline*}

To treat $S_2(\mb{n})-S_2(\mb{m})$ we add the term $\pm f(\mb{m})\mathcal{R}_i(\mb{n}-\mb{k})$ and apply \eqref{eq:smooth-Riesz} and \eqref{eq:sum-Riesz-partial} to have
\begin{multline*}
|S_2(\mb{n})-S_2(\mb{m})|\\
\begin{aligned}
&\le \sum_{|\mb{n}-\mb{k}|
> 2|\mb{n}-\mb{m}|}|f(\mb{k})-f(\mb{m})||\mathcal{R}_i(\mb{n}-\mb{k})-\mathcal{R}_i(\mb{m}-\mb{k})|\\
&\kern15pt +
|f(\mb{n})-f(\mb{m})|\left|\sum_{|\mb{n}-\mb{k}|> 2|\mb{n}-\mb{m}|}\mathcal{R}_i(\mb{n}-\mb{k})\right|
\\&\le C [f]_{C^{0,\alpha}(\mathbb{Z}^N)}\left(|\mb{n}-\mb{m}| \sum_{|\mb{n}-\mb{k}|> 2|\mb{n}-\mb{m}|}|\mb{n}-\mb{k}|^{\alpha-N-1}+|\mb{n}-\mb{m}|^\alpha\right)\\&\le C [f]_{C^{0,\alpha}(\mathbb{Z}^N)}|\mb{n}-\mb{m}|^{\alpha}
\end{aligned}
\end{multline*}
and this concludes the proof.
\end{proof}

It would be possible to define another Riesz transform by mean
\[
\overline{R}_if(\mb{n})=\delta_i^{-}(-\Delta_N)^{-1/2}f(\mb{n}).
\]
The identity $\delta_i^{-}f(\mb{n})=\delta_i^{+}f(\mb{n}-\mb{e}_i)$ proves the following result.

\begin{cor}
\label{cor:Riesz-adjoint}
  Let $N\ge 2$, $0<\alpha<1/2$, and $f\in C^{0,\alpha}(\mathbb{Z}^N)$. Then for $1\le i \le N$ the inequality
  \[
  [\overline{R}_i f]_{C^{0,\alpha}(\mathbb{Z}^N)}\le C [f]_{C^{0,\alpha}(\mathbb{Z}^N)}
  \]
  holds.
\end{cor}

To finish this section we give a result about the boundedness with weights of the Riesz transforms.

\begin{thm}
  Let $N\ge 2$, $w\in A_p(\mathbb{Z}^N)$, and $1\le i\le N$. Then the inequalities
  \[
  \|R_if\|_{\ell^{p}(\mathbb{Z}^n,w)}\le C \|f\|_{\ell^{p}(\mathbb{Z}^n,w)}, \qquad 1<p<\infty, \qquad f\in \ell^2(\mathbb{Z}^N)\cap \ell^p(\mathbb{Z}^N,w),
  \]
  and
  \[
  \|R_if\|_{\ell^{1,\infty}(\mathbb{Z}^n,w)}\le C\|f\|_{\ell^{1}(\mathbb{Z}^n,w)}, \qquad \ell^2(\mathbb{Z}^N)\cap \ell^1(\mathbb{Z}^N,w),
  \]
  hold.
\end{thm}

\begin{proof}
We will prove that each $R_i$ is Calder\'on-Zygmund operator.

First, we check that $R_i$ is bounded from $\ell^2(\mathbb{Z}^N)$ into itself. 
For $f\in\ell^2(\mathbb{Z}^N)$, we have
\[
\mathcal{F}(\delta_i f)(x)=\sum_{\mb{n}\in \mathbb{Z}^N}f(\mb{n}+\mb{e}_i)e^{2\pi i \langle \mb{n},x\rangle}-
\sum_{\mb{n}\in \mathbb{Z}^N}f(\mb{n})e^{2\pi i \langle \mb{n},x\rangle}=(e^{-2\pi i x_i}-1)\mathcal{F}f(x)
\]
and
\begin{align*}
\mathcal{F}((-\Delta_N)^{-1/2}f)(x)&=\frac{1}{\sqrt{\pi}}\int_{0}^{\infty}e^{-4t(\sum_{k=1}^{N}\sin^2(\pi x_k))}\mathcal{F}f(x)t^{-1/2}\,dt\\&=\left(4\sum_{k=1}^{N}\sin^2(\pi x_k)\right)^{-1/2}\mathcal{F}f(x).
\end{align*}
With this we deduce that $\delta_i^+$ and $(-\Delta_N)^{-1/2}$ can be seen as multipliers of the functions $e^{-2\pi i x_i}-1$ and $\left(4\sum_{k=1}^{N}\sin^2(\pi x_k)\right)^{-1/2}$, respectively.
Then, the operator $R_{i}$ is associated with the multiplier
\[
\frac{e^{-2\pi i x_i}-1}{\left(4\sum_{k=1}^{N}\sin^2(\pi x_k)\right)^{1/2}};
\]
i.e.,
\begin{align*}
\mathcal{F}(R_i f)(x)&=\frac{e^{-2\pi i x_i}-1}{\left(4\sum_{k=1}^{N}\sin^2(\pi x_k)\right)^{1/2}}\mathcal{F}f\\&=\frac{-ie^{\pi i x_i}\sin(\pi x_i)}{\left(\sum_{k=1}^{N}\sin^2(\pi x_k)\right)^{1/2}}\mathcal{F}f, \qquad i=1,\dots,N.
\end{align*}
In this way, using the bound
\[
\frac{|\sin(\pi x_i)|}{\left(\sum_{k=1}^{N}\sin^2(\pi x_k)\right)^{1/2}}\le 1,
\]
by Plancherel identity we have
\[
\|R_i f\|_{\ell^2(\mathbb{Z}^N)}=\|\mathcal{F}(R_i f)\|_{L^2\left([-1/2,1/2]^N\right)}\le C
\|\mathcal{F}f\|_{L^2\left([-1/2,1/2]^N\right)}=\|f\|_{\ell^2(\mathbb{Z}^N)}.
\]
Finally, the estimates \eqref{eq:size-CZ} and \eqref{eq:smooth-CZ} can be deduced from \eqref{eq:size-Riesz} and \eqref{eq:smooth-Riesz} and we have concluded.
\end{proof}

For the Riesz transforms $\overline{R}_i$ we have the next result.
\begin{cor}
  Let $N\ge 2$, $w\in A_p(\mathbb{Z}^N)$, and $1\le i \le N$. Then the inequalities
   \[
  \|\overline{R}_if\|_{\ell^{p}(\mathbb{Z}^n,w)}\le C \|f\|_{\ell^{p}(\mathbb{Z}^n,w)}, \qquad 1<p<\infty, \qquad f\in \ell^2(\mathbb{Z}^N)\cap \ell^p(\mathbb{Z}^N,w),
  \]
  and
  \[
  \|\overline{R}_if\|_{\ell^{1,\infty}(\mathbb{Z}^n,w)}\le C\|f\|_{\ell^{1}(\mathbb{Z}^n,w)}, \qquad f\in \ell^2(\mathbb{Z}^N)\cap \ell^1(\mathbb{Z}^N,w),
  \]
  hold.
\end{cor}

\section{Fractional powers of the Laplacian}
\label{sec:frac}
From the identity
\[
\lambda^s=\frac{1}{\Gamma(-s)}\int_{0}^{\infty}(e^{-\lambda t}-1)\frac{dt}{t^{s+1}}, \qquad 0<s<1,
\]
we define the fractional powers (in fact, positive powers) of $-\Delta_N$ by
\begin{equation}
\label{eq:def-frac}
(-\Delta_N)^s f(\mb{n})=\frac{1}{\Gamma(-s)}\int_{0}^{\infty}(W_tf(\mb{n})-f(\mb{n}))\frac{dt}{t^{s+1}}, \qquad f\in \ell^2(\mathbb{Z}^N).
\end{equation}
By using \eqref{eq:sum-L1-multi}, it is clear that
\begin{equation}
\label{eq:Wt-I}
W_tf(\mb{n})-f(\mb{n})=\sum_{\mb{k}\in \mathbb{Z}^N}G_{t,N}(\mb{n}-\mb{k})(f(\mb{k})-f(\mb{n})).
\end{equation}
Moreover, we can prove that the operators $(-\Delta_N)^{s}$ are well defined in $\ell^2(\mathbb{Z}^N)$ by \eqref{eq:def-frac}. Indeed, we consider the decomposition
\[
\int_{0}^{\infty}|W_tf(\mb{n})-f(\mb{n})|\frac{dt}{t^{s+1}}\le F_1+F_2
\]
where
\[
F_1=\int_{0}^{1}|W_tf(\mb{n})-f(\mb{n})|\frac{dt}{t^{s+1}}
\]
and $F_2$ is the integral on $[1,\infty)$ of the same function. For $F_1$, applying \eqref{eq:cota-l2}, we have
\[
F_1\le C \|f\|_{\ell^2(\mathbb{Z}^N)}\int_{0}^{1}t^{-s}\, dt\le C \|f\|_{\ell^2(\mathbb{Z}^N)}.
\]
Now, from Theorem \ref{thm:bound-heat} with $p=\infty$ and $q=2$, we deduce that
\[
|W_tf(\mb{n})-f(\mb{n})|\le C(t^{-N/4}+1)\|f\|_{\ell^2(\mathbb{Z}^N)}
\]
and
\[
F_2\le C \|f\|_{\ell^2(\mathbb{Z}^N)}\int_{1}^{\infty}(t^{-N/4}+1)t^{-s-1}\, dt\le C \|f\|_{\ell^2(\mathbb{Z}^N)}.
\]

In this section we will focus on the study of some aspects of the operators $(-\Delta_N)^s$. Firstly, we will show some basic properties of the fractional powers of the Laplacian, including a maximum principle. In the last part of the section we will deal with the mapping properties of $(-\Delta_N)^s$ on the H\"older classes. 

\subsection{Basic properties of the fractional powers of the Laplacian}
In this section we present some elementary properties of the fractional powers $(-\Delta_N)^s$. They are collected in the next theorem.

\begin{thm}
  \label{thm:frac-elem}
  Let $N\ge 1$ and $0<s<1$.
  \begin{enumerate}
  \item[i)] For $f\in \ell_{-s}(\mathbb{Z}^N)$, it is verified that
  \[
  (-\Delta_N)^s f(\mb{n})=\sum_{\begin{smallmatrix}
                                  \mb{k}\in \mathbb{Z}^N \\
                                  \mb{k}\not= \mb{n}
                                \end{smallmatrix}}\mathcal{K}_s(\mb{n}-\mb{k})(f(\mb{n})-f(\mb{k})),
  \]
  where
  \begin{equation}\label{eq:kernel-frac}
    \mathcal{K}_s(\mb{n})=\frac{1}{|\Gamma(-s)|}\int_{0}^{\infty}G_{t,N}(\mb{n})\frac{dt}{t^{s+1}}, \qquad \mb{n}\not= \mb{0},
  \end{equation}
and $\mathcal{K}_s(\mb{0})=0$.
\item[ii)] For $\mb{n}\not= \mb{0}$, the inequalities
\begin{equation}\label{eq:bound-kernel-frac}
0<\mathcal{K}_s(\mb{n})\le \frac{C}{|\mb{n}|^{N+2s}}\frac{1}{|\Gamma(-s)|}\left(\frac{1}{1-s}+\frac{2}{N+2s}\right)
\end{equation}
hold with a constant $C$ independent of $s$.

\item[iii)] For $f\in \ell_0(\mathbb{Z}^N)$, it is verified that
\[
\lim_{s\to 0^+}(-\Delta_N)^{s}f(\mb{n})=f(\mb{n}).
\]

\item[iv)] For $f\in \ell^\infty(\mathbb{Z}^N)$, it is verified that
\[
\lim_{s\to 1^{-}}(-\Delta_N)^{s}f(\mb{n})=-\Delta_N f(\mb{n}).
\]
  \end{enumerate}
\end{thm}

\begin{proof}
First, we observe that $W_t f$ is well defined on the spaces $\ell_{-s}(\mathbb{Z}^N)$ for $0\le s \le 1$. Effectively, for $M>0$ big enough, with $t$ fix, we have
\[
\sum_{|\mb{k}|>M} G_{t,N}(\mb{k})f(\mb{n}-\mb{k})
\le C \|f\|_{\ell_{-s}(\mathbb{Z}^N)}\sup_{|\mb{k}|>M}G_{t,N}(\mb{k})(1+|\mb{n}-\mb{k}|)^{N+2s}.
\]
From \eqref{eq:AM-GM-I} and the asymptotic expansion \eqref{eq:asym-I}, we deduce that
\begin{multline*}
\sup_{|\mb{k}|>M}G_{t,N}(\mb{k})(1+|\mb{n}-\mb{k}|)^{N+2s}\\\le C
e^{-2Nt}N^{N/2}\sup_{{|\mb{k}|>M}}
\frac{(Net)^{|\mb{k}|}(1+|\mb{n}|+|\mb{k}|)^{N+2s}}{|\mb{k}|^{|\mb{k}|+N/2}}
\le C_{t,M,|\mb{n}|,s,N}
\end{multline*}
and we conclude.

To prove i) we use \eqref{eq:Wt-I} to get
\begin{align*}
(-\Delta_N)^s f(\mb{n})&=\frac{1}{\Gamma(-s)}\int_{0}^{\infty} \sum_{\mb{k}\not= \mb{n}}G_{t,N}(\mb{n}-\mb{k})(f(\mb{k})-f(\mb{n}))\frac{dt}{t^{s+1}}
\\&=\sum_{\mb{k}\not= \mb{n}}\mathcal{K}_s(\mb{n}-\mb{k})(f(\mb{n})-f(\mb{k}))
\end{align*}
where $\mathcal{K}_s$ is given by \eqref{eq:kernel-frac}. The proof will be completed justifying the application of Fubini theorem in the second identity. To do this we need the upper bound in ii), so we proceed with its proof. 

From \eqref{eq:AM-GM-I} and \eqref{eq:bound-I} with $\alpha=1/N$, we obtain that
\[
G_{t,N}(\mb{n})t^{-1}\le \frac{C}{|\mb{n}|^{N+2}}.
\]
Then, with \eqref{eq:size-G-t}, we deduce that
\begin{align*}
\mathcal{K}_s(\mb{n})&\le \frac{C}{|\Gamma(-s)|} \left(\frac{1}{|\mb{n}|^{N+2}}\int_{0}^{|\mb{n}|^2}t^{-s}\, dt+\int_{|\mb{n}|^2}^{\infty}t^{-N/2-s-1}\, dt\right)\\&= \frac{C}{|\mb{n}|^{N+2s}}\frac{1}{|\Gamma(-s)|}\left(\frac{1}{1-s}+\frac{2}{N+2s}\right).
\end{align*}
The positivity of the kernel $\mathcal{K}_s$ is clear. 

Let us justify the interchange of the sum and the integral. First, we apply \eqref{eq:bound-kernel-frac} and Tonelli's theorem in the term
\[
\int_{0}^{\infty}\sum_{\mb{k}\not= \mb{n}}G_{t,N}(\mb{n}-\mb{k})|f(\mb{k})|\frac{dt}{t^{s+1}}
\]
to have that it is bounded by
\[
\sum_{\mb{k}\not = \mb{n}}\frac{|f(\mb{k})|}{|\mb{n}-\mb{k}|^{N+2s}}
\]
and this term is finite for each $\mb{n}$ because $f\in \ell_{-s}(\mathbb{Z}^N)$. Similarly for the term
\[
|f(\mb{n})|
\int_{0}^{\infty}\sum_{\mb{k}\not= \mb{n}}G_{t,N}(\mb{n}-\mb{k})\frac{dt}{t^{s+1}},
\]
we obtain the bound
\[
|f(\mb{n})|\sum_{\mb{k}\not= \mb{n}}\frac{1}{|\mb{n}-\mb{k}|^{N+2s}}
\]
and this last sum is finite.

To prove iii) we check that
\[
\lim_{s\to 0^+}N_1=1\qquad\text{ and }\qquad \lim_{s\to 0^+}N_2=0
\]
where
\[
N_1=\sum_{\mb{k}\not= \mb{n}}\mathcal{K}_s(\mb{n}-\mb{k})
\qquad
\text{ and }
\qquad
N_2=\sum_{\mb{k}\not= \mb{n}}\mathcal{K}_s(\mb{n}-\mb{k})|f(\mb{k})|.
\]
We start with $N_2$. From \eqref{eq:bound-kernel-frac} and \eqref{eq:size-G-t}, we have
\[
N_2\le \frac{C}{|\Gamma(-s)|}\left(\frac{1}{1-s}+\frac{2}{N+2s}\right)
\sum_{\mb{k}\not= \mb{n}}\frac{|f(\mb{k})|}{|\mb{n}-\mb{k}|^{N+2s}}.
\]
Now, using that $f\in \ell_0(\mathbb{Z}^N)$, applying dominated convergence, we deduce that for each $\mb{n}$ it is verified that 
\[
\sum_{\mb{k}\not= \mb{n}}\frac{|f(\mb{k})|}{|\mb{n}-\mb{k}|^{N+2s}}\le C\|f\|_{\ell_0(\mathbb{Z}^N)}
\]
and then $\lim_{s\to 0^+}N_2=0$. To treat $N_1$, we consider the decomposition $N_1=N_{1,1}+N_{1,2}$ where
\[
N_{1,1}=\frac{1}{|\Gamma(-s)|}\sum_{\mb{k}\not= \mb{n}}\int_{0}^{1}G_{t,N}(\mb{n}-\mb{k})\frac{dt}{t^{s+1}}
\]
and $N_{1,2}$ is the same sum but with the integral on the interval $[1,\infty)$. To analyze $N_{1,1}$ we apply \eqref{eq:AM-GM-I} and \eqref{eq:bound-I} with $\alpha=1/(2N)$ and we consider $0<s<1/2$. In this way,
\[
N_{1,1}\le \frac{C}{|\Gamma(-s)|}\sum_{\mb{k}\not= \mb{n}}\frac{1}{|\mb{n}-\mb{k}|^{N+1}}\int_{0}^{1}t^{-1/2-s}\,dt\le
                                \frac{C}{|\Gamma(-s)|(1-2s)}
\]
and $\lim_{s\to 0^+}N_{1,1}=0$. Now, using that
\[
\sum_{\mb{k}\not= \mb{n}}G_{t,N}(\mb{n}-\mb{k})=1-G_{t,N}(\mb{0}),
\]
we obtain that
\[
N_{1,2}=\frac{1}{|\Gamma(-s)|}\left(\frac{1}{s}-\int_{1}^{\infty}G_{t,N}(\mb{0})\frac{dt}{t^{s+1}}\right).
\]
With the bound \eqref{eq:size-G-t} we get
\[
\int_{1}^{\infty}G_{t,N}(\mb{0})\frac{dt}{t^{s+1}}\le C \int_{1}^{\infty}t^{-N/2-s-1}\, dt=\frac{C}{N+2s}
\]
and, with the identity $|\Gamma(-s)|s=\Gamma(1-s)$ we conclude that $\lim_{s\to 0^+}N_{1,2}=1$.

To prove iv) we start observing
\begin{multline*}
(-\Delta_N)^{s}f(\mb{n})=\sum_{k=1}^{N}(\mathcal{K}_s(\mb{e}_k))(f(\mb{n})-f(\mb{n}-\mb{e}_k))
+\mathcal{K}_s(-\mb{e}_k))(f(\mb{n})-f(\mb{n}+\mb{e}_k)))\\+\sum_{\begin{smallmatrix}
                                 \mb{k}\in \mathbb{Z}^N \\
                                 \mb{k}\notin \{\mb{0},\pm \mb{e}_1,\dots,\pm \mb{e}_N\}
                               \end{smallmatrix}}\mathcal{K}_s(\mb{k})(f(\mb{n})-f(\mb{n}-\mb{k})):=M_1+M_2.
\end{multline*}
Taking into account that $\mathcal{K}_s(-\mb{n})=\mathcal{K}_s(\mb{n})$, we have
\[
M_1=-\sum_{k=1}^{N}\mathcal{K}_s(\mb{e}_k)\Delta_{N,i}f(\mb{n}).
\]
Now, using that for any continuous and bounded function on $[0,\infty)$,
\[
\lim_{s\to 1^{-}}\frac{1}{\Gamma(1-s)}\int_{0}^{\infty}e^{-t}f(t)\frac{dt}{t^s}=f(0)
\]
(the identity is clear for polynomials and with the hypotheses on $f$ the general result is obtained), taking the function
\[
f(z)=\frac{2I_1(z)}{z}(I_0(z))^{N-1},
\]
we deduce that
\begin{align*}
\lim_{s\to 1^{-}}\mathcal{K}_s(\mb{e}_i)&=\lim_{s\to 1^{-}}\frac{1}{\Gamma(1-s)}\int_{0}^{\infty}e^{-2Nt}\frac{I_1(2t)}{t}(I_0(2t))^{N-1}\frac{dt}{t^s}
\\&=\lim_{s\to 1^{-}}\frac{(2N)^{s-1}}{\Gamma(1-s)}\int_{0}^{\infty}e^{-w}f\left(\frac{w}{N}\right)\frac{dw}{w^s}=f(0)=1.
\end{align*}
Then,
\[
\lim_{s\to 1^{-}}M_1=-\Delta_N f(\mb{n}).
\]
Finally, let us check that $\lim_{s\to 1^{-}}M_2=0$. By using that $G_{t,N}(\mb{n})t^{-3/2}\le C |\mb{n}|^{-N-3}$ (which can be deduced from \eqref{eq:AM-GM-I} and \eqref{eq:bound-I} with $\alpha=3/(2N)$) and proceeding as in the proof of \eqref{eq:bound-kernel-frac}, we can obtain the bound
\[
\mathcal{K}_s(\mb{n})\le \frac{C}{|\mb{n}|^{N+2s}}\frac{1}{|\Gamma(-s)|}\left(\frac{1}{3-2s}+\frac{1}{N+2s}\right).
\]
Then
\[
M_2\le \frac{C}{|\Gamma(-s)|}\|f\|_{\ell^\infty(\mathbb{Z}^N)}\sum_{\mb{k}\not= \mb{0}}\frac{1}{|\mb{k}|^{N+2s}}\le  \frac{C}{|\Gamma(-s)|}\|f\|_{\ell^\infty(\mathbb{Z}^N)}
\]
and $\lim_{s\to 1^{-}}M_2=0$.
\end{proof}

To conclude this section, we present maximum and comparison principles for the fractional powers which are obvious consequences of the positivity of the kernel $\mathcal{K}_s$.

\begin{thm}
  Let $N\ge 1$ and $0<s<1$.
  \begin{enumerate}
  \item[i)] (Maximum principle) Let $f\in \ell^2(\mathbb{Z}^N)$ be such that $f \ge 0$ and $f(\mb{n}_0) = 0$ for some $\mb{n}_0\in \mathbb{Z}^N$. Then
\[
(-\Delta_N)^sf(\mb{n}_0)\le 0.
\]
Moreover, $(-\Delta_N)^sf(\mb{n}_0)= 0$ only if $f(\mb{n}) = 0$ for all $\mb{n} \in \mathbb{Z}^N$.
  \item[ii)] (Comparison principle) Let $f,g\in \ell^2(\mathbb{Z}^N)$ be such that $f \ge g$ and $f(\mb{n}_0) = g(\mb{n}_0)$ for some $\mb{n}_0\in \mathbb{Z}^N$. Then
\[
(-\Delta_N)^sf(\mb{n}_0)\le (-\Delta_N)^sg(\mb{n}_0).
\]
Moreover, $(-\Delta_N)^sf(\mb{n}_0)= (-\Delta_N)^sg(\mb{n}_0)$ only if $f(\mb{n}) = g(\mb{n})$ for all $\mb{n}\in\mathbb{Z}^N$.
  \end{enumerate}
\end{thm}

\subsection{The fractional powers of the Laplacian on the H\"older classes}
The target of this subsection is the analysis of the behaviour of $(-\Delta_N)^s$ in the H\"older classes and it is contained in the next result.

\begin{thm}
  Let $N\ge 1$, $k\ge 0$, $0<\alpha\le 1$, $0<s<1$, and $f\in \ell_{-s}(\mathbb{Z}^N)$.
  \begin{enumerate}
  \item[i)] If $f\in C^{k,\alpha}(\mathbb{Z}^N)$ and $2s<\alpha$ then $(-\Delta_N)^{s}f\in C^{k,\alpha-2s}(\mathbb{Z}^N)$ and
      \[
      [(-\Delta_N)^{s}f]_{C^{k,\alpha-2s}(\mathbb{Z}^N)}\le C [f]_{C^{k,\alpha}(\mathbb{Z}^N)}.
      \]
 \item[ii)] If $f\in C^{k+1,\alpha}(\mathbb{Z}^N)$ and $\alpha<2s<\alpha+1$ then $(-\Delta_N)^{s}f\in C^{k,\alpha-2s+1}(\mathbb{Z}^N)$ and
      \[
      [(-\Delta_N)^{s}f]_{C^{k,\alpha-2s+1}(\mathbb{Z}^N)}\le C [f]_{C^{k+1,\alpha}(\mathbb{Z}^N)}.
      \]
  \end{enumerate}
\end{thm}

\begin{proof}
We will prove both results for $k=0$ because the other cases can be deduced by using that $\delta_i^{+}(-\Delta_N)^s=(-\Delta_N)^s\delta_i^{+}$.

To prove i), we consider the identity
\[
|(-\Delta_N)^{s}f(\mb{n})-(-\Delta_N)^{s}f(\mb{m})|=|B_1+B_2|,
\]
where
\[
B_1=\sum_{\begin{smallmatrix}
            \mb{k}\in \mathbb{Z}^N \\
            1\le |\mb{k}|\le |\mb{n}-\mb{m}|
          \end{smallmatrix}}
          (f(\mb{n})-f(\mb{n}+\mb{k})-f(\mb{m})+f(\mb{m}+\mb{k}))\mathcal{K}_s(\mb{k})
\]
and $B_2$ is the sum on $|\mb{k}|> |\mb{n}-\mb{m}|$. First, from the estimate
\begin{align*}
 |f(\mb{n})-f(\mb{n}+\mb{k})-f(\mb{m})+f(\mb{m}+\mb{k})|&\le |f(\mb{n})-f(\mb{n}+\mb{k})|+|f(\mb{m})-f(\mb{m}+\mb{k})|
\\& \le C\|f\|_{C^{0,\alpha}(\mathbb{Z}^N)}|\mb{k}|^{\alpha}
 \end{align*}
and \eqref{eq:bound-kernel-frac}, we get
\[
B_1\le C [f]_{C^{0,\alpha}(\mathbb{Z}^N)}\sum_{\begin{smallmatrix}
            \mb{k}\in \mathbb{Z}^N \\
            1\le |\mb{k}|\le |\mb{n}-\mb{m}|
          \end{smallmatrix}}|\mb{k}|^{\alpha-N-2s}\le C [f]_{C^{0,\alpha}(\mathbb{Z}^N)}|\mb{n}-\mb{m}|^{\alpha-2s}
\]
and this is appropriate for our purpose. To treat $B_2$, we use that
 \begin{align*}
 |f(\mb{n})-f(\mb{n}+\mb{k})-f(\mb{m})+f(\mb{m}+\mb{k})|&\le |f(\mb{n})-f(\mb{m})|+|f(\mb{n}+\mb{k})-f(\mb{m}+\mb{k})|
\\& \le C\|f\|_{C^{0,\alpha}(\mathbb{Z}^N)}|\mb{n}-\mb{m}|^{\alpha}
 \end{align*}
 and the bound \eqref{eq:bound-kernel-frac}, to get
 \[
 B_2\le [f]_{C^{0,\alpha}(\mathbb{Z}^N)}|\mb{n}-\mb{m}|^{\alpha}\sum_{\begin{smallmatrix}
            \mb{k}\in \mathbb{Z}^N \\
            |\mb{k}|> |\mb{n}-\mb{m}|
          \end{smallmatrix}}|\mb{k}|^{-N-2s}\le C [f]_{C^{0,\alpha}(\mathbb{Z}^N)}|\mb{n}-\mb{m}|^{\alpha-2s}
 \]
 and the proof of i) is completed.

 We consider $N\ge 2$, the case $N=1$ can be found in \cite[Theorem 1.5]{CRSTV}. We consider the identity
 \begin{align*}
 (-\Delta_N)^{s}&=(-\Delta_N)^{s-1/2}(-\Delta_N)^{-1/2}(-\Delta_N)
 =(-\Delta_N)^{s-1/2}(-\Delta_N)^{-1/2}\sum_{i=1}^{N}\delta_i^{-}\delta_i^{+}
 \\&=(-\Delta_N)^{s-1/2}\sum_{i=1}^{N}\overline{R}_i\delta_i^{+}
 \end{align*}
 and we distinguish two cases. For $s-1/2>0$ it is verified that $0<\alpha-2(s-1/2)<1$ and, applying i) and Corollary \ref{cor:Riesz-adjoint}, we deduce that
 \begin{align*}
 [(-\Delta_N)^sf]_{C^{0,\alpha-2s+1}(\mathbb{Z}^N)}&\le C \sum_{i=1}^{N}[\overline{R}_i\delta_i^{+}f]_{C^{0,\alpha}(\mathbb{Z}^N)}
\\&\le C \sum_{i=1}^{N}[\delta_i^{+}f]_{C^{0,\alpha}(\mathbb{Z}^N)}\le C [f]_{C^{1,\alpha}(\mathbb{Z}^N)}.
 \end{align*}
 In the case $s-1/2<0$ we have $0<\alpha+2(-s+1/2)$ and we can apply a) in Theorem \ref{thm:Schauder} and Corollary \ref{cor:Riesz-adjoint} to conclude
 \begin{align*}
 [(-\Delta_N)^sf]_{C^{0,\alpha-2s+1}(\mathbb{Z}^N)}&
 =[(-\Delta_N)^{-(-s+1/2)}\sum_{i=1}^{N}\overline{R}_i\delta_i^{+}f]_{C^{0,\alpha-2s+1}(\mathbb{Z}^N)}\\&\le C \sum_{i=1}^{N}[\overline{R}_i\delta_i^{+}f]_{C^{0,\alpha}(\mathbb{Z}^N)}
\\&\le C \sum_{i=1}^{N}[\delta_i^{+}f]_{C^{0,\alpha}(\mathbb{Z}^N)}\le C [f]_{C^{1,\alpha}(\mathbb{Z}^N)}.
 \end{align*}
 For $s=1/2$ the result is a direct consequence of Corollary \ref{cor:Riesz-adjoint}.
\end{proof}

\section{The discrete square functions}
\label{sec:gk}
In this section we will prove weighted inequalities for the discrete Littlewood-Paley-Stein $g_k$-square functions associated to the heat semigroup and for the $\mathfrak{g}_k$-square functions associated to the Poisson semigroup. Moreover, we conclude with a result about Laplace type multipliers.

\subsection{The $g_k$-square functions associated to the heat semigroup}
The discrete $g_k$-square functions associated to the heat semigroup are defined by
\[
g_k(f)(\mb{n})=\left(\int_{0}^{\infty}t^{2k-1}\left|\frac{\partial^k}{\partial t^k}W_tf(\mb{n})\right|^2\, dt\right)^{1/2}, \qquad k=1,2,\dots.
\]
Note that taking the Banach space $\mathbb{B}_k=L^2((0,\infty),t^{2k-1}\, dt)$, we have
\[
g_kf(\mb{n})=\|P_{t,k}f(\mb{n})\|_{\mathbb{B}_k}
\]
with
\[
P_{t,k}f(\mb{n})=\sum_{\mb{k}\in \mathbb{Z}^N}f(\mb{k})\mathcal{P}_{t,k}(\mb{n}-\mb{k})
\]
and
\begin{align*}
\mathcal{P}_{t,k}(\mb{n})&=\frac{d^k}{dt^k}G_{t,N}(\mb{n})=
\frac{d^k}{dt^k}\int_{[-1/2,1/2]^N}e^{-4t\sum_{i=1}^{N}\sin^2(\pi x_i)}e^{-2\pi i\langle x, \mb{n}\rangle}\, dt
\\&=\int_{[-1/2,1/2]^N}F_{t,k}(x)e^{-2\pi i\langle x, \mb{n}\rangle}\, dt,
\end{align*}
being
\[
F_{t,k}(x)=(-4)^k \left(\sum_{i=1}^{N}\sin^2(\pi x_i)\right)^k e^{-4t\sum_{i=1}^{N}\sin^2(\pi x_i)}.
\]

Our main result about the $g_k$-square functions is the following one.
\begin{thm}
\label{thm:equiv-gk-Heat}
  Let $N\ge 1$, $1<p<\infty$, $w\in A_p(\mathbb{Z}^N)$, and $k\in \mathbb{N}$. Then the inequalities
  \begin{equation}
  \label{eq:equiv-gk-Heat}
  C_1\|f\|_{\ell^p(\mathbb{Z}^N,w)}\le \|g_k(f)\|_{\ell^p(\mathbb{Z}^N,w)}\le C_2 \|f\|_{\ell^p(\mathbb{Z}^N,w)}, \qquad f\in \ell^2(\mathbb{Z}^N)\cap \ell^{p}(\mathbb{Z}^N,w),
  \end{equation}
  hold.
\end{thm}

To obtain the proof of the equivalence in the previous theorem, we start proving the boundedness of the $g_k$-square functions from $\ell^2(\mathbb{Z}^N)$ into itself.

\begin{lem}
\label{lem:L2-gk}
  For $N\ge 1$ and $k\in \mathbb{N}$, it is verified that
  \begin{equation}
  \label{eq:L2-gk}
  \|g_k(f)\|_{\ell^2(\mathbb{Z}^N)}^2=\frac{\Gamma(2k)}{2^{2k}}\|f\|_{\ell^2(\mathbb{Z}^N)}^2.
  \end{equation}
\end{lem}

\begin{proof}
For each sequence in $\ell^2(\mathbb{Z}^N)$, we have
\[
P_{t,k}f=\mathcal{F}^{-1}(F_{t,k}\mathcal{F}f).
\]
Then
\[
\|g_k(f)\|_{\ell^2(\mathbb{Z}^N)}^2=\sum_{\mb{k}\in \mathbb{Z}^N}\|P_{t,k}f(\mb{k})\|_{\mathbb{B}_k}^2=
\left\|\sum_{\mb{k}\in \mathbb{Z}^N}(P_{t,k}f(\mb{k}))^2\right\|_{\mathbb{B}_k}^2.
\]
By applying Parseval's identity,
\[
\sum_{\mb{k}\in \mathbb{Z}^N}(P_{t,k}f(\mb{k}))^2=\int_{[-1/2,1/2]^N}(F_{t,k}(x))^2(\mathcal{F}f(x))^2\, dx
\]
and
\begin{align*}
\|g_k(f)\|_{\ell^2(\mathbb{Z}^N)}^2&=\int_{0}^{\infty}t^{2k-1}\int_{[-1/2,1/2]^N}(F_{t,k}(x))^2(\mathcal{F}f(x))^2\, dx\, dt\\&=\int_{[-1/2,1/2]^N}(\mathcal{F}f(x))^2\int_{0}^{\infty}t^{2k-1}(F_{t,k}(x))^2\, dt\, dx
\\&=\frac{\Gamma(2k)}{2^{2k}}\int_{[-1/2,1/2]^N}(\mathcal{F}f(x))^2\, dx=\frac{\Gamma(2k)}{2^{2k}}\|f\|_{\ell^2(\mathbb{Z}^N)}^2.\qedhere
\end{align*}
\end{proof}

Now, we proceed with two reductions to simplify the proof of Theorem \ref{thm:equiv-gk-Heat}.

\begin{lem}
Let $N\ge 1$, $1<p<\infty$, $w\in A_p(\mathbb{Z}^N)$, and $k\in \mathbb{N}$. Then the inequality
\[
\|g_k(f)\|_{\ell^p(\mathbb{Z}^N,w)}\le C \|f\|_{\ell^p(\mathbb{Z}^N,w)}, \qquad f\in \ell^2(\mathbb{Z}^N)\cap \ell^p(\mathbb{Z}^N, w)
\]
implies
\[
\|f\|_{\ell^p(\mathbb{Z}^N,w)}\le C \|g_k(f)\|_{\ell^p(\mathbb{Z}^N,w)},\qquad f\in \ell^2(\mathbb{Z}^N)\cap \ell^p(\mathbb{Z}^N).
\]
\end{lem}

\begin{proof}
Polarising the identity \eqref{eq:L2-gk}, we have
\[
\sum_{\mb{k}\in \mathbb{Z}^N}f(\mb{k})h(\mb{k})=\frac{2^{2k}}{\Gamma(2k)}\sum_{\mb{k}\in\mathbb{Z}^N}
\int_{0}^{\infty}t^{2k-1}\left(\frac{d^k}{dt^k}W_tf(\mb{k})\right)\left(\frac{d^k}{dt^k}W_th(\mb{k})\right)\, dt
\]
and
\[
\left|\sum_{\mb{k}\in \mathbb{Z}^N}f(\mb{k})h(\mb{k})\right|\le C \sum_{\mb{k}\in \mathbb{Z}^N}g_k(f)(\mb{k})g_k(h)(\mb{k}).
\]
In this way, taking $h(\mb{k})=w^{1/p}(\mb{k})f_1(\mb{k})$ and applying a duality argument the result follows.
\end{proof}

\begin{lem}
Let $N\ge 1$, $1<p<\infty$, and $w\in A_p(\mathbb{Z}^N)$. Then the inequality
\[
\|g_1(f)\|_{\ell^p(\mathbb{Z}^N,w)}\le C \|f\|_{\ell^p(\mathbb{Z}^N,w)},\qquad f\in \ell^2(\mathbb{Z}^N)\cap \ell^p(\mathbb{Z}^N),
\]
implies
\[
\|g_k(f)\|_{\ell^p(\mathbb{Z}^N,w)}\le C \|f\|_{\ell^p(\mathbb{Z}^N,w)}, \qquad k>1,\qquad f\in \ell^2(\mathbb{Z}^N)\cap \ell^p(\mathbb{Z}^N).
\]
\end{lem}

\begin{proof}
We use an induction argument to prove the result. Let us suppose that the operator $P_{t,k}$ is bounded from $\ell^p(\mathbb{Z}^N,w)$ into $\ell^p_{\mathbb{B}_k}(\mathbb{Z}^N,w)$. Taking $k=1$ and applying Krivine's theorem (see \cite[Theorem 1.f.14]{L-Z}), we deduce that the operator $\overline{P}_{t,1}: \ell^p_{\mathbb{B}_k}(\mathbb{Z}^N,w)\longrightarrow \ell^p_{\mathbb{B}_k\times \mathbb{B}_1}(\mathbb{Z}^N,w)$, given by
\[
\{f_s(n)\}_{s\ge 0}\longmapsto \{P_{t,1}f_s\}_{t,s\ge 0},
\]
is bounded. Moreover, $\overline{P}_{t,1}\circ P_{s,k}$ is a bounded operator from $\ell^p(\mathbb{Z}^N,w)$ into $\ell^p_{\mathbb{B}_k\times \mathbb{B}_1}(\mathbb{Z}^N,w)$. With the Chapman-Kolmogorov type identity
\[
\sum_{\mb{k}\in \mathbb{Z}^N}G_{t,N}(\mb{k})G_{s,N}(\mb{n}-\mb{k})=G_{t+s,N}(\mb{n}),
\]
which can be deduced from Neumann's identity \eqref{eq:Neumann}, we obtain that
\[
\frac{\partial}{\partial t}\left(W_t\left(\frac{\partial^k}{\partial s^k}W_sf\right)\right)=\left.\frac{\partial^{k+1}}{\partial u^{k+1}}W_u f\right|_{u=s+t},
\]
and using this we have
\begin{align*}
\left\|\overline{P}_{t,1}\circ P_{s,k}f\right\|_{\mathbb{B}_k\times \mathbb{B}_1}^2
&=\int_{0}^{\infty}\int_{0}^{\infty}ts^{2k-1}\left|\left.\frac{\partial^{k+1}}{\partial u^{k+1}}W_u f\right|_{u=s+t}\right|^2\,ds\, dt\\&=
\int_{0}^{\infty}\int_{t}^{\infty}t(r-t)^{2k-1}\left|\left.\frac{\partial^{k+1}}{\partial u^{k+1}}W_u f\right|_{u=r}\right|^2\,dr\, dt
\\&=\int_{0}^{\infty}\left|\frac{\partial^{k+1}}{\partial r^{k+1}}W_r f\right|^2\int_{0}^{r}t(r-t)^{2k-1}\,dt\, dr
\\&=\frac{1}{(2k+1)(2k)}\int_{0}^{\infty}r^{2k+1}\left|\frac{\partial^{k+1}}{\partial r^{k+1}}W_r f\right|^2\, dr
\\&=\frac{g_{k+1}(f)}{(2k+1)(2k)}
\end{align*}
and the proof is finished.
\end{proof}

After the reductions in the previous lemmas, Theorem \ref{thm:equiv-gk-Heat} will be a consequence of the following result.

\begin{thm}
  Let $N\ge 1$ and $w\in A_p$. Then the inequalities
  \[
  \|g_1(f)\|_{\ell^{p}(\mathbb{Z}^n,w)}\le C \|f\|_{\ell^{p}(\mathbb{Z}^n,w)}, \qquad 1<p<\infty, \qquad \ell^2(\mathbb{Z}^N)\cap \ell^p(\mathbb{Z}^N)
  \]
  and
  \[
  \|g_1(f)\|_{\ell^{1,\infty}(\mathbb{Z}^n,w)}\le C\|f\|_{\ell^{1}(\mathbb{Z}^n,w)},\qquad \ell^2(\mathbb{Z}^N)\cap \ell^1(\mathbb{Z}^N),
  \]
  hold.
\end{thm}

To prove this theorem we check that $g_1$ is a Calder\'on-Zygmund operator. With Lemma \ref{lem:L2-gk}
 and the estimates in the following lemma this fact will be clear.

\begin{lem}
\label{lem:CZ-gk}
  Let $N\ge 1$, then inequalities
  \begin{equation}
  \label{eq:size-g}
  \|\mathcal{P}_{t,1}(\mb{n})\|_{\mathbb{B}_1}\le \frac{C}{(|\mb{n}|+1)^N}
  \end{equation}
  and
  \begin{equation}
  \label{eq:smooth-g}
  \|\mathcal{P}_{t,1}(\mb{n})-\mathcal{P}_{t,1}(\mb{m})\|_{\mathbb{B}_1}\le C\frac{|\mb{n}-\mb{m}|}{(|\mb{n}|+|\mb{m}|)^{N+1}},\qquad |\mb{n}|>2|\mb{n}-\mb{m}|
  \end{equation}
  hold.
\end{lem}

Before proving the previous lemma, we need another one.
\begin{lem}
\label{lem:diff3-G}
  Let $N\ge 1$, $\mb{n},\mb{m}\in \mathbb{Z}^N$ such that $|\mb{n}|>2|\mb{n}-\mb{m}|$, and
  \[
  \mathfrak{H}_{t,N}(z)=\left(\frac{z}{t^2}+\frac{z^3}{t^3}\right)
  \left(G_{t,1}(z)\right)^N, \qquad z>0.
  \]
  Then the inequality
  \begin{equation}
  \label{eq:diff3-G}
  |\Delta_{N,i}(G_{t,N}(\mb{n})-G_{t,N}(\mb{m}))|\le C |\mb{n}-\mb{m}| \mathfrak{H}_{t,N}\left(\frac{|\mb{n}|+|\mb{m}|}{k}\right),\qquad i=1,\dots,N,
  \end{equation}
  holds with constants $C$ and $K>1$ independent of $\mb{n}$ and $\mb{m}$.
\end{lem}
\begin{proof}
Again, it is enough to analyze the case $i=1$ because the other ones work similarly.

By using the ideas in Lemma \ref{lem:diff-G} and Lemma \ref{lem:diff2-G}, we can reduce the proof to the case where $n_i\not=m_i$ for $i=1,\dots,N$. Moreover, the assumptions \eqref{eq:mono-n} on $\mb{n}$ and $\mb{m}$ in the proof of Lemma \ref{lem:diff-G} can be considered also.

We start with the case $m_1\ge 2$. From \eqref{eq:rara} we deduce that
\begin{multline*}
\Delta_{N,1}(G_{t,N}(\mb{n})-G_{t,N}(\mb{m}))=\Delta(G_{t,1}(n_1)- G_{t,1}(m_1))\prod_{k=2}^{N}G_{t,1}(n_k)\\+
\sum_{j=2}^{N}\Delta G_{t,1}(m_1)(G_{t,1}(n_i)-G_{t,1}(m_i))\prod_{j=2}^{i-1}G_{t,1}(m_j)\prod_{k=i+1}^{N}G_{t,1}(n_k).
\end{multline*}
We can control $G_{t,1}(n_i)-G_{t,1}(m_i)$ as in the proof of Lemma \ref{lem:diff-G} and, applying \eqref{eq:diff2-I}, we have
\begin{multline}
\label{eq:Delta-m}
|\Delta G_{t,1}(m_1)(G_{t,1}(n_i)-G_{t,1}(m_i))|\\
\begin{aligned}
&\le C \frac{|n_i-m_i|(m_i+n_i+1)}{t}\left(\frac{1}{t}+\frac{m_1^2}{t^2}\right)G_{t,1}(m_1-1)\max\{G_{t,1}(n_i),G_{t,1}(m_i)\}
\\& \le C \frac{|n_i-m_i|(m_i+n_i+1)}{t}\left(\frac{1}{t}+\frac{m_1^2}{t^2}\right)G_{t,1}(m_1/2)\max\{G_{t,1}(n_i),G_{t,1}(m_i)\},
\end{aligned}
\end{multline}
where in the last step we have used that $m_1-1\ge m_1/2$ because $m_1\ge 2$.

For $n_1< m_1 $, we can obtain that
\begin{multline*}
\Delta(G_{t,1}(n_1)-G_{t,1}(m_1))=-\sum_{k=n_1}^{m_1-1}\Delta(G_{t,1}(k+1)-G_{t,1}(k)))
\\=-\sum_{k=n_1}^{m_1-1}(G_{t,1}(k+3)-3G_{t,1}(k+2)+3G_{t,1}(k+1)-G_{t,1}(k))
\end{multline*}
and, applying \eqref{eq:diff-3},
\begin{align*}
|\Delta(G_{t,1}(n_1)-G_{t,1}(m_1))|&\le C \sum_{k=n_1}^{m_1-1}\left(\frac{k+2}{t^2}+\frac{(k+1)(k+2)(k+3)}{t^3}\right)G_{t,1}(k)
\\&\le C(m_1-n_1)\left(\frac{m_1+n_1}{t^2}+\frac{m_1^3+n_1^3}{t^3}\right)G_{t,1}(n_1).
\end{align*}
More generally,
\begin{multline*}
|\Delta(G_{t,1}(n_1)-G_{t,1}(m_1))|\\\le C|m_1-n_1|\left(\frac{m_1+n_1}{t^2}+\frac{m_1^3+n_1^3}{t^3}\right)\max\{G_{t,1}(n_1),G_{t,1}(m_1)\}.
\end{multline*}
Now, the proof of this case can be concluded as in Lemma \ref{lem:diff-G}.

To finish we have to analyze the cases $m_1=0,1$. To obtain the result in these cases we need an appropriate estimate for $\Delta G_{t,1}(m_1)$ to insert it in \eqref{eq:Delta-m}. For $m_1=0$, we have
\[
\Delta G_{t,1}(0)=G_{t,1}(1)-2G_{t,1}(0)+G_{t,1}(-1)=2(G_{t,1}(1)-G_{t,1}(0))
\]
and, applying \eqref{eq:diff-I}, it is verified that
\[
|\Delta_{t,1}(0)|\le \frac{C}{t}G_{t,1}(0)
\]
and this is enough for our purpose. In the case $m_1=1$, for $0<t\le 1$ it is clear that
\[
|\Delta_{t,1}(1)|\le C\le \frac{C}{t}G_{t,1}(1)
\]
and for $t>1$, applying \eqref{eq:diff-3} and \eqref{eq:cantabros}, we have
\begin{align*}
|\Delta_{t,1}(1)|&\le C\left(\frac{1}{t}+\frac{1}{t^3}\right)G_{t,1}(0)\\&\le C \left(\frac{1}{t}+\frac{1}{t^3}\right)\frac{1}{f_1(t)}G_{t,1}(1)\le C \left(\frac{1}{t}+\frac{1}{t^3}\right)G_{t,1}(1),
\end{align*}
where in the last step we have used that $1/f_1(t)\le C$ for $t>1$.
\end{proof}

\begin{rem}
  From the proof of the previous result, it is cleat that
  \begin{equation}
  \label{eq:HHH-Infty}
    |\Delta_{N,i}(G_{t,N}(\mb{n})-G_{t,N}(\mb{m}))|\le C |\mb{n}-\mb{m}| \left(\frac{|\mb{n}|}{t^2}+\frac{|\mb{n}|^3}{t^3}\right)t^{-N/2}, \qquad \mb{n},\mb{m}\in \mathbb{Z}^N.
  \end{equation}
\end{rem}

\begin{proof}[Proof of Lemma \ref{lem:CZ-gk}]
First, it is easy to check that
\begin{equation*}
\|\mathcal{P}_{t,1}(\mb{n})\|_{\mathbb{B}_1}^2=\left\|\frac{\partial}{\partial t}G_{t,N}(\mb{n})\right\|_{\mathbb{B}_1}^2
=\left\|\Delta_N G_{t,N}(\mb{n})\right\|_{\mathbb{B}_1}^2\le \sum_{i=1}^{N}\left\|\Delta_{N,i}G_{t,N}(\mb{n})\right\|_{\mathbb{B}_1}^2.
\end{equation*}
Now, for $|\mb{n}|\ge 1$, by using \eqref{eq:diff2-I}, we have
\[
|\Delta_{N,i}G_{t,N}(\mb{n})|\le C\left(\frac{1}{t}+\frac{|\mb{n}|^2}{t^2}\right)\left(G_{t,1}\left(\frac{|\mb{n}|-1}{N}\right)\right)^N.
\]
We consider the decomposition
\[
\left\|\Delta_{N,i}G_{t,N}(\mb{n})\right\|_{\mathbb{B}_1}^2\le C(D_1+D_2),
\]
where
\[
D_1=\int_{0}^{|\mb{n}|^2}t\left(\frac{1}{t}+\frac{|\mb{n}|^2}{t^2}\right)^2
\left(G_{t,1}\left(\frac{|\mb{n}|-1}{N}\right)\right)^{2N}\, dt
\]
and $D_2$ the integral of the same function on $[|\mb{n}|^2,\infty)$. Applying \eqref{eq:bound-I} with $\alpha=3/2N$ we deduce that
\begin{align*}
D_1&\le C \int_{0}^{|\mb{n}|^2}(t+|\mb{n}|^2)^2\left(t^{-3/(2N)}G_{t,1}\left(\frac{|\mb{n}|-1}{N}\right)\right)^{2N}\, dt\\&\le \frac{C}{|\mb{n}|^{2N+6}}\int_{0}^{|\mb{n}|^2}(t+|\mb{n}|^2)^2\, dt=\frac{C}{|\mb{n}|^{2N}}.
\end{align*}
For $D_2$, using the bound $(G_{t,1}((|\mb{n}|-1)/N))^N\le C t^{-N/2}$, we obtain the estimate
\[
D_2\le C \left(\int_{|\mb{n}|^2}^{\infty}t^{-N-1}\, dt+|\mb{n}|^4\int_{|\mb{n}|^2}^{\infty}t^{-N-3}\, dt\right)=\frac{C}{|\mb{n}|^{2N}}.
\]
The case $|\mb{n}|=0$ is clear and the proof of \eqref{eq:size-g} is completed.

To obtain \eqref{eq:smooth-g}, it is enough to prove that
\begin{equation}
  \label{eq:smooth-g-aux}
  \|\Delta_{N,i}(G_{t,N}(\mb{n})-G_{t,N}(\mb{m}))\|_{\mathbb{B}_1}\le C\frac{|\mb{n}-\mb{m}|}{(|\mb{n}|+|\mb{m}|)^{N+1}},\qquad |\mb{n}|>2|\mb{n}-\mb{m}|.
  \end{equation}
To do this we consider \eqref{eq:diff3-G} and we have 
\[
 \|\Delta_{N,i}(G_{t,N}(\mb{n})-G_{t,N}(\mb{m}))\|_{\mathbb{B}_1}^2\le C |\mb{n}-\mb{m}|^2 (E_1+E_2),
\]
with
\[
E_1=\int_{0}^{(|\mb{n}|+|\mb{m}|)^2} t\left(\mathfrak{H}_{t,N}\left(\frac{|\mb{n}|+|\mb{m}|}{k}\right)\right)^2\, dt
\]
and $E_2$ equals to the integral of the same function on $[(|\mb{n}|+|\mb{m}|)^2,\infty)$. For $E_1$, we use \eqref{eq:bound-I} with $\alpha=5/(2N)$ to obtain that
\begin{align*}
E_1&\le \frac{C}{(|\mb{n}|+|\mb{m}|)^{2N+10}}\int_{0}^{(|\mb{n}|+|\mb{m}|)^2} (t(|\mb{n}|+|\mb{m}|)+(|\mb{n}|+|\mb{m}|)^3)^2\, dt\\&\le \frac{C}{(|\mb{n}|+|\mb{m}|)^{2N+2}}.
\end{align*}
Finally, to estimate $E_2$ we use \eqref{eq:HHH-Infty} to obtain the required bound. Indeed,
\begin{align*}
E_2&\le C\left((|\mb{n}|+|\mb{m}|)^2\int_{(|\mb{n}|+|\mb{m}|)^2}^{\infty} t^{-N-3}\, dt+(|\mb{n}|+|\mb{m}|)^6\int_{(|\mb{n}|+|\mb{m}|)^2}^{\infty} t^{-N-5}\, dt\right)\\&\le \frac{C}{(|\mb{n}|+|\mb{m}|)^{2N+2}}
\end{align*}
and the proof of \eqref{eq:smooth-g-aux} is finished.
\end{proof}

\subsection{The $\mathfrak{g}_k$-square functions associated to the Poisson semigroup}
It is very common to define some square functions in terms of the Poisson semigroup
instead of the heat semigroup. In our case such $\mathfrak{g}_k$-square functions are given by
\[
\mathfrak{g}_k(f)(\mb{n})=\left(\int_{0}^{\infty}t^{2k-1}\left|\frac{\partial^k}{\partial t^k}P_tf(\mb{n})\right|^2\, dt\right)^{1/2}, \qquad k=1,2,\dots.
\]
By using that
\[
P_tf(\mb{n})=\sum_{\mb{k}\in \mathbb{Z}^N}f(\mb{k})Q_{t}(\mb{n}-\mb{k})
\]
with
\[
Q_t(\mb{n})=\int_{[-1/2,1/2]^N}e^{-2t\sqrt{\sum_{i=1}^{N}\sin^2(\pi x_i)}}e^{-2\pi i\langle x,\mb{n}}\rangle \, dx,
\]
it is clear that
\begin{equation}
\label{eq:gk-Pois-Banach}
\mathfrak{g}_k(f)(\mb{n})=\|R_{t,k}f(\mb{n})\|_{\mathbb{B}_k}
\end{equation}
where
\[
R_{t,k}f(\mb{n})=\sum_{\mb{k}\in \mathbb{Z}^N}f(\mb{k})Q_{t,k}(\mb{n}-\mb{k})
\]
being
\[
Q_{t,k}(\mb{n})=\frac{d^k}{dt^k}Q_{t}(\mb{n})=\int_{[-1/2,1/2]^N}S_{t,k}(x)e^{-2\pi i\langle x,\mb{n}}\rangle \, dx
\]
and
\[
S_{t,k}(x)=(-2)^k \left(\sum_{i=1}^{N}\sin^2(\pi x_i)\right)^{k/2}e^{-2t\sqrt{\sum_{i=1}^{N}\sin^2(\pi x_i)}}e^{-2\pi i\langle x,\mb{n}}\rangle \, dx.
\]

For this family of square functions we have the following result.
\begin{thm}
\label{thm:equiv-gk-Poisson}
  Let $N\ge 1$, $1<p<\infty$, $w\in A_p(\mathbb{Z}^N)$, and $k\in \mathbb{N}$. Then the inequalities
  \begin{equation*}
  C_1\|f\|_{\ell^p(\mathbb{Z}^N,w)}\le \|\mathfrak{g}_k(f)\|_{\ell^p(\mathbb{Z}^N,w)}\le C_2 \|f\|_{\ell^p(\mathbb{Z}^N,w)}, \qquad f\in \ell^2(\mathbb{Z}^N)\cap \ell^{p}(\mathbb{Z}^N,w),
  \end{equation*}
  hold.
\end{thm}

To prove the previous result we need the two next lemmas.
\begin{lem}
\label{lem:L2-gk-P}
  For $N\ge 1$ and $k\in \mathbb{N}$, it is verified that
  \begin{equation}
  \label{eq:L2-gk-P}
  \|\mathfrak{g}_k(f)\|_{\ell^2(\mathbb{Z}^N)}^2=\frac{\Gamma(2k)}{2^{2k}}\|f\|_{\ell^2(\mathbb{Z}^N)}^2.
  \end{equation}
\end{lem}

The proof of this lemma can be done exactly as the proof of Lemma \ref{lem:L2-gk} but using \eqref{eq:gk-Pois-Banach} so we omit the details.

\begin{lem}
Let $N\ge 1$, $1<p<\infty$, $w\in A_p(\mathbb{Z}^N)$, and $k\in \mathbb{N}$. Then the inequality
\[
\|\mathfrak{g}_k(f)\|_{\ell^p(\mathbb{Z}^N,w)}\le C \|f\|_{\ell^p(\mathbb{Z}^N,w)}, \qquad f\in\ell^2(\mathbb{Z}^N)\cap \ell^p(\mathbb{Z}^N),
\]
implies
\[
\|f\|_{\ell^p(\mathbb{Z}^N,w)}\le C \|\mathfrak{g}_k(f)\|_{\ell^p(\mathbb{Z}^N,w)}, \qquad f\in \ell^2(\mathbb{Z}^N)\cap \ell^p(\mathbb{Z}^N).
\]
\end{lem}

In this case, the proof of this lemma can be obtained polarising the identity \eqref{eq:L2-gk-P} and using duality. Again we omit the details.

Finally, to conclude the proof of Theorem \ref{thm:equiv-gk-Poisson} we will use the estimate
\[
\|g_k(f)\|_{\ell^p(\mathbb{Z}^N,w)}\le C \|f\|_{\ell^p(\mathbb{Z}^N,w)}, \qquad f\in \ell^2(\mathbb{Z}^N)\cap \ell^{p}(\mathbb{Z}^N,w)
\]
and the next result.

\begin{lem}
\label{lem:gkPoi-gkHeat}
Let $k\in \mathbb{N}$, then
\[
\mathfrak{g}_{k}(f)(\mb{n})\le \sum_{j=0}^{\lfloor k/2 \rfloor} A_{k,j} g_{k-j}(f)(\mb{n}),
\]
where $A_{k,j}$ are some positive constants and $\lfloor \cdot \rfloor$ denotes the floor function.
\end{lem}
\begin{proof}
First, we observe that
\[
\frac{\partial^k}{\partial t^k}h\left(\frac{t^2}{4u}\right)=\sum_{j=0}^{[k/2]}B_{k,j}\left.\frac{\partial^{k-j} }{\partial s^{k-j}}h(s)\right|_{s=\frac{t^2}{4u}}\frac{t^{k-2j}}{(4u)^{k-j}},
\]
for some constants $B_{k,j}$. Then, from \eqref{eq:def-Poisson}, we have
\[
\frac{\partial^k}{\partial t^k}P_tf(\mb{n})=\frac{1}{\sqrt{\pi}}\sum_{j=0}^{[k/2]}B_j\int_{0}^{\infty}\frac{e^{-u}}{\sqrt{u}}
\left(\left.\frac{\partial^{k-j} }{\partial s^{k-j}}W_{s}f(\mb{n})\right|_{s=\frac{t^2}{4u}}\right)\frac{t^{k-2j}}{(4u)^{k-j}}\, du
\]
and, by Minkowsky's integral inequality,
\[
\mathfrak{g}_k(f)(\mb{n})\le \sum_{j=0}^{[k/2]}B_{k,j} S_j(\mb{n})
\]
where
\begin{multline*}
S_j(\mb{n})\\=\frac{1}{\sqrt{\pi}}\int_{0}^{\infty}\frac{e^{-u}}{\sqrt{u}(4u)^{k-j}}\left(\int_{0}^{\infty}t^{4k-4j-1}\left(\left.\frac{\partial^{k-j} }{\partial s^{k-j}}W_{s}f(\mb{n})\right|_{s=\frac{t^2}{4u}}\right)^2\, dt\right)^{1/2}\, du.
\end{multline*}
Now, by using an appropriate change of variables, we have
\begin{align*}
S_j(\mb{n})&=\frac{1}{\sqrt{2\pi}}\int_{0}^{\infty}\frac{e^{-u}}{\sqrt{u}}\left(\int_{0}^{\infty}s^{2k-2j-1}\left(\frac{\partial^{k-j} }{\partial s^{k-j}}W_{s}f(\mb{n})\right)^2\, ds\right)^{1/2}\, du\\&=\frac{1}{\sqrt{2}}g_{k-j}(f)(\mb{n})
\end{align*}
and the result follows.
\end{proof}

\subsection{The Laplace type multipliers}
Given a bounded function $M$ defined on $[0,4N]$, the multiplier associated with $M$ is the operator, initially defined on $\ell^{2}(\mathbb{N})$, by the identity
\[
T_Mf(\mb{n})=\int_{[-1/2,1/2]^N}M\left(4\sum_{i=1}^{N}\sin^2(\pi x_i)\right)\mathcal{F}f(x)e^{-2\pi i\langle x,\mb{n}\rangle}\, dx.
\]
We say that $T_M$ is a Laplace type multiplier when
\[
M(x)=x \int_{0}^{\infty}e^{-xt}a(t)\, dt,
\]
with $a$ being a bounded function. 

The Laplace type multipliers were introduced by Stein in \cite[Ch. 2]{Stein}. There, it is observed that they verify $|x^k M^{(k)}(x)|\le C_k$ for $k=0,1,\dots$, and then form a subclass of Marcinkiewicz multipliers. For the operators $T_M$ we have the following result.
\begin{thm}
\label{th:laplace-multi}
  Let $1<p<\infty$ and $w\in A_p(\mathbb{Z}^N)$. Then,
\[
\|T_Mf\|_{\ell^p(\mathbb{Z}^N,w)}\le C\|f\|_{\ell^p(\mathbb{Z}^N,w)},\qquad  f\in \ell^2(\mathbb{Z}^N)\cap\ell^{p}(\mathbb{Z}^N,w),
\end{equation*}
where $C$ is a constant independent of $f$.
\end{thm}

From the identity
\[
x^{i\gamma}=\frac{x}{\Gamma(1-i\gamma)}\int_{0}^{\infty}e^{-xt}t^{-i\gamma}\, dt,\qquad \gamma \in \mathbb{R},
\]
we deduce the following corollary.
\begin{cor}
  Let $1<p<\infty$, $\gamma\in \mathbb{R}$, and $w\in A_p(\mathbb{Z}^N)$. Then,
\[
\|(-\Delta_N)^{i\gamma}f\|_{\ell^p(\mathbb{Z}^N,w)}\le C\|f\|_{\ell^p(\mathbb{Z}^N,w)}, \qquad  f\in \ell^2(\mathbb{Z}^N)\cap\ell^{p}(\mathbb{Z}^N,w),
\end{equation*}
where $C$ is a constant independent of $f$.
\end{cor}

\begin{proof}[Proof of Theorem \ref{th:laplace-multi}]
We only need prove that
\begin{equation}
\label{eq:laplace}
g_1(T_Mf)(\mb{n})\le C g_2(f)(\mb{n}),
\end{equation}
since by Theorem \ref{thm:equiv-gk-Heat} we get that
\[
\|T_Mf\|_{\ell^p(\mathbb{Z}^N,w)}\le C \|g_1(T_Mf)\|_{\ell^p(\mathbb{Z}^N,w)}\le C \|g_2(f)\|_{\ell^p(\mathbb{Z}^N,w)}\le C \|f\|_{\ell^p(\mathbb{Z}^N,w)}.
\]

Moreover, it is enough to prove \eqref{eq:laplace} for sequences in $c_{00}$, the space of sequences having a finite number of non-null terms. First, we have
\[
T_Mf(\mb{n})=-\int_{0}^{\infty}a(s)\frac{\partial}{\partial s}W_sf(\mb{n})\, ds,
\]
which is an elementary consequence of the relation
\begin{multline*}
\int_{[-1/2,1/2]^N}
M\left(4\sum_{i=1}^{N}\sin^2(\pi x_i)\right)e^{-2\pi i \langle x, \mb{n}\rangle}\, dx\\
\begin{aligned}
&= \int_{0}^{\infty}a(s)\int_{[-1/2,1/2]^N}4\sum_{i=1}^{N}\sin^2(\pi x_i)e^{-4s\sum_{i=1}^{N}\sin^2(\pi x_i)}e^{-2\pi i \langle x, \mb{n}\rangle}\, dx\, ds
\\&=-\int_{0}^{\infty}a(s)\frac{\partial}{\partial s}\int_{[-1/2,1/2]^N}4\sum_{i=1}^{N}\sin^2(\pi x_i)e^{-4s\sum_{i=1}^{N}\sin^2(\pi x_i)}e^{-2\pi i \langle x, \mb{n}\rangle}\, dx\, ds
\\& =-\int_{0}^{\infty}a(s)\frac{\partial}{\partial s}G_{s,N}(\mb{n})\, ds.
\end{aligned}
\end{multline*}
Then, applying the semigroup property of $W_t$ we obtain
\begin{equation*}
W_t(T_Mf)(\mb{n})=
-\int_{0}^{\infty}a(s)\frac{\partial}{\partial s}W_{s+t}f(\mb{n})\, ds
\end{equation*}
and hence,
\begin{equation*}
\frac{\partial}{\partial t}W_t(T_Mf)(\mb{n})=-\int_{0}^{\infty}a(s)\frac{\partial}{\partial t}\frac{\partial}{\partial s}W_{s+t}f(\mb{n})\, ds=-\int_{0}^{\infty}a(s)\frac{\partial^2}{\partial s^2}W_{s+t}f(\mb{n})\, ds.
\end{equation*}
In this way,
\begin{align*}
\left|\frac{\partial}{\partial t}W_t(T_Mf)(\mb{n})\right|&\le C \int_{t}^{\infty}s\left|\frac{\partial^2}{\partial s^2}W_{s}f(\mb{n})\right|\,\frac{ds}{s}\\&\le C t^{-1/2}\left(\int_{t}^{\infty}s^2\left|\frac{\partial^2}{\partial s^2}W_{s}f(\mb{n})\right|^2\,ds\right)^{1/2}.
\end{align*}
Finally,
\begin{align*}
(g_1(T_Mf)(\mb{n}))^2&=\int_{0}^{\infty} t\left|\frac{\partial}{\partial t}W_t(T_Mf)(\mb{n})\right|^2\, dt
 \le C
\int_{0}^{\infty}\int_{t}^{\infty}s^2\left|\frac{\partial^2}{\partial s^2}W_{s}f(\mb{n})\right|^2\,ds\, dt\\
&=C
\int_{0}^{\infty}s^3\left|\frac{\partial^2}{\partial s^2}W_{s}f(\mb{n})\right|^2\,ds=C (g_2(f)(\mb{n}))^2
\end{align*}
and the proof of \eqref{eq:laplace} is completed.
\end{proof}





\begin{thebibliography}{50}
\bibitem{Ab-et-al} L. Abadias, J. González-Camus, P. J. Miana, and J. C. Pozo,  Large time behaviour for the heat equation on $\mathbb{Z}$, moments and decay rates, \textit{J. Math. Anal. Appl.} \textbf{500} (2021), Paper No. 125137, 25 pp.

\bibitem{Ar-Os} G. I. Arkhipov and K. I. Oskolkov, A special trigonometric series and its applications, \textit{Mat. Sb.
(N.S.)} \textbf{134(176)} (1987), 147--157, 287; translation in \textit{Math. USSR-Sb.} \textbf{62} (1989), 145--155.

\bibitem{Bour} J. Bourgain, Pointwise ergodic theorems for arithmetic sets, \textit{Inst. Hautes Etudes Sci. Publ. Math.}
\textbf{69} (1989), 5--45.

\bibitem{Calderon-Zygmund} A. P. Calder\'on and A. Zygmund, On the existence of certain singular integrals, \textit{Acta Math.} \textbf{88} (1952), 85--139.

\bibitem{Caze} T. Cazenave and A. Haraux, \textit{An introduction to semilinear evolution equations}, Oxford Lecture Series in Mathematics and its Applications 13, The Clarendon Press, Oxford University Press, New York, 1998.

\bibitem{CGRTV} \'O. Ciaurri, T. A. Gillespie, L. Roncal, J. L. Torrea, and J. L. Varona, Harmonic analysis associated with a discrete Laplacian, \textit{J. Anal. Math.} \textbf{132} (2017), 109--131.

\bibitem{CRSTV} \'O. Ciaurri, L. Roncal, P. R. Stinga, J. L. Torrea, and J. L. Varona, Nonlocal discrete diffusion equations and the fractional discrete Laplacian, regularity and applications, \textit{Adv. Math.} \textbf{330} (2018), 688--738.

\bibitem{Duo-Zua} J. Duoandikoetxea and E. Zuazua, Moments, masses de Dirac et d\'ecomposition de fonctions (Moments, Dirac deltas and expansion of functions), \textit{C. R. Acad. Sci. Paris S\'er. I Math.} \textbf{315} (1992), 693--698.

\bibitem{Chinos-Convo} W. Guo, D. Fan, Dashan, H. Wu, and G. Zhao, Sharp weighted convolution inequalities and some applications, \textit{Studia Math.} \textbf{241} (2018), 201--239.

\bibitem{Feller} W. Feller, \textit{An Introduction to Probability Theory and its Applications, Vol. 2}, second ed., Wiley, New York, 1971.

\bibitem{GI} F. A. Gr\"unbaum and P. Iliev, Heat kernel expansions on the integers, \textit{Math. Phys. Anal. Geom.} \textbf{5} (2002), 183--200.

\bibitem{HMW} R. Hunt, B. Muckenhoupt, and R. Wheeden, Weighted norm inequalities for the conjugate function and Hilbert transform, \textit{Trans. Amer. Math. Soc.} \textbf{176} (1973), 227--251.

\bibitem{Ignat} L. I. Ignat, Qualitative properties of a numerical scheme for the heat equation, Numerical mathematics and advanced applications, 593--600, Springer, Berlin, 2006.

\bibitem{Kinchin} A. Ya. Khinchin, \textit{Continued fractions}, University of Chicago Press, 1964.

\bibitem{Lebedev} N. N. Lebedev, \textit{Special functions and its applications}, Dover, New York, 1972.

\bibitem{L-Z} J. Lindenstrauss and L. Tzafriri, \textit{Classical Banach spaces II}, Springer-Verlag, Berlin, 1979.

\bibitem{NIST} F. W. J. Olver (editor-in-chief), \textit{NIST Handbook of Mathematical Functions},
Cambridge University Press, New York, 2010.

\bibitem{Pierce} L. B. Pierce, \textit{Discrete Analogues in Harmonic Analysis}, Ph.D. thesis, Princeton University,
Princeton, 2009.

\bibitem{Riesz} M. Riesz, Sur les fonctions conjugu\'ees, \textit{Math. Z.} \textbf{27} (1928), 218--244.

\bibitem{RRT} J. L. Rubio de Francia, F. J. Ruiz, and J. L. Torrea, Calder\'on-Zygmund theory for operator-valued kernels, \textit{Adv. in Math.} \textbf{62} (1986), 7--48.

\bibitem{RT} F. J. Ruiz and J. L. Torrea, Vector-valued Calderón-Zygmund theory and Carleson measures on spaces of homogeneous nature, \textit{Studia Math.} \textbf{88} (1988), 221--243.

\bibitem{Cantabros} D. Ruiz-Antol\'{\i}n and J. Segura, A new type of sharp bounds for ratios of modified
Bessel functions, \textit{J. Math. Anal. Appl.} \textbf{443} (2016) 1232--1246.

\bibitem{Stein} E. M. Stein, \textit{Topics in Harmonic Analysis Related to the Littlewood-Paley Theory}, Princeton Univ.
Press, Princeton, NJ, 1970.

\bibitem{Ste-Wa-1} E. M. Stein and S. Wainger, Discrete analogues in harmonic analysis, I: $\ell^2$ estimates for singular Radon transforms, \textit{Amer. J. Math.} \textbf{121} (1999), 1291--1336.

\bibitem{Ste-Wa-2} E. M. Stein and S. Wainger, Discrete analogues in harmonic analysis II: fractional integration,
\textit{J. Anal. Math.} \textbf{80} (2000), 335--355.

\bibitem{Indios} V. R. Thiruvenkatachar and T. S. Nanjundiah, Inequalities concerning Bessel functions and orthogonal polynomials, \textit{Proc. Indian Acad. Sci. A} \textbf{33} (1951), 373--384.

\bibitem{Chinos}  Z. H. Yang and S. Z. Zheng, Monotonicity and convexity of the ratios of the first kind modified Bessel functions and applications, \textit{Math. Inequal. Appl.} \textbf{21} (2018), 107--125.
\end{thebibliography}
\end{document}